\normalfont
\documentclass[12pt,a4paper]{amsart}
\usepackage{amsfonts,layout,color,amssymb}
\usepackage{bm,epsfig,eucal}

\catcode`@=11
\@addtoreset{equation}{section}

\catcode`@=12

\newtheorem{Theorem}{Theorem}[section]
\newtheorem{Definition}[Theorem]{Definition}
\newtheorem{Proposition}[Theorem]{Proposition}
\newtheorem{Lemma}[Theorem]{Lemma}
\newtheorem{Corollary}[Theorem]{Corollary}

\newtheorem{Remark}[Theorem]{Remark}

\newtheorem{Hyp.}[Theorem]{Hyp.}

\newcommand{\TT}{\mathcal T}
\newcommand{\RR}{\mathbb{R}}
\newcommand{\norm}[1]{\left\lVert #1 \right\rVert}

\newcommand{\bT}{{\pmb{T}}}
\newcommand{\bTi}{{\bT^0}}
\newcommand{\tildebT}{{\pmb{\tilde T}}}
\newcommand{\tildebTi}{{\pmb{\tilde T}^0}}

\newcommand{\bV}{{\pmb{V}}}
\newcommand{\bH}{{\pmb{H}}}
\newcommand{\bA}{{\pmb{\mathcal A}}}
\newcommand{\bDA}{{\pmb{D(\mathcal A)}}}
\newcommand{\bG}{{\pmb{G}}}
\newcommand{\bu}{{\pmb{u}}}

\newcommand{\bM}{{\pmb{M}}}
\newcommand{\bY}{{\pmb{Y}}}

\newcommand{\bB}{{\pmb{B}}}

\newcommand{\eps}{\varepsilon}

\begin{document}

\title[]{Comparison principles and long time behavior for a diffusive Energy Balance Model with vertical resolution}

\author{P. Cannarsa} 
\address{Dipartimento di Matematica, Universit\`a di Roma "Tor Vergata",
Via della Ricerca Scientifica, 00133 Roma, Italy}
\email{cannarsa@axp.mat.uniroma2.it}

\author{V. Lucarini} 
\address{School of Computing and Mathematical Sciences,
University of Leicester, Leicester,  LE17RH  United Kingdom
and School of Sciences, Great Bay University, Dongguan, P.R. China}
\email{v.lucarini@leicester.ac.uk}

\author{P. Martinez} 
\address{Institut de Math\'ematiques de Toulouse; UMR 5219, Universit\'e de Toulouse; CNRS \\ 
UPS IMT F-31062 Toulouse Cedex 9, France} \email{patrick.martinez@math.univ-toulouse.fr}

\author{C. Urbani} 
\address{Dipartimento di Ingegneria e Scienze, Universitas Mercatorum, Piazza Mattei 10, 00186, Roma, Italy}
\email{cristina.urbani@unimercatorum.it}

\author{J. Vancostenoble} 
\address{Institut de Math\'ematiques de Toulouse; UMR 5219, Universit\'e de Toulouse; CNRS \\ 
UPS IMT F-31062 Toulouse Cedex 9, France} \email{judith.vancostenoble@math.univ-toulouse.fr}

\keywords{Energy balance climate model, degenerate parabolic equation, nonlinear system, cooperative system, global existence, blow-up, invariant rectangle, maximum principles, attractor}


\begin{abstract}

We study a two-layer one-dimensional energy balance model, which allows for vertical energy exchanges between a surface layer and the atmosphere, as well as meridional energy transport  across latitudes via a diffusion law. The evolution equations of the surface temperature and the atmospheric temperature are coupled by exchange of infrared radiation as well as other non-radiative energy exchanges. The energy enters the system as solar radiation, which is partially absorbed and partially reflected by the two layers. The system is then composed of two degenerate parabolic equations coupled by nonlinear terms, the growth of these terms being crucial for the choice of the functional setting.

An essential parameter is the absorptivity of the atmosphere, denoted $\varepsilon _a$, whose value depends critically on greenhouse gases. We prove that blow up in finite time occurs if $\varepsilon _a >2$, while global existence of solutions and the existence of a global attractor hold when $\varepsilon _a \in (0,2)$. 

Proofs are based on comparison principles that derive from the cooperative structure of the problem, and that provide invariant rectangles for smooth initial conditions, and on regularity properties.
\end{abstract}
\maketitle


\section{Introduction}\label{sec-intro} \hfill

The climate is a multiphase and heterogeneous nonequilibrium system whose dynamical processes take place on a vast range of temporal and spatial scales. The nonequilibrium conditions are essentially due to unequal absorption of the incoming solar radiation - which provides the energy input - throughout the climate system: absorption takes preferentially place near the surface rather than aloft or in the deep ocean, and near the equator rather than at high latitudes. Radiation is then emitted mostly in the infrared range towards space. Quasi-steady state conditions are achieved as a result of the compensating effect of energy fluxes due to the motions of the geophysical fluids and the associated vertical and horizontal exchanges of energy, primarily in the form of sensible heat and latent heat \cite{Lucarini2014, Peixoto1992}. The investigation of the statistical and dynamical properties of the Earth's climate takes advantage of the use of numerical models of different levels of complexity, ranging from simple low-dimensional toy models up to Earth system models (ESMs), which provide formidable challenges in terms of high performance computing \cite{GL2020}.
 
Given the complexity of full-scale climate models, simplified frameworks such as \emph{Energy Balance Models} (EBMs) remain valuable tools for capturing essential climate mechanisms and gaining theoretical insight. Introduced in the 1960s independently by Budyko~\cite{Budyko} and Sellers~\cite{sellers}, EBMs describe the evolution in the energy content of the climate system due the competition between incoming solar radiation and outgoing terrestrial radiation and internal exchange processes \cite{North81}. 

Budyko and Sellers were able to show that the Earth's climate was expected to be multistable within a realistic range of physical parameters, with the current climate being in competition with a much colder, so-called snowball climate, featuring global glaciation. This surprising theoretical result was famously confirmed by paleoclimatological investigations, which discovered that about 650 Mya the Earth did enter and exit such a snowball state \cite{Hoffman2000,Pierrehumber2011}; see the discussion in \cite{Lucarini-Bodai2017,Lucarini-Bodai2020}. Thanks to their flexibility and physical meaningfulness, EBMs are currently being developed also for studying general exoplanetary atmospheres \cite{HaqqMisra2022,Ramirez2024}.

EBMs differ in terms of spatial detail in the description of the climate system. In the simplest case, the dynamics is written in terms of a single scalar variable, usually identified with some measure of the temperature of the climate system. Such a formulation is already able to capture the climate's multistability. A more advanced formulation of the climate energy budget allows for the use of two temperature variables, describing the surface properties, and one describing the free atmosphere properties. Going from one to two temperature variables allows, e.g., to provide a basic yet insightful description of the greenhouse effect, whereby the optical properties of the atmosphere in the infrared lead to a warming of the planetary surface \cite{Hartmann}. 

Following Sellers' viewpoint, the consideration of spatial variability leads to cast EBM as a partial differential equation describing the evolution of the surface temperature $T = T(t,x)$ according to the equation
\begin{equation}
\gamma \frac{\partial T}{\partial t} -\kappa \frac{\partial}{\partial x} \left( (1 - x^2) \frac{\partial T}{\partial x} \right) = R_s(t,x,T) - R_e(T),
\end{equation}
where $x \in (-1,1)$ represents the sine of the colatitude, $\gamma$ is the effective heat capacity, $\kappa$ is the effective heat diffusivity, where meridional heat fluxes are due the motion of the geophysical fluids, $R_s$ is the absorbed solar radiation, and $R_e$ is the emitted radiation. The latter is typically modeled by the Stefan--Boltzmann law, $R_e(T) = \sigma_B T^4$. Instead, $R_s$ has a more complex functional representation, as it is in general time-dependent, as a result of the intrinsic periodicities in the incoming solar radiation, and, more critically, features a strong dependence on $T$ because lower temperatures support the presence of surface ice, which is very effective in reflecting solar radiation. This property leads to the presence of the ice-albedo feedback, which is key to the existence of multistability \cite{Ghil76,Pierrehumber2011}.

To more accurately describe radiative and thermodynamic processes in the vertical direction, two-layer EBMs (2LEBMs) have been introduced \cite{Paltridge1974, Gill1982}, sometimes coupled with an hydrological cycle \cite{Jentsch0, Jentsch1, Jentsch2}. These models represent the Earth's surface and atmosphere as interacting layers, allowing for a more refined treatment of the \emph{greenhouse effect}, vertical energy transfer, and feedback mechanisms. The meridional heat transport of the surface layer succinctly describes the role of the ocean in reducing the meridional temperature gradient of the climate system \cite{Knietzsch2015}. Note that while the atmosphere and the ocean are extremely different in terms of dynamics, characteristic timescales, and physico-chemical features, they provide comparable contributions to the overall meridional heat transport \cite{Peixoto1992,Lucarini2014}.

Given the physical relevance and wide use of EBMs for interpreting Earth's climate and its response to perturbations \cite{Lucarini2024}, advancing understanding of their mathematical properties is key to developing a rigorous picture of the mathematics of climate change and climate variability. To this purpose, we mention \cite{Hetzer-Schmidt1990, Hetzer-Schmidt1991}, among other results, that gave a mathematical background to \cite{Jentsch1, Jentsch2}. In a recent study~\cite{CLMUV-EDO}, we examined the following ODE system for the global average temperatures $T_s(t)$ (surface) and $T_a(t)$ (atmosphere):
\begin{equation}\label{eq:ode-model}
\left\{
\begin{aligned}
\gamma_a \frac{dT_a}{dt} &= \phantom{-}\lambda(T_s - T_a) + \varepsilon_a \sigma_B T_s^4 - 2\varepsilon_a \sigma_B T_a^4 + R_a(T_a),
\\
\gamma_s \frac{dT_s}{dt} &= -\lambda(T_s - T_a) - \sigma_B T_s^4 + \varepsilon_a \sigma_B T_a^4 + R_s(T_s)
\end{aligned}
\right.
\end{equation}
where $\lambda$ represents non-radiative vertical transfer, and $\varepsilon_a$ is the atmospheric absorptivity parameter. It was shown that the system admits global solutions when $\varepsilon_a \in (0,2)$, while finite-time blow-up may occur for $\varepsilon_a > 2$. We also proved that $\varepsilon_a$, as well-known from physical arguments, is a key modulator of the greenhouse effect, with larger values of $\varepsilon_a$ associated with higher surface temperature. We have also been able to prove that the presence of the ice-albedo feedback leads in some cases to the existence of more than one stable equilibrium solution. These results indicate the presence of a sharp transition between stable and unstable regimes.

\subsection{Main achievements of this paper}
We bring forward our program of creating a rigorous mathematical framework for climate science and investigate the following spatially extended version of the above model, leading to the following degenerate parabolic system for $T_s = T_s(t,x)$ and $T_a = T_a(t,x)$:
\begin{equation}\label{eq:pde-model}
\left\{
\begin{aligned}
\gamma_a \left[ \frac{\partial T_a}{\partial t} - \kappa_a \frac{\partial}{\partial x} \left( (1 - x^2) \frac{\partial T_a}{\partial x} \right)\right] &= \phantom{-}\lambda(T_s - T_a) + \varepsilon_a \sigma_B T_s^4 - 2\varepsilon_a \sigma_B T_a^4 + R_a(x,T_a),
\\
\gamma_s \left[\frac{\partial T_s}{\partial t} - \kappa_s \frac{\partial}{\partial x} \left( (1 - x^2) \frac{\partial T_s}{\partial x} \right)\right] &= -\lambda(T_s - T_a) - \sigma_B T_s^4 + \varepsilon_a \sigma_B T_a^4 + R_s(x,T_s)
\end{aligned}
\right.
\end{equation}
complemented with suitable Neumann boundary conditions and positive initial data. The degenerate diffusion terms reflect reduced energy transport efficiency near the poles. The nonlinearity and coupling introduce mathematical challenges that require refined analysis. The presence of spatial structure and of two layers, representative of the atmosphere and ocean components, makes the model fairly relevant at physical level.

Our main contributions are as follows:

\begin{itemize}
\item When $\varepsilon_a \in (0,2)$
  
	\begin{itemize}
	\item we prove \emph{existence, uniqueness}, and \emph{uniform boundedness} of classical solutions to the PDE system \eqref{eq:pde-model};	  
	  \item we demonstrate the existence of a \emph{global attractor}, and the existence of \emph{stationary solutions}.
	\end{itemize}
\item When $\varepsilon_a > 2$: \emph{finite-time blow-up} can occur, confirming the threshold behavior identified in the ODE setting.
\end{itemize}

This analysis provides a rigorous mathematical framework for understanding how increased greenhouse gas concentrations may influence the spatially resolved climate dynamics,  and how vertical structure and latitudinal transport affect stability.

The techniques we use blend methods from the theory of degenerate parabolic equations with qualitative analysis of cooperative dynamical systems. In particular, we:
\begin{itemize}
  \item Adapt classical comparison theorems to the degenerate and strongly coupled PDE setting.
  \item Use a priori estimates and maximum principles to control the nonlinearities and ensure positivity of solutions.
  \item Exploit the cooperative structure of the equations, which is associated with the presence of vertical energy exchanges between the two layers, to analyze long-time behavior.
  \item Establish blow-up by comparison with suitable subsolutions derived from the ODE model.
  \item Combine maximal regularity properties with comparison methods, energy estimates and compactness results to prove the existence of an attractor.
\end{itemize}
Our results are closely related to those of \cite{Hetzer-Schmidt1990, Hetzer-Schmidt1991}, which were established in a more general abstract framework. In the present work, however, we extend the well-posedness theory to the space ${\bf V}$ (see definition \eqref{eq:def V}, \eqref{eq:def H V}), which is natural for the functional setting of the model and allows for possibly unbounded initial data -- a feature not covered in \cite{Hetzer-Schmidt1990, Hetzer-Schmidt1991}. Moreover, our global attractor result is proved directly in the same space ${\bf V}$.

Another significant difference concerns the absorptivity parameter $\varepsilon_a$. While \cite{Hetzer-Schmidt1990, Hetzer-Schmidt1991} impose from the outset the restrictive assumption $\varepsilon_a<1$, our analysis identifies the threshold value $\varepsilon_a=2$ as the critical bound ensuring the validity of the theory (see Remark \ref{rem:eps_a} for more details).

Similar remarks apply to the recent work \cite{Fornasaro}, which considers a related model including ocean--atmosphere interactions, but restricted to mid-latitudes. There, the treatment of the nonlinear terms is more direct and again requires the assumption $\varepsilon_a<1$, whereas our approach allows for a sharper characterization of the admissible range of $\varepsilon_a$, in addition to removing all constraints on the latitude.

The detailed analysis of the qualitative properties of equilibrium solutions is left for future investigation. 

\subsection{Plan of the paper} The paper is organized as follows.
\begin{itemize}
\item In Section~\ref{sec-pbm-hyp}, we introduce the PDE model, discuss its derivation and state the main assumptions. We also present the well-posedness results: 
\begin{itemize}
    \item existence of a maximal solution (Proposition~\ref{thm-max-mild});
    \item global existence (Theorem~\ref{prop-bounds});
    \item finite-time blow-up (Theorem~\ref{prop-blowup}).
\end{itemize}
The proofs of the latter two results are postponed to Section~\ref{sec-proof-7}, as they rely crucially on the comparison principles established in Section~\ref{sec:MP}.

\item Section~\ref{sec:MP} is devoted to maximum principles and comparison results for degenerate parabolic equations and systems. In particular, we prove:
	\begin{itemize}
		\item a weak maximum principle for linear cooperative degenerate parabolic systems (Proposition~\ref{lem-ppmaxw-syst});
		\item a comparison principle for a single degenerate parabolic equation (Corollary~\ref{cor-ppmax-eq});
		\item a strong maximum principle for linear degenerate parabolic systems (Proposition~\ref{prop-ppmax-system-deg});
		\item a comparison principle for the two-layer Energy Balance Model (Theorem~\ref{lem-comp-2layer});
		\item a comparison principle between solutions of the two-layer EBM and suitable super- and sub-solutions (Theorem~\ref{lem-comp-2layer-sub-sup}).
	\end{itemize}

\item Section~\ref{sec-proof-7} exploits the maximum and comparison principles established in Section~\ref{sec:MP}. We first recall key results from \cite{CLMUV-EDO} concerning the spatially homogeneous two-layer EBM. These results are then used to construct explicit super- and sub-solutions for the PDE system. As a consequence, we prove:
	\begin{itemize}
		\item positivity of solutions of the two-layer EBM (Proposition~\ref{prop-nonneg});
        \item the existence of invariant rectangles for regular initial conditions (Proposition~\ref{prop-rect-inv});
		\item Theorem~\ref{prop-bounds} (global existence) and Theorem~\ref{prop-blowup} (finite-time blow-up).
	\end{itemize}

\item Section \ref{sec-attractor-statement} is devoted to the analysis of the long time behavior of solutions of the autonomous EBM. We prove:
    \begin{itemize}
        \item the existence of distinguished (warmest or coldest) equilibrium points (Theorem~\ref{thm-warmest})
        \item the existence of a global attractor (Theorem~\ref{thm-attractor}).
    \end{itemize}

\item Finally, conclusions and perspectives are presented in Section~\ref{sec:conclusions}.
\end{itemize}


\section{Setting of the problem and well-posedness of solutions}\label{sec-pbm-hyp}

\subsection{Setting of the problem\vspace{.2cm}}\hfill

We consider the following two-layer model in the domain $(0,\infty)\times I$, with $I:=(-1,1)$,
\begin{equation}
\label{2layer-pbm}
\begin{cases}
\displaystyle{\gamma _a } \left[ \frac{\partial T_a}{\partial t} -  \kappa_a \frac{\partial }{\partial x} \left( (1-x^2) \frac{\partial T_a}{\partial x}\right) \right] 
\vspace{.15cm}\\
\hspace{1.5cm} 
= -\lambda (T_a - T_s)  
+ \varepsilon _a \sigma _B \vert T_s \vert ^3 T_s  - 2 \varepsilon _a \sigma _B \vert T_a \vert ^3 T_a + r(t) q(x) \beta_a (T_a) 
\vspace{.1cm}\\
\displaystyle{\gamma _s} \left[ \frac{\partial T_s}{\partial t} -  \kappa_s \frac{\partial }{\partial x} \left( (1-x^2) \frac{\partial T_s}{\partial x}\right) \right]\vspace{.15cm} \\
\hspace{1.5cm}= -\lambda (T_s - T_a) - \sigma _B \vert T_s \vert ^3 T_s  + \varepsilon _a \sigma _B \vert T_a \vert ^3 T_a + r(t) q(x) \beta_s (T_s) , \vspace{.15cm} \\
(1-x^2) \displaystyle{\frac{\partial T_a}{\partial x} _{\vert x = \pm 1}} = 0 = (1-x^2) \frac{\partial T_s}{\partial x} _{\vert x = \pm 1},\vspace{.1cm}
\end{cases}
\end{equation}
with initial conditions
\begin{equation}
\label{eq:bound and in cond}
T_a (0,x)=T_a ^{0} (x) ,
 \ \ \ \ 
T_s (0,x)=T_s ^{0} (x),\quad   x \in I.
\end{equation}
\begin{Remark}
    {\rm We observe that the fourth powers $T^4_a$ and $T^4_s$ in \eqref{eq:pde-model} have been replaced with the nonlinear terms $|T_a|^3 T_a$ and $|T_s|^3 T_s$ in the above equations. Indeed, such a structure of the nonlinear terms allows for a better mathematical treatment of the long time behavior of solutions for general initial conditions. On the other hand, when the model is restricted to the physical relevant regime of positive initial conditions, the two formulations coincide, since the corresponding solutions remain positive (see Proposition \ref{prop-nonneg}).}
\end{Remark}
Throughout the paper, we consider the following set of assumptions
\begin{Hyp.}\label{HYP1} \hfill
\begin{itemize}
\item[{\bf (i)}]  The coefficients $\gamma_a,\,\gamma_s,\,\sigma_B,\,\eps_a,\,\lambda$ satisfy
$$\gamma_a, \gamma_s >0, \quad  \kappa_a, \kappa_s >0, \quad 
\sigma_B>0, \quad \varepsilon_a>0, \quad \lambda \geq 0.$$
\item[{\bf(ii)}] The coalbedo functions $\beta_a, \beta _s:\RR\to\RR$ verify
\begin{equation*}
\beta _a, \beta _s \in W^{1,\infty}(\RR) \text{ and are nondecreasing }.
\end{equation*}
\item[{\bf (iii)}] The incoming solar flux $Q(t,x)=r (t)q (x)$ satisfies
\begin{equation*}
q \in C^{\theta_0}(I)\cap L^\infty(I),\text{ for some }\theta_0\in(0,1),\quad q \geq 0,
\end{equation*}
\begin{equation*}
 r \text{\,\,is positive, bounded and globally Lipschitz in }(0,+\infty).
\end{equation*}
\end{itemize}
\end{Hyp.}
We observe that assuming $\beta _a, \beta _s \geq 0$ is meaningful for the model interpretation. However, this hypothesis is not necessary to prove the well-posedness of the problem and the existence of an attractor. 

\smallskip

\begin{Remark}
\label{rem:eps_a}
{ \rm While the assumption $\varepsilon_a > 0$ is sufficient to ensure local existence of solutions, studying global existence requires further restrictions on the range of $\varepsilon_a$. For the corresponding ODE system, we proved in \cite{CLMUV-EDO} that blow-up may occur when $\varepsilon_a > 2$, whereas global existence holds for $\varepsilon_a \in (0, 2)$. In the present work, we establish global existence for the PDE model \eqref{2layer-pbm}-\eqref{eq:bound and in cond} when $\varepsilon_a \in (0, 2)$, and demonstrate the possibility of blow-up when $\varepsilon_a > 2$.

The mathematical analysis can be complemented with physics-informed comments. 
From a physical standpoint, the absorptivity $\varepsilon_a$ of a single atmospheric layer represents the fraction of infrared radiation emitted by the surface that is absorbed by the atmosphere  ($\varepsilon_a = 1$ corresponds to 100\% absorption, while $\varepsilon_a = 0$ corresponds to 0\% absorption). Basic physical principles therefore constrain $\varepsilon_a$ to lie in the interval $(0, 1]$: the value $\varepsilon_a = 1$ corresponds to a perfectly opaque atmosphere (one that absorbs all outgoing infrared radiation), while $\varepsilon_a \to 0$ corresponds to a fully transparent one. Moreover, showing that larger values of $\varepsilon_a$ lead to a warmer surface temperature amounts to mathematically proving the greenhouse effect within this simple model.

When the atmosphere is very opaque, the physically correct modelling procedure 
requires adding extra atmospheric layers, each behaving radiatively as a blackbody 
and stacked on top of each other, in order to respect the basic laws of 
thermodynamics (see e.g.\ \cite{Hartmann}). In this sense, the range 
$\varepsilon_a \in (1, 2)$ has no direct physical interpretation within the single-layer framework.

Nevertheless, the mathematical threshold $\varepsilon_a = 2$ retains a clear 
physical significance, and the range $\varepsilon_a \in (1,2)$ is not devoid of 
physical meaning. If one accepts the atmospheric single-layer model with $\varepsilon_a \geq 1$ as a rough representation of a very opaque atmosphere, then the following picture emerges. For $\varepsilon_a \in (1, 2)$, global solutions still exist. 
This can be interpreted as a regime in which a very opaque atmosphere, while leading 
to a warm surface, remains compatible with a stable energy balance: the system avoids 
the runaway greenhouse instability. By contrast, when $\varepsilon_a \geq 2$, 
finite-time blow-up occurs, which can be interpreted as the mathematical signature of 
the runaway greenhouse effect. This effect manifests itself physically when surface 
warming triggers excessive evaporation, further increasing atmospheric opacity and 
driving the system away from any stable equilibrium --- a mechanism that is at work 
on Venus. The threshold $\varepsilon_a = 2$ thus marks a sharp transition between a 
regime of stable energy balance and a regime of runaway instability, and is in this 
sense not a mere mathematical artifact.
} 
\end{Remark}

\subsection{Functional setting for scalar degenerate diffusion operators\vspace{.2cm}} \hfill

Let us start by describing the functional setting for a single degenerate equation.

Let  
$$\rho(x):=1-x^2, \qquad x \in \bar I,$$
denote the coefficient of the diffusion operator. Since $\rho$ vanishes at both the endpoints $x=\pm1$, the diffusion operator degenerates at the boundary of the domain. Thus, the problem must be to formulated in suitable weighted function spaces. 

\subsubsection{The spaces $H$ and $V$\vspace{.2cm}} \hfill

Following \cite[Equation (10)]{Campiti1998}, we consider 
$$ H:= L^2(I),$$
endowed with the standard inner product, and we introduce the space
\begin{equation}
\label{eq:def V}
V:=\left\{u\in H:\, u\text{ locally absolutely continuous in }I\,,\,\sqrt{\rho}\,\frac{d}{d x} u\in L^2(I)\right\}.
\end{equation}
The space $V$ is endowed with the inner product\begin{equation*}
(u,v)_V:=(u,v)_{H}+\left(\sqrt{\rho}\,\frac{d u}{dx},\sqrt{\rho}\,\frac{d v}{dx}\right)_{H},\qquad\forall\,u,v\in V,
\end{equation*}
and with the associated norm
\begin{equation*}
\norm{u}_V:=\sqrt{(u,u)_V} = \sqrt{\Vert  u \Vert^2_{L^2(I)} + \norm{ \sqrt{\rho}\,\frac{d u}{dx}} ^2_{L^2(I)} } ,\qquad\forall\,u\in V.
\end{equation*}
With this structure, $(V,\norm{\cdot}_V)$ is a Hilbert space and is embedded into $H^1_{loc}(I)$. 

In the following proposition, we collect several useful properties about $V$. 

\begin{Proposition}\label{resu-prop-V}
The space $V$ satisfies the following properties: 
\begin{enumerate}
\item $ C_0 ^\infty (I)$, the space of smooth functions that are compactly supported in $I$, is dense in $V$;
\item $V$ is continuously embedded in $L^p(I)$ for $1\leq p<+\infty$;
\item $V$ is compactly embedded in $L^2(I)$.
\end{enumerate}
\end{Proposition}

\begin{proof}
Assertion (1) is proved in \cite[Lemma 2.6]{Campiti1998}. Assertion (2) follows from \cite[Lemma (ii)]{Diaz1997}. Finally, assertion (3) can be obtained by adapting the proof of \cite[Theorem 6.1]{AL-CAN-FRA}, which is stated for a diffusion coefficient that vanishes at a single endpoint. 
\end{proof}

\begin{Remark}
{\rm We observe that a stronger property than (3) has recently been proved in \cite{BCCKU}: $V$ is compactly embedded in $L^p(I)$, for all $p\geq1$.}
\end{Remark}

\begin{Remark}\label{rq-trace-V}{\rm 
We observe that $V$ is not embedded into $L^\infty(I)$ and, in general, functions in $V$ do not necessarily admit well-defined traces at the endpoints $x=\pm1$. 

As an illustrative example, consider the function
$$ \forall\, x \in I, \quad u(x) := \left \vert  \ln (1-x^2) \right \vert  ^\beta,$$
where $\beta \in (\frac{1}{4}, \frac{1}{2})$. A direct computation shows that $u\in V$, while $u(x)\to+\infty$ as $x \to \pm 1$. 
}
\end{Remark}

\subsubsection{The operator $A$ and the space $D(A)$\vspace{.2cm}} \hfill

We now define the degenerate operator $A: D(A)\subseteq H \to H$  (see \cite[Equation (3)]{Campiti1998}):
\begin{align*}
& D(A):=
\left\{u\in H:\, u\text{ locally absolutely continuous in }I,\, 
\rho \frac{d u}{dx}\in H^1_0(I)
\right\},
\\
& Au:=\frac{d}{dx}\left(\rho \frac{d u}{dx}\right),\quad\forall\,u \in D(A).
\end{align*}

In the following proposition we collect several fundamental properties of the operator $(A,D(A))$. 
\begin{Proposition}\label{resu-prop-A}
The operator $A: D(A)\subseteq H \to H$ enjoys the following properties:
\begin{enumerate}
\item  $(A,D(A))$ is self-adjoint, dissipative, and has a compact resolvent and dense domain. 
Consequently, $A$ is the infinitesimal generator of an analytic semigroup of contractions $\{e^{tA}\}_{t\geq0}$ on $H$;
\item The domain $D(A)$ coincides with the space 
$$\left\{u\in V\,:\,\rho \frac{d u}{dx}\in H^1(I)\right\};$$ 
\item The real interpolation space between $D(A)$ and $H$ satisfies 
$$[D(A), H]_{1/2} = V.$$
 \end{enumerate}
\end{Proposition}

\begin{proof}
For assertion (1), we refer for instance to \cite[Proposition 2.2]{CPAA2004}, where it is shown that an operator with a diffusion coefficient that vanishes at a single endpoint is closed, self-adjoint, dissipative, and has dense domain. The extension to the current case, in which the diffusion coefficient degenerates at both endpoints, is straightforward. The generation of an analytic semigroup of contractions follows from \cite[Proposition 2.11]{bensoussan1992}. We also refer to \cite[Lemma 2.7, Theorem 2.8]{Campiti1998} for alternative proofs of the self-adjointness of $A$ and analyticity of the associated semigroup. 

Assertion (2) is proved in \cite[Lemma 2.5]{Campiti1998}. We emphasize that the Neumann boundary conditions are implicitly encoded in the definition of the domain $D(A)$ in both characterizations. Indeed, if $(\rho u_x)_x\in H$, one can show that $\left.(\rho u_x)\right|_{x=\pm1}=0$ in the sense of traces.

Finally, owing to the variational framework, we have the embeddings $V \subset H \subset V'$. Assertion (3) then follows from the classical interpolation result \cite[Prop. 2.1, p. 22]{lions1968problemes}. 
An alternative proof can be obtained by adapting the computations in \cite[Supplementary materials, SM2]{CA-MA-UR} which are carried out in the case of a diffusion coefficient degenerating at a single endpoint. 
\end{proof}

\medskip

The following regularity result plays a key role in the subsequent analysis.
\begin{Proposition}
\label{prop-reg-borne}
The domain $D(A)$ is continuously embedded in $C (\bar I)$.
More precisely, for every $u$ of $D(A)$ it holds that
\begin{equation*}
    u^2 \in W^{1,1} (I).
\end{equation*} 
Moreover, there exists a constant $C$ such that
$$ \forall\, u \in D(A), \quad \Vert u \Vert _{C (\bar I)} \leq C \Vert u \Vert _{D(A)} .$$
\end{Proposition}

\begin{Remark}
{\rm In \cite[p. 520]{BSchmidt}, it is observed that the injection of $D(A)$ into $C(\bar I)$ is compact, based on an interpolation argument. For completeness, in Appendix \ref{app: A} we provide a proof of Proposition \ref{prop-reg-borne}, adapting the ideas developed in \cite{CMV-note}.
}\end{Remark}


\subsection{Existence and uniqueness of solutions\vspace{.2cm}} \hfill
\label{sec:results} 

Consider the two-layer model \eqref{2layer-pbm}-\eqref{eq:bound and in cond}
with $(t,x)\in (0,\TT) \times  I$ for some $\TT>0$. We rewrite it in vectorial form. 
{\sl From now on, bold symbols denote vectors of two components}. We introduce the product spaces
\begin{equation}
\label{eq:def H V}
  \bH:=H \times H \quad \text{ and } \quad \bV:= V\times V,  
\end{equation} 
and 
\begin{equation}\label{def:V+}
    \bV_+=\{\pmb u\in \bV\,:\, \pmb u\geq\pmb 0\}.
\end{equation}
Note that whenever an inequality between vectors is written, it is always understood component-wise.

We define
$$\pmb B_{\bV}(R):=\{\pmb{u}\in \bV\,:\,\norm{\pmb u}_{\bV}< R\},\qquad \pmb B_\bV^+(R):=\{\pmb{u}\in \bV_+\,:\,\norm{\pmb u}_{\bV}< R\}.$$
We set  
$$\bT:=\begin{pmatrix}T_a\\T_s \end{pmatrix}, 
\qquad 
\bTi:=\begin{pmatrix}T_a ^0\\T_s ^0 \end{pmatrix} .$$
We define the operator $\bA : \bDA\subseteq \bH \to \bH $ by
\begin{align}
\label{eq:def A}
&\bDA:=D(A)\times D(A),
\\ 
& \bA\pmb{u}:=\bA\begin{pmatrix}u_a\\u_s\end{pmatrix}=\begin{pmatrix} \kappa_a Au_a&0\\0&\kappa_s Au_s \end{pmatrix},\quad\forall\,\pmb{u}=(u_a,u_s) \in \bDA.
\end{align}
The operator $\bA$ is self-adjoint, dissipative, has a compact resolvent and a dense domain; hence it generates a strongly continuous semigroup $\{e^{t \bA}\}_{t\geq0}$ on $\bH$. The Neumann boundary conditions are encoded in the definition of $\bDA$. Notice that $\bDA$ is a Hilbert space with the graph inner product \begin{equation*}
\langle\pmb u,\pmb v\rangle_{\bDA}=(\pmb u,\pmb v)_\bH+(\bA \pmb u, \bA \pmb v)_\bH,\quad \forall\,\pmb u,\,\pmb v \in\bDA.
\end{equation*}
We also introduce the space
\begin{equation}\label{def:DA+}
    \bDA_+=\{\pmb u\in \bDA\,:\, \pmb u\geq\pmb 0\}
\end{equation}
and define
\begin{equation}
\label{dfn:ball in D(A)}    
    \pmb B_{\bDA}(R):=\{\pmb{u}\in \bDA\,:\,\norm{\pmb u}_{\bDA}< R\},
\end{equation}
\begin{equation}
\label{dfn:ball in D(A)+}
    \pmb B_\bDA^+(R):=\{\pmb{u}\in \bDA_+\,:\,\norm{\pmb u}_{\bDA}< R\}.
\end{equation}
The nonlinear term $\pmb{G}: [0,+\infty) \times {\bf V} \to {\bf H}$ is defined as follows
\begin{equation*}
\pmb{G} \left(t,\begin{pmatrix} u_a\\u_s\end{pmatrix}\right)
:=\begin{pmatrix} 
\frac{1}{\gamma_a}\left[-\lambda(u_a-u_s)+\varepsilon_a\sigma_B|u_s|^3u_s-2\varepsilon_a\sigma_B|u_a|^3u_a+ r(t) \, q \, \beta_a(u_a) \right]  
\\ 
\frac{1}{\gamma_s}\left[-\lambda(u_s-u_a)-\sigma_B|u_s|^3u_s+\varepsilon_a\sigma_B|u_a|^3u_a+ r(t) \,q \,\beta_s(u_s) 
\right] \end{pmatrix}.
\end{equation*}
With this notation, system \eqref{2layer-pbm}-\eqref{eq:bound and in cond} can be written as the abstract Cauchy problem
\begin{equation}\label{2layers-abstract-form}
\left\{
\begin{array}{ll}
\bT'(t)=\bA\bT(t)+\bG(t,\bT(t)),& \qquad t\in(0,\TT),\\
\bT(0)=\bTi.
\end{array}
\right.
\end{equation}
In the result that follows we collect some properties of $\bG$ which will play a key role to deduce the local existence of solutions to \eqref{2layers-abstract-form}.

\begin{Lemma}
\label{lem-propr-G}
Let $\TT>0$ and assume that Hyp. \ref{HYP1} are satisfied. $\bG$ is well-defined and continuous on  $[0,\TT]\times \bV\to \bH$. Then $G$ satisfies
\begin{itemize}
\item $\exists\,C>0$ such that
\begin{equation}\label{propr2-G}
\forall\,\pmb{u}\in \bV,\,\forall\,t,s\in[0,\TT], \qquad 
\norm{\bG(t,\pmb{u})-\bG(s,\pmb{u})}_{\bH}\leq C|t-s|;
\end{equation}
\item $\exists\, C_1: \bV \times \bV \to [0, +\infty)$ locally bounded  such that
\begin{equation}\label{propr3-G}
\forall\, \pmb{u},\pmb{\tilde u}\in \bV,\, \forall\,t\in[0,\TT],
\qquad
\norm{\bG(t,\pmb{u})-\bG(t,\pmb{\tilde u})}_{\bH}\leq C_1(\pmb{u},\pmb{\tilde u}) \norm{\pmb{u}-\pmb{\tilde u}}_{\bV}.
\end{equation}
\end{itemize}
\end{Lemma}
The proof of Lemma \ref{lem-propr-G} can be found in Appendix \ref{app: A}.

Let us now recall some definitions of solution of \eqref{2layers-abstract-form} (see \cite[Def. 7.0.1, pp. 254-255]{lunardi1995analytic}).

\begin{Definition} \hfill
\begin{enumerate}
\item Let $\bTi\in \bDA$.
A function $\bT \in C^1([0,\TT];\bH)
 \,\cap \,
 C([0,\TT];\bDA)$
is called a {\rm strict solution} of \eqref{2layers-abstract-form} in the interval $[0, \TT]$ if $ \bT'(t)=\bA\bT(t)+\bG(t,\bT(t))$ for each $t\in [0,\TT]$
and $\bT(0)=\bTi$.

\item Let $\bTi\in \bH$.  A function $\bT \in C^1((0,\TT];\bH)
 \,\cap \,
 C((0,\TT];\bDA)
 \,\cap \,
 C([0,\TT];\bH)
 $
is called a {\rm classical solution} of \eqref{2layers-abstract-form} in the interval $[0, \TT]$ if 
$ \bT'(t)=\bA\bT(t)+\bG(t,\bT(t))$  for each $t\in (0,\TT]$
and $\bT(0)=\bTi$.
\end{enumerate}
\end{Definition}

From \cite[Prop. 4.1.2]{lunardi1995analytic}, it follows that any classical solution $\bT$ of \eqref{2layers-abstract-form} such that $ \bG(\cdot, \bT (\cdot) ) \in L^1(0,\TT; \bH) $ satisfies the variation of constants formula
\begin{equation}\label{form-var-cste}
\bT(t)=e^{t\bA}\bTi+\int_{0}^t e^{(t-s)\bA}\bG(s,\bT(s))ds,
\qquad 0 < t 
\leq\TT.
\end{equation}
This motivates the definition of {\it mild} solutions: 
\begin{Definition}
Let $\TT \geq 0$ and let $\bTi\in \bV$  be given. A function $\bT \in C([0,\TT];\bV)$
is called a {\rm mild solution} of \eqref{2layers-abstract-form} in the interval $[0, \TT]$ if it satisfies the variation of constant formula \eqref{form-var-cste}.
\end{Definition}  
We observe that since $\bT\in C([0,\TT];\bV)$, we have that $t\mapsto \bG(t,\bT(t))$ is of class $L^\infty(0,\TT;\bH)$. We can now state existence and regularity results for solutions of \eqref{2layers-abstract-form}.
\begin{Proposition}[\textbf {Local well-posedness}]
\label{thm-loc-wellpos}
Assume that Hyp \ref{HYP1} is satisfied. For every $R>0$ there exists $\tau_R>0$ such that for every $\bTi\in\overline{\pmb B}_\bV(R)$, problem \eqref{2layers-abstract-form} has a unique mild solution $\bT\in C([0,\tau_R];\bV)$. Moreover, there exists $M(R)>0$ such that
\begin{equation}
\label{eq-T-bounded}
    \norm{\bT(t)}_\bV\leq M(R),\qquad\forall\,t\in[0,\tau_R],\,\forall\,\bTi\in \overline{\pmb B}_\bV(R).
\end{equation}
Furthermore 
\begin{equation}
\label{eq: max reg}
    \bT\in H^1(0,\tau_R;\bH)\cap L^2(0,\tau_R;\bDA).
\end{equation}
and equation in \eqref{2layers-abstract-form} is satisfied for a.e. $t\in [0,\tau_R]$.
\end{Proposition}
The proof follows from a standard fixed point argument (see, for instance, \cite[Theorem 7.1.3]{lunardi1995analytic} or \cite[Theorem 3.3.3]{henry}). The regularity \eqref{eq: max reg} derives from the fact that $\bG(\cdot,\bT(\cdot))\in L^2(0,\TT;\bH)$ and from the analyticity of the semigroup (see, for instance, \cite[Theorem 3.1, p. 143]{bensoussan1992}). 

Such solutions can be extended by classical continuation arguments yielding the following
\begin{Proposition}[\textbf {Existence of maximal solutions}]
\label{thm-max-mild}
Assume that Hyp \ref{HYP1} is satisfied. For every initial value $\bT ^0\in \bV$, problem \eqref{2layers-abstract-form} has a unique maximal 
solution, defined on
$$I(\bTi)=[0, \TT ^* ({\bf T}^0)) .$$
Moreover, the following alternative holds 
$$
\TT^* (\bTi) = +\infty 
\quad \text{ or } \quad 
\left[ 
\TT^* (\bTi) < +\infty \text{ and } 
\limsup_{t \to \TT^*(\bTi) }  \Vert \bT (t) \Vert_{\bV} = +\infty \right].
$$
\end{Proposition}
\begin{Remark}
{\rm  
Obtaining a priori bounds of the $\bV$-norm of $\bT$ will be crucial to establish global existence. General results such as \cite[Prop. 7.2.1 or Prop. 7.2.2]{lunardi1995analytic} cannot be applied to system \eqref{2layers-abstract-form} due to the high-order nonlinearities. Instead, we will strongly exploit the specific cooperative structure of \eqref{2layer-pbm} to derive the required $\bV$-estimates. 
}
\end{Remark} 
\begin{Theorem}[\textbf {Global existence of solutions}]\label{prop-bounds}
Assume that Hyp \ref{HYP1} holds, that $\varepsilon _a \in (0,2)$, $\beta_a \geq 0$, $\beta_s >0$, and that $q>0$. 
Let $\bTi= (T_a ^{0}, T_s ^{0}) \in \bV_+$. Then problem \eqref{2layers-abstract-form} admits a unique mild solution $\bT= (T_a,T_s)$, which is globally defined for all $t\in[0,+\infty)$. Moreover, for every $\tau \in (0, +\infty )$ there exists a constant $M_{\tau}>0$ such that
\begin{equation}
\label{eq-bounds}
\forall\, t \in [\tau, +\infty ),\,\forall\,x\in I \qquad  0 \leq T_a  (t,x) \leq M_{\tau}, \qquad 0 \leq T_s (t,x) \leq M_{\tau}. 
\end{equation}
\end{Theorem}
The proof (postponed to Section \ref{sec-proof-7}) proceeds as follows. Starting from the maximally defined mild solution provided by Proposition \ref{thm-max-mild}, we establish uniform $L^\infty$ bounds of the solution components by means of the comparison principle stated in Theorem \ref{lem-comp-2layer-sub-sup}. These bounds are combined with suitable energy estimates to control  the $\bV$-norm of the solution. 

As a consequence, finite-time blow-up in the $\bV$-norm is excluded for $\varepsilon _a \in (0,2)$, and the maximal existence time satisfies $\TT^*(\bTi)=+\infty$. Since comparison arguments play a central role, the complete proof of the theorem is deferred to Section~\ref{sec-proof-7}.

\begin{Theorem}[\bf Blow-up]
\label{prop-blowup}
Assume that Hyp \ref{HYP1} holds, that $q>0$, $\beta_a \geq 0$, $\beta_s >0$,
and let $\varepsilon _a >2$. Then,
\begin{itemize}
\item[(a)] there exist initial conditions $\bTi= (T_a ^{0}, T_s ^{0})\in \bV_+$ such that the associated solutions $\bT= (T_a,T_s)$ of \eqref{2layers-abstract-form} blow up in finite time, 
\item[(b)] if $\lambda=0$, for all initial conditions $\bTi= (T_a ^{0}, T_s ^{0}) \in \bDA_+$ the associated solutions  $\bT= (T_a,T_s)$ of \eqref{2layers-abstract-form} blow up in finite time. 
\end{itemize}
\end{Theorem}
Also in this case, the proof is based on the comparison principle and therefore postponed to Section~\ref{sec-proof-7}.

\begin{Proposition}{\rm \bf (Continuous dependence on the initial condition)}\label{prop-dep-continue-CI}

Assume that Hyp. \ref{HYP1} is satisfied. Let $\,\bTi \in  \bV$ and let $\bT(\cdot; \bTi )$ denote the corresponding maximal solution of \eqref{2layers-abstract-form}, defined on the interval $I(\bTi)$. Then, for every $\tau < \TT^* (\bTi)$ there exist constants $r,K>0$ such that, for any $\tildebTi \in  \bV$ satisfying
$$\Vert \bTi - \tildebTi  \Vert _\bV \leq r,$$
the corresponding mild solution $\bT(\cdot;\tildebTi)$ is defined at least on $[0,\tau]$, that is,
$$ \TT^* (\tildebTi)  \geq \tau $$ 
and the following stability estimate holds:
$$\Vert  T(t;\bTi) - T(t; \tildebTi ) \Vert_\bV 
\leq K \Vert  \bT^0 - \tildebTi \Vert_\bV
, \quad  0 \leq t \leq \tau.$$
\end{Proposition}
\begin{proof}
The result follows directly from \cite[Prop. 7.1.9, points (i) and (ii),  pp. 266-267]{lunardi1995analytic}. The assumptions of that proposition are satisfied in our setting since $\overline{\bDA}^{\bV} = \bV$.
\end{proof}
\begin{Proposition}[\bf Existence of classical and strict solutions]
\label{thm-further-reg} 
Assume that Hyp. \ref{HYP1} is satisfied. Let $\bTi \in  \bV$ and denote by $\bT=\bT(\cdot; \bTi)$ the corresponding maximal solution of \eqref{2layers-abstract-form}, defined on the interval $I(\bTi) =[0, \TT ^* ({\bf T} ^0))$. Then $\bT$ is a classical solution of \eqref{2layers-abstract-form} in $I(\bTi)$. Moreover, for any
$$0< \varepsilon < \tau<\TT ^* ({\bf T} ^0),$$ 
there exists $\theta_1\in(0,1)$ such that
$$ \bT \in  C^{\theta_1}  ([\varepsilon, \tau]; \bDA) \cap  C^{1+ \theta_1}  ([\varepsilon, \tau]; \bH),$$
where we have set for any $\theta \in(0,1)$
\begin{equation*}
    C^\theta  ([\varepsilon, \tau]; \bDA)=\left\{\pmb{u}\in C^0  ([\varepsilon, \tau]; \bDA)\,:\, \sup_{t\neq s}  \frac{\norm{\pmb{u}(t)-\pmb{u}(s)}_{\bDA}}{|t-s|^\theta}\right\},
\end{equation*}
and
\begin{equation*}
    C^{1+ \theta}  ([\varepsilon, \tau]; \bH)=\left\{\pmb{u}\in C^1 ([\varepsilon, \tau]; \bH)\,:\, \pmb{u}_t\in  C^\theta  ([\varepsilon, \tau]; \bH) \right\}.
\end{equation*} 
If, in addition, $\bTi \in \bDA$, then $\bT$ is also a strict solution of \eqref{2layers-abstract-form} on $I(\bTi)$. In particular,
$$\bT \in C^1([0,\TT];\bH) \cap C([0,\TT];\bDA),\quad\text{for all }0\leq \TT< \TT ^* ({\bf T} ^0).$$
\end{Proposition}
\begin{proof}
The result follows from \cite[Prop. 7.1.10, pp. 268-269]{lunardi1995analytic}. To apply this proposition, it suffices to verify assumption \cite[(7.1.29), p.~268]{lunardi1995analytic}, namely, that there exists $\theta \in (0,1)$ such that for every $\bu^0 \in \bV$ there exist constants $r,K>0$ with
$$ \Vert \bG(t, \bu) - \bG(s,  \bu) \Vert _\bH \leq K \vert t-s \vert ^\theta$$
for all $0 \leq s \leq t \leq \TT $ and all $\bu \in \bV$ satisfying $\Vert \bu - \bu^0 \Vert _\bV \leq r $. This condition is fulfilled in our setting thanks to Lemma \ref{lem-propr-G}. Therefore, \cite[Proposition~7.1.10, items~(i) and~(iii)]{lunardi1995analytic} applies and yields the stated regularity results.
\end{proof}
Let us introduce the following notation for the 2-layer system
\begin{equation}
\label{2layer-pbm-f}
\begin{cases}
\displaystyle  \frac{\partial T_a}{\partial t} -  \kappa_a \frac{\partial }{\partial x} \left( (1-x^2) \frac{\partial T_a}{\partial x}\right) = \frac{1}{\gamma_a}f_a(t,x), 
& x \in I, t>0  
\vspace{.2cm}\\
\displaystyle  \frac{\partial T_s}{\partial t} -  \kappa_s \frac{\partial }{\partial x} \left( (1-x^2) \frac{\partial T_s}{\partial x}\right)  = \frac{1}{\gamma _s}f_s(t,x)
\vspace{.2cm}\\
(1-x^2) \displaystyle{\frac{\partial T_a}{\partial x} _{\vert x = \pm 1}} = 0 = (1-x^2) \frac{\partial T_s}{\partial x} _{\vert x = \pm 1} ,
&   t>0\vspace{.2cm}\\
T_a (0,x)=T_a ^{0} (x) ,\quad
T_s (0,x)=T_s ^{0} (x), &   x \in I,
\end{cases}
\end{equation}
where $f_a,f_s$ denote the right-hand side of \eqref{2layer-pbm}.

Given $\TT>0$, we define
$$Q_\TT:=(0,\TT)\times I.$$
Let $D$ be a compact subset of $Q_\TT$, $D\Subset Q_\TT$. Then, given $\theta\in(0,1)$ we recall that $C^{\frac{\theta}{2},\theta}(D)$ denotes the space of all $u\in C(D)$ such that
\begin{equation*}
    \sup\left\{\frac{|u(t,x)-u(s,y)|}{|t-s|^{\frac{\theta}{2}}+|x-y|^\theta}<\infty\,:\, (t,x),\,(s,y)\in D\text{ and }(t,x)\neq(s,y)\right\}.
\end{equation*}
Hereafter, we will use the notation $u_x$ to denote the derivative $\frac{\partial u}{\partial x}$ for brevity.

Let us also recall the definitions of the following spaces:
\begin{equation*}
   C^{1,2}(Q_\TT)=\left\{u\in C(Q_\TT)\,:\, u_x,\,u_t,\,u_{xx}\in C(Q_\TT)\right\}, 
\end{equation*}
and for any $\theta\in(0,1)$
\begin{equation*}
    C_{\text{loc}}^{\frac{\theta}{2},\theta}(Q_\TT)=\left\{ u\in C(Q_\TT)\,:\, u\in C^{\frac{\theta}{2},\theta}(D),\,\forall\,D\Subset Q_\TT\right\},
\end{equation*}
\begin{equation*}
    C_{\text{loc}}^{1+\frac{\theta}{2},2+\theta}(Q_\TT)=\left\{u\in C^{1,2}(Q_\TT)\,:\, u_t,\,u_{xx}\in C_{\text{loc}}^{\frac{\theta}{2},\theta}(Q_\TT)\right\}.
\end{equation*}
\begin{Proposition}[\bf Local regularity of classical solutions]\label{prop: loc reg}
    Assume that Hyp \ref{HYP1} holds. Let $\bTi=(T_a^0,T_s^0)\in \bV$ and denote by $(T_a,T_s)$ the corresponding solution of \eqref{2layer-pbm-f}. 
    
    Then, for any $0<\TT<\TT^*(\bTi)$, there exists $\theta\in(0,1)$ such that
    \begin{equation}
        T_a,T_s\in C^{1+\frac{\theta}{2},2+\theta}_{\text{loc}}(Q_\TT).
    \end{equation}
\end{Proposition}
\begin{proof}
The proof is divided into three steps:
    \begin{itemize}
        \item {\bf Step 1.} First, we show that there exists $\theta\in(0,1)$ such that
        \begin{equation}
            f_a,f_s\in C^{\frac{\theta}{2},\theta}_{\text{loc}} (Q_\TT),
        \end{equation}
        \item {\bf Step 2.} We then prove that if $u\in W^{1,2}_{\text{loc}}(Q_\TT)$ is the solution of the following equation
        \begin{equation}\label{eq:u}
            \frac{\partial u}{\partial t}-\frac{\partial}{\partial x}\left((1-x^2)\frac{\partial u}{\partial x}\right)=f
        \end{equation}
        where $f\in C^{\frac{\theta}{2},\theta}_{\text{loc}}(Q_\TT)$, $\theta\in(0,1)$, then
        \begin{equation*}
            u\in C^{1+\frac{\theta}{2},2+\theta}_{\text{loc}}(Q_\TT)
        \end{equation*}
        \item {\bf Step 3.} Finally, by applying the first two steps, we conclude that the solution of \eqref{2layer-pbm-f} satisfies
        \begin{equation*}
            T_a,T_s\in C^{1+\frac{\theta}{2},2+\theta}_{\text{loc}}(Q_\TT).
        \end{equation*}
     \end{itemize}
{\bf Step 1.} Since $T_a,T_s\in V$, Propositions \ref{thm-max-mild} and \ref{thm-further-reg} imply the existence of a mild solution $\bT=(T_a,T_s)$ of \eqref{2layer-pbm-f} such that for any $0<\eps<\TT$
\begin{equation}
    \bT\in C([0,\TT];\bV)\quad\text{and}\quad\bT\in C^{\theta_1}([\eps,\TT];\bDA)\cap C^{1+\theta_1}([\eps,\TT];\bH),
\end{equation}
with $\theta_1\in(0,1)$. Furthermore, recalling that $q$ is H{\"o}lder continuous, $r$ is Lipschitz continuous, the regularity of the co-albedo functions $\beta_a,\beta_s$, and that $V\subset H^1_{\text{loc}}(I)$, we deduce that there exists $\theta\in(0,1)$ such that
\begin{multline*}
    f_a(t,x)=-\lambda (T_a(t,x) - T_s(t,x))  
+ \varepsilon _a \sigma _B \vert T_s(t,x) \vert ^3 T_s(t,x)  \\
- 2 \varepsilon _a \sigma _B \vert T_a (t,x)\vert ^3 T_a(t,x) + r(t) q(x) \beta_a (T_a(t,x)),
\end{multline*}
and 
\begin{multline*}
    f_s(t,x)=\lambda (T_s(t,x) - T_a(t,x)) - \sigma _B \vert T_s(t,x) \vert ^3 T_s(t,x)  \\
    + \varepsilon _a \sigma _B \vert T_a(t,x) \vert ^3 T_a(t,x) + r(t) q(x) \beta_s (T_s(t,x))
\end{multline*}
belong to $C^{\frac{\theta}{2},\theta}_{\text{loc}}(Q_\TT)$.

{\bf Step 2.} Let $u\in W^{1,2}_{\text{loc}}(Q_\TT)$ be a solution to \eqref{eq:u} with $f\in C^{\frac{\theta}{2},\theta}_{\text{loc}}(Q_\TT)$, $\theta\in(0,1)$. For any $\eps>0$, we define the subsets of $Q_\TT$
\begin{equation*}
    Q_\TT^\eps:=(\eps,\TT-\eps)\times(-1+\eps,1-\eps),\qquad Q^{2\eps}_\TT:=(2\eps,\TT-2\eps)\times(-1+2\eps,1-2\eps)
\end{equation*}
and a cut-off function $\xi^\eps \in C^\infty_0(Q^\eps_\TT)$ such that
\begin{equation*}
    0\leq\xi\leq1\quad\text{and}\quad \xi\equiv 1\quad\text{in}\quad Q^{2\eps}_\TT.
\end{equation*}
We localize problem \eqref{eq:u} by introducing
\begin{equation*}
    u^\eps(t,x):=\xi^\eps(t,x)\,u(t,x).
\end{equation*}
Then, it is possible to check that $u^\eps\in W^{1,2}(Q^\eps_\TT)$ satisfies
\begin{equation*}
    \frac{\partial u^\eps}{\partial t}-\frac{\partial}{\partial x}\left((1-x^2)\frac{\partial u^\eps}{\partial x}\right)=f^\eps,
\end{equation*}
with
\begin{equation*}
    f^\eps:=\xi^\eps f+\frac{\partial \xi^\eps}{\partial t}\,u-(1-x^2)\frac{\partial^2 \xi^\eps}{\partial x^2}\,u-2(1-x^2)\frac{\partial \xi^\eps}{\partial x}\,\frac{\partial u}{\partial x}.
\end{equation*}
On the other hand, a direct computation shows that $f^\eps\in C^{\frac{\theta}{2},\theta}(Q^\eps_\TT)$, for some $\theta\in(0,1)$, and therefore we can apply \cite[Theorem 5.14]{Lieberman} (see also \cite[Theorem 16.1]{Ladyz} or \cite[Theorem 10.3.3]{Krylov}) and deduce that there exists a solution $v\in C^{1+\frac{\theta}{2},2+\theta}(Q^\eps_\TT)$ to the problem
\begin{equation*}
    \begin{cases}
    \displaystyle\frac{\partial v}{\partial t}-\frac{\partial}{\partial x}\left((1-x^2)\frac{\partial v}{\partial x}\right)=f^\eps& \text{in }Q^\eps_\TT\\
    v=0&\text{on }\partial Q^\eps_\TT.
    \end{cases}
\end{equation*}
Since $u^\eps$ solve the same equation, by the uniqueness of solutions in $W^{1,2}(Q^\eps_\TT)$ we deduce that $\xi^\eps u=v$ in $Q^\eps_\TT$. Consequently, $u\in C^{1+\frac{\theta}{2},2+\theta}(Q^{2\eps}_\TT)$.

{\bf Step 3.} We observe that the equations in \eqref{2layer-pbm-f} are uncoupled in their principal part, hence we can treat them as separate equations once we know that $f_a,f_s\in C^{\frac{\theta}{2},\theta}_{\text{loc}} (Q_\TT)$, thanks to Step 1. From Step 2 applied to both equations, we finally deduce that there exists $\theta\in(0,1)$ such that
\begin{equation*}
    T_a,T_s\in C^{1+\frac{\theta}{2},2+\theta}_{\text{loc}}(Q_\TT).
\end{equation*}
The proof is thus complete.
\end{proof}


\section{Maximum and comparison principles for cooperative degenerate parabolic systems} \label{sec:MP}

\noindent Comparison principles will be the key to prove many properties of the solutions of \eqref{2layer-pbm}, \eqref{eq:bound and in cond}. In particular, we will exploit such properties to prove that solutions to \eqref{2layer-pbm}, \eqref{eq:bound and in cond} are globally defined for $\varepsilon _a \in (0,2)$, and that finite-time blow-up occurs when $\varepsilon_a>2$.

\subsection{A weak maximum principle for linear cooperative degenerate parabolic systems\vspace{.2cm}} \hfill

The aim of this section is to establish a weak maximum principle for linear cooperative degenerate parabolic systems of the form
\begin{equation}
\label{eq-ppmax-syst}
\begin{cases}
u_t - k((1-x^2) u_x)_x = a(t,x) u + b(t,x) v + f, \\
v_t - k' ((1-x^2) v_x)_x = c(t,x) u + d(t,x) v + g, \\
\left( (1-x^2) u_x \right) _{\vert x=\pm1} = 0 , \qquad \left( (1-x^2) v_x \right) _{\vert x=\pm1} = 0 , \\
u(0,x)=u_0(x), \qquad 
v(0,x)=v_0(x), \end{cases}
\end{equation}
where $f,g \in L^2 ((0,\TT)\times I)$, $a,b,c,d \in L^\infty ((0,\TT)\times I)$, and $k,k'>0$.

\begin{Proposition}
\label{lem-ppmaxw-syst}
Let $\TT>0$, $u_0, v_0\in V$, $f,g \in L^2 ((0,\TT)\times I)$, $a,b,c,d \in L^\infty ((0,\TT)\times I)$, and let $k,k'>0$. Let $u,v \in H^1(0,\TT;H)\cap L^2(0,\TT,D(A)) $ solve system \eqref{eq-ppmax-syst} a.e.. Assume that 
\begin{equation} 
\label{eq:cooperative}
    b, c \geq 0\,\,\,\, \text{a.e. in } Q_\TT=(0,\TT)\times I.
\end{equation}
Then,
\begin{equation}
\label{ppmax-eq-bis}
f,g \geq 0\,\,\,\, \text{a.e. in } Q_\TT \quad \text{ and } \quad u_0,v_0 \geq 0\,\,\,\, \text{in } I \quad  \implies  \quad  u,v \geq 0 \text{ in } Q_\TT.
\end{equation}
\end{Proposition}
\begin{Remark}
    {\rm Note that assumption \eqref{eq:cooperative} makes the spatially homogeneous problem associated to \eqref{eq-ppmax-syst} a \emph{cooperative} system, in the sense of \cite[pp. 32-34]{Smith}.
    }
\end{Remark}
To prove Proposition \ref{lem-ppmaxw-syst}, we first establish a weak maximum principle for a single degenerate parabolic equation. We begin by recalling a technical lemma concerning the positive and negative parts of functions in the space $V$, which will be used repeatedly in the sequel.
\begin{Lemma}
\label{lem6.1TV}
Let $w \in V$. Then, for all $M\in\mathbb{R}$, the functions
$$(w-M)^+:= \sup (w-M,0) \text{ and } (w+M)^-:= \sup (-(w+M),0)$$ 
belong to $V$. Moreover, for a.e. $x \in I$,
\begin{equation}
\label{65}
((w-M)^+)_x (x) =
\begin{cases}
w_x (x) & (w-M)(x) > 0\\
0 & (w-M)(x) \leq 0,
\end{cases}
\end{equation}
and 
\begin{equation}
\label{66}
((w+M)^-)_x(x) =
\begin{cases}
0 & (w+M)(x) \leq 0\\
-w_x (x) & (w+M)(x) < 0.
\end{cases}
\end{equation}
\end{Lemma}
We refer to \cite[Lemma 6.1]{TORT2012683} (see also \cite[page 292]{Evans}) for the proof of Lemma \ref{lem6.1TV}. 

We now state the weak maximum principle for a single degenerate parabolic equation.
\begin{Lemma}
\label{lem-ppmaxw}
Let $\TT>0$, $u_0\in V$, $h\in L^\infty(Q_\TT)$, $f\in L^2(Q_\TT)$, and $k>0$. Let $u\in H^1(0,\TT;H)\cap L^2(0,\TT;D(A))$ solve a.e. the following problem
\begin{equation}
\label{eq-ppmax}
\begin{cases}
u_t - k ((1-x^2) u_x)_x + h(t,x) u = f(t,x) , \\
\left( (1-x^2) u_x \right) _{\vert x=\pm1} = 0 , \\
u(0,x)=u_0(x). \end{cases}
\end{equation}
Then, we have
\begin{equation}
\label{ppmax-eq}
f \geq 0\,\,\,\, \text{a.e. in $Q_\TT$} \quad \text{ and } \quad u_0 \geq 0 \text{ in $I$} \quad \implies \quad  u \geq 0\,\,\,\, \text{a.e. in $Q_\TT$}.
\end{equation}
\end{Lemma}
\begin{proof}
Using Lemma \ref{lem6.1TV}, we observe that
$$ \frac{d}{dt} \left( (u^-) ^2 \right) = -2 u_t \, u^- .$$
Multiplying equation \eqref{eq-ppmax} by $-u^-$ and integrating over $I$, we obtain
\begin{multline*} 
\int _{-1} ^1 -f \, u^-\,dx 
= \int _{-1} ^1 -u_t \, u^-\,dx  + k \int _{-1} ^1 ((1-x^2) u_x)_x \, u^-\,dx
- \int _{-1} ^1 h u u^-\,dx
\\
= \frac{1}{2} \int _{-1} ^1 \frac{d}{dt} \left( (u^-) ^2 \right)dx
+ k [ (1-x^2) u_x \, u^- ]_{-1} ^1 - k \int _{-1} ^1 (1-x^2) u_x \, (u^-) _x\,dx
+ \int _{-1} ^1 h (u^- ) ^2dx
\\
= \frac{d}{dt} \left(\frac{1}{2} \int _{-1} ^1 (u^-) ^2 dx\right)
+ k \int _{-1} ^1 (1-x^2) u_x ^2 \, \chi _{u \leq 0}\,dx + \int _{-1} ^1 h (u^- ) ^2dx .
\end{multline*}
Hence
$$ 
\frac{d}{dt} \left(\frac{1}{2} \int _{-1} ^1 (u^-) ^2dx \right)
+ \int _{-1} ^1 h (u^-) ^2dx
\leq  \int _{-1} ^1 -f \, u^- \,dx
\leq 0
$$
since $f \geq 0$.

Recalling that $h$ is bounded, there exists $m\geq 0$ such that $-m\leq h$ for all $(t,x)\in Q_\TT$. Hence, 
$$ 
\frac{d}{dt} \left(\frac{1}{2} \int _{-1} ^1 (u^-) ^2\,dx \right)
\leq - \int _{-1} ^1 h (u^-) ^2\,dx \leq m \int _{-1} ^1 (u^-) ^2\,dx.$$
By Gronwall's lemma and the assumption $u_0\ge0$, we conclude that $u^- =0$, and therefore $u\geq 0$ in $Q_\TT$. 
\end{proof}
As a direct consequence of the weak maximum principle established above, we obtain the following comparison result.
\begin{Corollary}
\label{cor-ppmax-eq}
Let $\TT>0$, $u_0, v_0 \in V$, $h\in L^\infty(Q_\TT)$, $f,g\in L^2(Q_\TT)$, and $k >0$. Let $u,v\in H^1(0,\TT;H)\cap L^2(0,\TT;D(A))$ solve a.e. the following problems
$$ \begin{cases}
u_t - k ((1-x^2) u_x)_x +hu= f , \\
\left( (1-x^2) u_x \right) _{\vert x=\pm1} =0, \\
u(0,x)=u_0(x) \end{cases}
\quad \text{ and } \quad 
\begin{cases}
v_t - k ((1-x^2) v_x)_x +hv= g  , \\
\left( (1-x^2) v_x \right) _{\vert x=\pm1} =0 , \\
v(0,x)=v_0(x) . \end{cases}
$$
Then,
\begin{equation}
\label{ppmax-eq-comp}
f \geq g\,\,\,\, \text{a.e. in $Q_\TT$} \quad \text{ and } \quad u_0 \geq v_0 \text{ in $I$} \quad \implies \quad  u \geq v\,\,\,\, \text{a.e. in $Q_\TT$}.
\end{equation}
\end{Corollary}
\begin{proof}
It suffices to apply Lemma \ref{lem-ppmaxw} to the difference $w:=u-v$. 
\end{proof}
We now proceed with the proof of the weak maximum principle for a system of linear cooperative degenerate equations.

\begin{proof}[Proof of Proposition \ref{lem-ppmaxw-syst}]
We multiply the first equation of \eqref{eq-ppmax-syst} by $-u ^-$ and integrate over $I$:
\begin{equation}
\label{comp-da-u}
\int _{-1} ^1 - \frac{\partial u}{\partial t} \, u ^-\,dx
+ k \int _{-1} ^1 \frac{\partial }{\partial x} \left( (1-x^2) \frac{\partial  u}{\partial x}\right) \, u^-\,dx
= \int _{-1} ^1 (-au \, u^- - bv \, u ^- - f \, u ^-)dx .
\end{equation}
Observe that
$$ \int _{-1} ^1 - \frac{\partial u}{\partial t} \, u ^-\,dx = \frac{d}{dt} \left( \frac{1}{2} \int _{-1} ^1 ( u ^- )^2\,dx \right),$$
and
\begin{multline*}
\int _{-1} ^1 \frac{\partial }{\partial x} \left( (1-x^2) \frac{\partial  u}{\partial x}\right) \, u ^-\,dx
= [ (1-x^2) \frac{\partial  u}{\partial x} \, u ^- ]_{-1} ^1 
- \int _{-1} ^1 (1-x^2) \frac{\partial  u}{\partial x}  \, \frac{\partial  u ^-}{\partial x}\, dx
\\
=  \int _{-1} ^1 (1-x^2) \left( \frac{\partial  u }{\partial x} \right) ^2 \, \chi _{u \leq 0 }\,dx.
\end{multline*}
Note that the right-hand side is nonnegative. Using the decomposition $u = u ^+ - u ^-$, we obtain
$$ \int _{-1} ^1 -a u \, u ^-\,dx = \int _{-1} ^1 a \,( u ^-)^2\,dx .$$
Moreover, since $b \geq 0$ and $v = v ^+ - v ^-$,
$$ - \int _{-1} ^1 bv \, u ^-\,dx
= - \int _{-1} ^1 b v ^+ \, u ^- + \int _{-1} ^1 b v ^- \, u ^-\,dx
\leq  \int _{-1} ^1 b v ^- \, u ^-\,dx, $$
since $b v ^+ \, u ^-$ is nonnegative.

Therefore, from \eqref{comp-da-u} we deduce
\begin{equation*}
\frac{d}{dt} \left( \frac{1}{2} \int _{-1} ^1 ( u ^- )^2 \,dx\right)
\leq \int _{-1} ^1 a \, (u ^-)^2\,dx + \int _{-1} ^1 b \, v ^- \, u ^-\,dx .
\end{equation*}
We now multiply the second equation of \eqref{eq-ppmax-syst} by $-v ^-$. Arguing analogously, we obtain
\begin{equation*}
\frac{d}{dt} \left( \frac{1}{2} \int _{-1} ^1 ( v ^- )^2\,dx \right)
\leq \int _{-1} ^1 d \, (v ^-)^2\,dx + \int _{-1} ^1 c \, v ^- \, u ^-\,dx .
\end{equation*}
Since $a,b,c,d$ are bounded, there exists $L>0$ such that
$$ \frac{d}{dt} \left( \frac{1}{2} \int _{-1} ^1 (( u ^- )^2 + ( v ^- )^2)dx \right)
\leq L \int _{-1} ^1  ((u ^-)^2 + (v ^-)^2)dx . $$
Applying Gronwall's lemma and using the assumption $u_0,\,v_0\geq0$, we conclude that
$$ \forall\, t\geq 0, \quad \frac{1}{2} \int _{-1} ^1 (( u ^- )^2 + ( v ^- )^2)dx = 0 ,$$
and therefore $u ^- = 0 = v ^-$, i.e., $u,\,v\geq0$ in $Q_\TT$. 
\end{proof}


\subsection{A strong maximum principle for linear cooperative degenerate parabolic systems\vspace{.2cm}} \hfill

In this section we state and prove a strong maximum principle for a system of linear degenerate equations. We introduce the operators
$$ P_1 u := u_t - k_1((1-x^2) u_x)_x , \quad P_2 u := u_t - k_2((1-x^2) u_x)_x.$$
\begin{Proposition}
\label{prop-ppmax-system-deg}
Let $\TT>0$ and set $Q_\TT=(0,\TT)\times I$. Let $k_1,k_2>0$ and $a,b,c,d \in L^\infty (Q_\TT)$ be such that
\begin{equation}
\label{eq:cooperative-max}
    b, c \geq 0\,\,\,\, \text{a.e. in } Q_\TT.
\end{equation}
Let $u,v \in C(\overline{Q_\TT}) \cap C^{1,2} (Q_\TT)$ satisfy
\begin{equation}
\label{eq:syst-ineq}
    \begin{cases}
        P_1 u \leq a(t,x) u + b(t,x) v  \\
        P_2 v \leq c(t,x) u + d(t,x) v
        
    \end{cases} 
 \quad \text{ in } Q_\TT ,
 \end{equation}
and
$$u, v \leq 0\quad\text{in }Q_\TT.$$
If there exists $(t_1,x_1) \in Q_\TT$ such that $u(t_1,x_1)=0$, then
$$ \forall\, x \in I,\, \forall\, t\in (0,t_1], \quad u (t,x)=0 .$$
Analogously, if there exists $(t_2,x_2) \in Q_\TT$ such that $v(t_2,x_2)=0$, then
$$ \forall\, x \in I,\, \forall\, t\in (0,t_2], \quad v(t,x)=0 .$$
\end{Proposition}
\begin{Remark}
{\rm
Observe that assumption \eqref{eq:cooperative-max} implies that the spatially homogeneous problem associated to \eqref{eq:syst-ineq} is a \emph{cooperative} system (see the definition in \cite[pp. 32-34]{Smith}).
}
\end{Remark}
Before proving Proposition \ref{prop-ppmax-system-deg}, we state the corresponding maximum principle for a single degenerate parabolic equation (see \cite{Protter-Weinberger}, chapter 3, section 2).
\begin{Proposition}
\label{prop-ppmax-1eq-deg}
Let $\TT>0$, let $d\in C^1(I)$ be strictly positive, and let $h: Q_\TT \to \mathbb R$ be a nonnegative, locally bounded function. Assume that $u \in C(\overline{Q_\TT}) \cap C^{1,2} (Q_\TT)$ satisfies
$$ Pu := u_t - (d(x)) u_x)_x + h(t,x) u \leq 0 \quad \text{ in } Q_\TT .$$
Assume $M:=\max _{\overline{Q_\TT}} u \geq 0$ and suppose that there exists $(t_0,x_0) \in Q_\TT$ such that $u(t_0,x_0)=M$.
Then,
$$ \forall\, t\in (0,t_0],\, \forall\,  x \in I, \quad u (t,x)=M .$$
\end{Proposition}
\begin{proof}
The result follows from \cite[Theorem 4 p. 172]{Protter-Weinberger}. Although such a result is stated for uniformly parabolic operators, it applies here on compact subsets $[-1+\delta ,1-\delta]$ for any $0<\delta<1$. For completeness, we briefly sketch the argument, which consists of two steps.
\begin{itemize}
\item We first show that $u(t_0,x_0)=M$ implies that $u(t, x_0)=M$ for all $t\in (0,t_0]$. We follow the strategy of \cite[Lemma 3 pages 167-168]{Protter-Weinberger} or \cite[Lemma 2.5]{Yang-Deng}. Assume by contradiction that there exists $t_1 \in (0,t_0)$ such that $u(t_1, x_0) <M$. Define $\omega: [t_1,t_0] \times [x_0-r, x_0+r] \to \mathbb R$ as follows
$$ \omega (t,x) := \left( u(t,x) - M \right)
+ \left( M-u(t_1,x_0) \right) \left( r^2 - (x-x_0)^2 \right)^2 e^{-\alpha (t-t_1) } , 
$$
for $\alpha>0$, and $r>0$ is small enough so that $u(t_1,x)<M$ on $[x_0-r,x_0+r]$). Then, $\omega(t_0,x_0)>0$, while one can show that
$$ \begin{cases}
P \omega \leq 0 \text{ on } [t_1,t_0] \times [x_0-r, x_0+r] , \\
\omega \leq 0 \text{ on } t=t_1, \text{ on } x=x_0-r, \text{ and on } x=x_0+r .
\end{cases} $$
The weak maximum principle yields $\omega \leq 0$ on $[t_1,t_0] \times [x_0-r, x_0+r] $. In particular, $\omega(t_0,x_0)\leq 0$, a contradiction.

\item Next, we show that $u(t_0,x_0)=M$ implies $u(t_0,x)=M$ for all $x\in I$. This follows from \cite[Lemmas 1 and 2, pages 164-167]{Protter-Weinberger}, whose proof works on compact subset of $Q_\TT$ and allows for degeneracy of $d$ at the boundary.
\end{itemize}
Combining the two steps concludes the proof.
\end{proof}
We can now prove Proposition \ref{prop-ppmax-system-deg}.
\begin{proof}[Proof of Proposition \ref{prop-ppmax-system-deg}]
Assume that $u(t_1,x_1)=0$ for some $(t_1,x_1)\in Q_\TT$. Then
$$ \begin{cases}
P_1 u - a(t,x) u \leq b(t,x) v \leq 0  \text{ in $Q_\TT$}, \\
u \leq 0  \text{ in $Q_\TT$} , \\
u (t_1,x_1) = 0 .
\end{cases} $$
Let $\omega \geq 0$ and define $ U (t,x) := e^{-\omega t} u (t,x) $.
Then
$$ P_1 u = \omega e^{\omega t} U + e^{\omega t} P_1 U,$$
and a direct computation yields
$$\begin{cases}
 P_1 U + (\omega - a(t,x)) U \leq 0 \text{ in $Q_\TT$}, \\
 U \leq 0  \text{ in $Q_\TT$} , \\
U (t_1,x_1) = 0 .
\end{cases} $$
Choosing $\omega $ sufficiently large so that $\omega - a(t,x) \geq 0$, we can apply Proposition \ref{prop-ppmax-1eq-deg} and conclude that
$$U=0\quad\text{in }(0,t_1] \times I.$$
Hence $u=0$ in $(0,t_1] \times I$. 

The argument for $v$ is identical.
\end{proof}

\subsection{Comparison principles for the 2-layer PDE system\vspace{.2cm}} \hfill
\label{sec:comparison}

In this section, we state and prove a comparison principle for the 2-layer EBM. \emph{We recall that whenever an inequality between vectors is written, it is always understood component-wise.}
\begin{Theorem}
\label{lem-comp-2layer}{\rm \bf (Comparison principle)}
Let $\bTi , \tildebTi \in {\bf V}$ and let $\bT,\,\tildebT$ be the corresponding maximal solution of \eqref{2layers-abstract-form}, defined on $I(\bTi)=[0, \TT^\star (\bTi))$ and $I(\tildebTi)=[0, \TT^\star (\tildebTi))$, respectively. Then, 
\begin{equation}
\bTi \leq \tildebTi  
 \quad \implies \quad \bT (t) \leq \tildebT (t)\quad \text{ for all } t \in [0, \TT^\star (\bTi) ) \cap  [0, \TT^\star (\tildebTi) ).
\end{equation}
\end{Theorem}

\begin{Remark}
\rm
The cooperative structure of the system, deeply used in the proof of Theorem \ref{lem-comp-2layer}, originates from the vertical energy exchange between the two layers. Such exchanges reduce the temperature difference between the atmosphere and the surface at each latitude and, due to their dissipative nature, are associated with entropy production \cite{Lucarini2014}.

From a physical viewpoint, cooperation also manifests itself in terms of heat-transport compensation mechanisms. For instance, a positive anomaly in the oceanic meridional heat transport induces a surface temperature increase at high latitudes. Through vertical coupling, this anomaly is partially transferred to the atmosphere, thereby reducing the atmospheric meridional temperature gradient and, consequently, the atmospheric heat transport. This compensation mechanism, first proposed by Bjerknes \cite{Bjerknes1964} and later generalized by Stone \cite{Stone1978}, has been shown to play a crucial role in climate dynamics \cite{Sheffrey2006,Knietzsch2015,Outten2018}.
\end{Remark}

\begin{proof}
The proof is divided into two steps.
\medskip

\noindent
\textbf{Step 1: regular initial data.} Assume first that $\bTi , \tildebTi \in \bDA$ and fix 
$$\tau < \min (\TT^\star (\bTi),\TT^\star (\tildebTi)).$$
Introduce the nonlinearities $F_a, F_s : (0,+\infty) \times I \times \mathbb R \times \mathbb R \to \mathbb R$ by
\begin{equation}\label{def-Fa}
F_a (t,x,u, v) := \frac{1}{\gamma _a} \left[ -\lambda (u - v) + \varepsilon _a \sigma _B \vert v \vert ^3 v  - 2 \varepsilon _a \sigma _B \vert u \vert ^3 u + 
r(t) q(x) \beta_a(u)\right]
,
\end{equation}
and
\begin{equation}\label{def-F0}
F_s (t,x,u, v) := \frac{1}{\gamma _s} \left[ -\lambda (v - u) - \sigma _B \vert v \vert ^3 v  + \varepsilon _a \sigma _B \vert u \vert ^3 u +
r(t) q(x) \beta_s(v)
  \right] .
\end{equation}
Hence, $\bT$ and $\tildebT$ satisfy the following problem
\begin{equation}
\label{2layer-pbm-comp-abs}
\begin{cases}
 \displaystyle  \frac{\partial u}{\partial t} -  \kappa_a \frac{\partial }{\partial x} \left( (1-x^2) \frac{\partial u}{\partial x}\right) = F_a (t,x,u , v ), \vspace{.15cm}\\
\displaystyle \frac{\partial v}{\partial t} -  \kappa_s \frac{\partial }{\partial x} \left( (1-x^2) \frac{\partial v}{\partial x}\right)  = F_s (t,x,u , v ) , \vspace{.15cm}\\
\displaystyle(1-x^2) \frac{\partial u}{\partial x} _{\vert x = \pm 1} = 0  = (1-x^2) \frac{\partial v}{\partial x} _{\vert x = \pm 1} ,\vspace{.15cm}\\
u (0,x)=u ^{0} (x) ,\quad v (0,x)=v ^{0} (x) ,
\end{cases}
\end{equation}
with initial condition $\bTi$ and $\tildebTi$, respectively.

Consider the differences
$$ D_a := \tilde T_a - T_a , \quad D_s := \tilde T_s - T_s.$$
Then, $D_a$ and $D_s$ satisfy
$$
\frac{\partial D_a}{\partial t} -  \kappa_a \frac{\partial }{\partial x} \left
(1-x^2) \frac{\partial  D_a}{\partial x}\right)
= F_a (t,x, \tilde T_a (t,x) ,  \tilde T_s (t,x)) - F_a (t,x, T_a (t,x),  T_s (t,x)) ,
$$
$$
\frac{\partial D_s}{\partial t} -  \kappa_s \frac{\partial }{\partial x} \left( (1-x^2) \frac{\partial  D_s}{\partial x}\right)
= F_s (t,x, \tilde T_a (t,x) ,  \tilde T_s (t,x)) - F_s (t,x, T_a (t,x),  T_s (t,x)).
$$
We decompose 
$$ \begin{cases}
F_a (t,x,u, v) = f_a (u,v) + g_a (t,x, u) , \\
F_s (t,x,u, v) = f_s (u,v) + g_s (t,x, u), \end{cases} $$
where
$$ f_a (u,v) := \frac{1}{\gamma _a} \left[ -\lambda (u - v) + \varepsilon _a \sigma _B \vert v \vert ^3 v  - 2 \varepsilon _a \sigma _B \vert u \vert ^3 u \right] ,\,\,\,\,  g_a (t,x, u) :=  \frac{1}{\gamma _a} r(t)  q(x)  \beta_a (u) ,$$
and
$$ f_s (u,v) := \frac{1}{\gamma _s} \left[ -\lambda (v - u) - \sigma _B \vert v \vert ^3 v  + \varepsilon _a \sigma _B \vert u \vert ^3 u \right],\quad g_s (t,x, v) :=   \frac{1}{\gamma _s} r(t) q(x) \beta_s (v) .$$
The functions $f_a$ and $f_s$ are $C^1$ on $\mathbb R^2$, while $g_a$ and $g_s$ are globally Lipschitz with respect to the third variable. We have
\begin{multline*}
F_a (t,x, \tilde T_a,  \tilde T_s) - F_a (t,x, T_a ,  T_s)
\\
= \left( f_a (\tilde T_a,  \tilde T_s) + g_a (t,x, \tilde T_a) \right) -  \left( f_a (T_a, T_s) + g_a (t,x, T_a) \right)
\\
= \left( f_a (\tilde T_a,  \tilde T_s) -  f_a (T_a, T_s) \right) +  \left( g_a (t,x, \tilde T_a) -  g_a (t,x, T_a) \right)
\\
= \left( f_a (\tilde T_a, \tilde T_s) - f_a (\tilde T_a,  T_s)\right) + \left(f_a (\tilde T_a,  T_s) - f_a (T_a ,  T_s) \right)
+ \left( g_a (t,x, \tilde T_a) - g_a (t,x, T_a) \right) .
\end{multline*}
Denote $m_a, m_s: \mathbb R^4 \to \mathbb R$ the functions
$$ m_a (T_a, T_s, \tilde T_a, \tilde T_s) := \int _0 ^1 \frac{\partial f_a}{\partial u} (T_a +(\tilde T_a - T_a)z, T_s) \, dz ,$$
$$ m_s (T_a, T_s, \tilde T_a, \tilde T_s) := \int _0 ^1 \frac{\partial f_a}{\partial v} (\tilde T_a, T_s + (\tilde T_s - T_s) y) \, dy ,$$
and let $m_a ^* : [0,+\infty) \times I  \times \mathbb R^2 \to \mathbb R$ be defined as
$$
 m_a ^* (t,x, T_a, \tilde T_a) := \begin{cases} 
 \displaystyle { \frac{r(t) q(x)}{\gamma _a} \frac{\beta _a (\tilde T_a) - \beta _a (T_a)}{\tilde T_a - T_a} }
 &\text{ if } \tilde T_a \neq T_a , \\ 0 &\text{ if } \tilde T_a = T_a . \end{cases}
$$
Thus, we get
\begin{multline*} 
 f_a (\tilde T_a, \tilde T_s) - f_a (\tilde T_a,  T_s)
= \int _{T_s} ^{\tilde T_s} \frac{\partial f_a}{\partial v} (\tilde T_a, w) \, dw
\\
= (\tilde T_s - T_s) \int _0 ^1 \frac{\partial f_a}{\partial v} (\tilde T_a, T_s + (\tilde T_s - T_s) y) \, dy 
= \left( m_s (T_a, T_s, \tilde T_a, \tilde T_s) \right) \, (\tilde T_s - T_s) ,
\end{multline*}
and moreover
\begin{multline*} 
\left(f_a (\tilde T_a,  T_s) - f_a (T_a ,  T_s) \right)
+ \left( g_a (t,x, \tilde T_a) - g_a (t,x, T_a) \right)
\\
= \left( m_a (T_a, T_s, \tilde T_a, \tilde T_s) + m_a ^* (t,x,T_a, \tilde T_a) \right) \, (\tilde T_a - T_a) .
\end{multline*}
Therefore
\begin{equation}
\label{comp-form_a}
F_a (t,x, \tilde T_a,  \tilde T_s) - F_a (t,x, T_a ,  T_s) = (m_a + m_a ^*) \, (\tilde T_a - T_a) + m_s \, (\tilde T_s - T_s),
\end{equation}
where we have dropped the arguments of $m_a,\,m_s,\,m^*_s$ for brevity.  

We now denote by $n_a, n_s: \mathbb R^4 \to \mathbb R$ the functions
$$ n_a (T_a, T_s, \tilde T_a, \tilde T_s) := \int _0 ^1 \frac{\partial f_s}{\partial u} (T_a +(\tilde T_a - T_a)z, T_s) \, dz ,$$
$$ n_s (T_a, T_s, \tilde T_a, \tilde T_s) := \int _0 ^1 \frac{\partial f_s}{\partial v} (\tilde T_a, T_s + (\tilde T_s - T_s) y) \, dy ,$$
and define $n_s ^* : [0,+\infty) \times I \times \mathbb R^2 \to \mathbb R$ as follows
$$ n_s ^* (t,x,T_s, \tilde T_s) := \begin{cases} 
\displaystyle { \frac{r(t) q(x)}{\gamma _s} \frac{\beta _s (\tilde T_s) - \beta _s (T_s)}{\tilde T_s - T_s}}
 &\text{ if } \tilde T_s \neq T_s , \\ 0 &\text{ if } \tilde T_s = T_s . \end{cases} $$
Then
\begin{equation}
\label{comp-form_s}
F_s (x,t, \tilde T_a,  \tilde T_s) - F_s (x,t, T_a ,  T_s) = n_a  \, (\tilde T_a - T_a) + ( n_s + n_s ^*) \, (\tilde T_s - T_s) .
\end{equation}

Now, using this notation, we rewrite the problems satisfied by $(D_a, D_s)$ as
\begin{equation}
\label{2layer-pbm-comp-abs-diff}
\begin{cases}
\displaystyle\frac{\partial D_a}{\partial t} -  \kappa_a \frac{\partial }{\partial x} \left( (1-x^2) \frac{\partial D_a}{\partial x}\right) = (m_a
+ m_a ^*) D_a + m_s D_s, \vspace{.2cm}\\
\displaystyle\frac{\partial D_s}{\partial t} -  \kappa_s \frac{\partial }{\partial x} \left( (1-x^2) \frac{\partial  D_s}{\partial x}\right)
= n_a D_a + (n_s + n_s ^*) D_s , \vspace{.2cm}\\
\displaystyle(1-x^2) \frac{\partial D_a}{\partial x} _{\vert x = \pm 1} = 0 = (1-x^2) \frac{\partial D_s}{\partial x} _{\vert x = \pm 1} ,\vspace{.2cm}\\
D_a (0,x)= \tilde T_a ^{0} (x) - T_a ^{0} (x) ,\quad
D_s (0,x)= \tilde T_s ^{0} (x) - T_s ^{0} (x) .
\end{cases}
\end{equation}
Since $f_a$ is nondecreasing with respect to the second argument, and $f_s$ is nondecreasing with respect to the its first argument, we have the cooperative conditions
\begin{equation}
\label{coop}
 m_s \geq 0 , \quad  n_a \geq 0 .
 \end{equation}
Moreover, we note that $m_a ^*$ and $n_s ^*$ are nonnegative and uniformly bounded due to the global Lipschitz continuity of $\beta _a$ and $\beta _s$. 

By the regularity of strict solutions, we have
$$ \bT,\,\tildebT \in C ([0,\tau]; \bDA) .$$
Therefore, using that $\bDA \subset C (\bar{I})\times C (\bar{I})$, there exist $M_a$ and $M_s$ such that  such that
$$ \forall\, t \in [0,\tau],\, \forall\, x \in I, \quad 
\begin{cases} T_a (t,x), \tilde T_a (t,x) \in [-M_a,M_a], \\
 T_s (t,x), \tilde T_s (t,x) \in [-M_s,M_s] . \end{cases} $$
Denote by $L_a$ and $L_s$ the Lipschitz constant of $ \beta _a$ and $\beta _s$, respectively. We define
\begin{multline}
\label{La0-Ta}
L_{a,s} :=\\ \sup _{u\in [-M_a,M_a], v\in [-M_s,M_s]} \left\{\left| \frac{\partial f_a}{\partial u} \right|, \left| \frac{\partial f_a}{\partial v} \right| , \left|\frac{\partial f_s}{\partial u} \right|, \left| \frac{\partial f_s}{\partial v} \right| , 
\frac{\norm{r}_\infty \norm{q}_\infty  }{\gamma _a} L_a, \frac{\norm{r}_\infty \norm{q}_\infty  }{\gamma _s} L_s\right\}.
\end{multline}
Then, $L_{a,s}$ is a uniform bound for the coefficients of \eqref{2layer-pbm-comp-abs-diff}, that is,
\begin{equation}
\label{La0-Ta-appl}
\forall\, t\in [0,\tau],\, \forall\, x\in I, \quad \vert m_a \vert , \vert m_s \vert, \vert m_a ^* \vert,
 \vert n_a \vert, \vert n_s \vert, \vert n_s ^* \vert \leq L_{a,s} .
\end{equation}
Therefore, system
\eqref{2layer-pbm-comp-abs-diff} satisfies the hypotheses of the parabolic maximum principle for cooperative systems (Proposition~\ref{lem-ppmaxw-syst}), and we conclude that
$$
D_a(t,x)\ge0, \quad D_s(t,x)\ge0
\quad \text{for all } (t,x)\in[0,\tau]\times I.
$$
Since $\tau$ is arbitrary, the comparison principle holds on $[0,\TT^\star (\bTi))\cap [0,\TT^\star (\tildebTi))$ for all regular initial data in $\bDA$. 

\medskip
\noindent
\textbf{Step 2: Initial data in $\bV$.}

We now remove the additional regularity assumption and consider $\bTi,\,\tildebTi \in \bV$.
The proof is based on a regularization and approximation argument.

\smallskip
\noindent
\emph{Approximation of the initial data.}
Since $\overline{\bDA}^{\,\bV}=\bV$, it is sufficient to construct approximating sequences in $\bDA$ which preserve the ordering of the initial conditions. This can be done by applying the semigroup as follows. We set 
$$
\bT^{0,n}:=e^{\frac{1}{n}\bA}\bTi, \qquad
\tildebT^{0,n}:=e^{\frac{1}{n}\bA}\tildebTi.
$$
Then,
$$
\bT^{0,n},\,\tildebT^{0,n}\in \bDA, \qquad
\bT^{0,n}\le\tildebT^{0,n},
$$
and
$$
\|\bT^{0,n}-\bTi\|_{\bV}\to0, \qquad
\|\tildebT^{0,n}-\tildebTi\|_{\bV}\to0
\quad \text{as } n\to\infty.
$$

\smallskip
\noindent
\emph{Passage to the limit.}
Fix
$$
\tau<\min\bigl(\TT^\star(\bTi),\TT^\star(\tildebTi)\bigr).
$$
By the continuous dependence result (Proposition~\ref{prop-dep-continue-CI}), for $n$ sufficiently large the mild solutions $\bT^{(n)}$ and $\tildebT^{(n)}$ associated with $\bT^{0,n}$ and $\tildebT^{0,n}$ are well defined on $[0,\tau]$ and satisfy
$$
\bT^{(n)}(t,x)\le\tildebT^{(n)}(t,x)
\quad \text{for all } (t,x)\in[0,\tau]\times I,
$$
thanks to Step~1.

Moreover, Proposition~\ref{prop-dep-continue-CI} yields the uniform convergence
$$
\sup_{t\in[0,\tau]}
\|\bT^{(n)}(t)-\bT(t)\|_{\bV}\to0,
\qquad
\sup_{t\in[0,\tau]}
\|\tildebT^{(n)}(t)-\tildebT(t)\|_{\bV}\to0.
$$
Passing to the limit in the above inequality, we conclude that
$$
\bT(t,x)\le\tildebT(t,x)\quad\text{for all } (t,x)\in[0,\tau]\times I.
$$

Since $\tau$ is arbitrary, the comparison principle holds on $[0,\TT^\star(\bTi))\cap[0,\TT^\star(\tildebTi))$ for all initial data in $\bV$.
\end{proof}
\begin{Remark}
    {\rm
    In the above proof we deeply exploit the cooperative nature of \eqref{2layer-pbm-comp-abs} that implies that $F_a$ is nondecreasing with respect to the fourth variable, and $F_s$ is nondecreasing with respect to the third variable. Such a property turns out to be fundamental in establishing comparison principles for systems. Observe that we cannot directly apply the results of \cite[page 130]{Smith} and \cite[Theorem 5.1 page 125]{Fife}, because of the required regularity.
    }
\end{Remark}
We now define super- and sub-solutions to \eqref{2layer-pbm}. 
\begin{Definition}
Let $\TT^+>0$. We call
\begin{equation}
\label{reg-sup-sol}
\bT^+ \in C^1([0,\TT^+ );\bH)\cap C([0,\TT^+ );\bDA) 
 \end{equation}
a \emph{super-solution} of \eqref{2layer-pbm} if for all $(t,x) \in [0,\TT^+ )\times I$ it satisfies 
\begin{equation}
\label{2layer-pbm-comp-tilde-sup}
\begin{cases}
\displaystyle\gamma _a \left[ \frac{\partial T_a ^+}{\partial t} -  \kappa_a \frac{\partial }{\partial x} \left( (1-x^2) \frac{\partial T_a ^+}{\partial x}\right) \right] \vspace{.2cm}\\
\hspace{1cm} \geq -\lambda (T_a ^+ - T_s ^+) + \varepsilon _a \sigma _B \vert T_s ^+ \vert ^3 T_s ^+  - 2 \varepsilon _a \sigma _B \vert T_a ^+ \vert ^3 T_a ^+ + r(t) q(x) \beta_a(T_a^+)  ,\vspace{.2cm}\\
\displaystyle\gamma _s \left[ \frac{\partial T_s ^+}{\partial t} -  \kappa_s \frac{\partial }{\partial x} \left( (1-x^2) \frac{\partial T_s ^+}{\partial x}\right) \right] \vspace{.2cm}\\
\hspace{1cm} \geq -\lambda (T_s ^+ - T_a ^+) - \sigma _B \vert T_s ^+ \vert ^3 T_s ^+  + \varepsilon _a \sigma _B \vert T_a ^+ \vert ^3 T_a ^+ +r(t) q(x) \beta_s(T_s^+).
\end{cases}
\end{equation}
Similarly, we call $\bT^-$ a  \emph{sub-solution} of \eqref{2layer-pbm} if it is defined on some interval $[0, \TT^-)$, it has the regularity \eqref{reg-sup-sol}, and it satisfies \eqref{2layer-pbm-comp-tilde-sup} with reverse inequalities.
\end{Definition}
The following comparison principle for super- sub-solutions of \eqref{2layer-pbm} holds.
\begin{Theorem} {\rm \bf (Comparison principles for super and sub-solutions)}
\label{lem-comp-2layer-sub-sup}
Let $\,\bTi \in \bDA_+ $ and let $\bT$ be the associated strict solution of \eqref{2layers-abstract-form}, maximally defined on $I(\bTi)$. 

If $\bT ^+$ is a super-solution of \eqref{2layer-pbm} defined on $[0,\TT^+)$, then 
\begin{equation}
\bT ^{0} \leq \bT ^+ (0)\,\, \implies \,\, \bT (t,x) \leq \bT ^+  (t,x)\quad\text{for all } (t,x) \in (I(\bTi) \cap [0,\TT^+))\times I .
\end{equation}
Similarly, if $\bT ^-$ is a sub-solution of \eqref{2layer-pbm} defined on $[0,\TT^-)$, then 
\begin{equation}
\bT^{0} \geq \bT^- (0) \,\, \implies \,\,\bT(t,x) \geq \bT^-  (t,x)\quad\text{ for all } t \in (I(\bTi) \cap [0,\TT^-))\times I .
\end{equation}
\end{Theorem}
\begin{proof}
The reasoning is similar to the proof of Theorem \ref{lem-comp-2layer}.
\end{proof}

\section{Proof of well-posedness}
\label{sec-proof-7} 

The aim of this section is: first, to establish the positivity of solutions to \eqref{2layer-pbm}, \eqref{eq:bound and in cond}; second, to prove the existence of an invariant rectangle, which in turn implies global existence for $\varepsilon_a\in(0,2)$ (proof of Theorem~\ref{prop-bounds}) and finite-time blow-up for $\varepsilon_a> 2$ (proof of Theorem~\ref{prop-blowup}).
The proofs rely crucially on the maximum principle and the comparison results developed in Section~\ref{sec:MP}.

The main idea consists in constructing suitable super- and sub-solutions of the PDE system that are independent of the spatial variable $x$. These are obtained by considering solutions of the ODE system associated with \eqref{2layer-pbm}, in which the incoming solar flux $r(t)q(x)$ is replaced by a constant forcing term. The qualitative properties of this ODE system have been thoroughly analyzed in \cite{CLMUV-EDO}.

By exploiting the comparison principle stated in Theorem~\ref{lem-comp-2layer-sub-sup}, we transfer the information obtained at the ODE level to the full PDE system \eqref{2layer-pbm}, thereby deriving bounds and qualitative properties of its solutions. For the reader’s convenience, we begin by recalling some key results from \cite{CLMUV-EDO}.

Consider the following ODE system
\begin{equation}
\label{2layer-CLMUV-EDO}
\begin{cases}
\gamma _a  T_a '(t) = -\lambda (T_a - T_s) + \varepsilon _a \sigma _B \vert T_s \vert ^3 T_s  - 2 \varepsilon _a \sigma _B \vert T_a \vert ^3 T_a  
+   \bar q \beta _a (T_a)  ,  & t>0,\\
\gamma _s T' _s (t) = -\lambda (T_s - T_a) - \sigma _B \vert T_s \vert ^3 T_s  + \varepsilon _a \sigma _B \vert T_a \vert ^3 T_a +  \bar q \beta _s (T_s)  , & t>0, \\
T_a (0)= \bar T_a ^{ 0} ,\\
T_s (0)= \bar T_s ^{ 0} ,
\end{cases}
\end{equation}
subject to the following set of assumptions 
\begin{Hyp.}\label{HYP2}\hfill
\begin{itemize}
\item[{\bf (i)}]  The coefficients $\sigma_B,\,\eps_a,\,\lambda$ satisfy
$$ \gamma_a, \gamma_s >0, \qquad  
\sigma_B>0, \qquad \varepsilon_a>0, \qquad \lambda \geq 0.$$
\item[{\bf(ii)}] The coalbedo functions $\beta_a, \beta _s:\RR\to\RR$ are  globally Lipschitz continuous and moreover $\beta_a \geq 0$, $\beta_s >0$.
\item[{\bf(iii)}] The insolation coefficient 
$\bar q \in \mathbb R$ verifies 
\begin{equation*}
\bar q > 0.
\end{equation*}
\end{itemize}
\end{Hyp.}

Observe that the ODE system \eqref{2layer-CLMUV-EDO} can be regarded as a particular spatially homogeneous case of the PDE problem \eqref{2layer-pbm}. Moreover, the set of assumptions Hyp. \ref{HYP2} is compatible with Hyp. \ref{HYP1} if we further assume $\beta_a \geq 0$, $\beta_s >0$ in Hyp. \ref{HYP1}. 

In \cite{CLMUV-EDO}, the qualitative behavior of solutions to the ODE system \eqref{2layer-CLMUV-EDO} has been thoroughly analyzed. In particular, the following results were established:
\begin{itemize}
\item global existence, boundedness, positivity, and the existence of an invariant rectangle for $\varepsilon_a \in (0,2)$;
\item the occurrence of finite-time blow-up when $\varepsilon_a >2$. 
\end{itemize}

We now recall these results more precisely.
\begin{Proposition}\label{Prop-II.1-chaos}\cite[Proposition 2.1 : Global existence, boundedness and positivity]{CLMUV-EDO}
Assume that Hyp \ref{HYP2} holds. Let $\varepsilon_a \in (0,2)$ and $\bar T_a ^{ 0}, \bar T_s ^{ 0} \geq 0$. 
Then problem \eqref{2layer-CLMUV-EDO} admits a unique solution, which is globally defined for all positive times and remains bounded on $[0, +\infty)$. Moreover, 
$$T_a(t)>0 \ \text{ and } \ T_s(t)>0,\quad \forall\,t >0.$$
\end{Proposition} 
\begin{Lemma} \label{Lemma-IV.2-chaos}
\cite[Lemma 4.2: Invariant rectangle]{CLMUV-EDO} 
\\
Assume that Hyp \ref{HYP2} holds. Let $\varepsilon_a \in (0,2)$ and $\bar T_a ^{ 0}, \bar T_s ^{ 0} \geq 0$. Let $\mu \in (\varepsilon_a^{1/4},2^{1/4})$. There exists a sufficiently large $M_a$ so that, for any $M\geq M_a$, if
$$ (\bar T_a ^{ 0} , \bar T_s ^{ 0}) \in [0, M] \times [0, \mu M],$$ 
then the solution $(T_a(t), T_s(t))$ of \eqref{2layer-CLMUV-EDO} satisfies
$$ ( T_a (t) ,  T_s (t)) \in [0, M] \times [0, \mu M], \qquad \forall\, t>0.$$
\end{Lemma}
On the other hand, when $\varepsilon_a > 2$, finite-time blow-up may occur.
\begin{Proposition}\label{prop.II.6.b-chaos}
\cite[Proposition 2.6 (b): Blow-up in finite time]{CLMUV-EDO}
Assume that Hyp \ref{HYP2} holds and let $\varepsilon _a >2$. Then:
\begin{itemize}
\item[(a)] there exist initial conditions $\bar T_a ^{ 0}, \bar T_s ^{ 0} \geq 0$ such that the associated solution  $( T_a (t) ,  T_s (t))$  of \eqref{2layer-CLMUV-EDO} blows up in finite time;
\item[(b)] if $\lambda=0$ and $\beta_a \equiv 0$, then for all initial conditions $\bar T_a ^{ 0}, \bar T_s ^{ 0} \geq 0$, the corresponding solution of \eqref{2layer-CLMUV-EDO} blows up in finite time. 
\end{itemize}
\end{Proposition}
Finally, in the subcritical regime and in the absence of atmospheric albedo feedback, the long-time behavior is fully characterized.
\begin{Proposition}\label{prop.II.3-chaos}
\cite[Propositions 2.2 and 2.3: Equilibrium points and convergence]{CLMUV-EDO}
Assume that Hyp \ref{HYP2} holds and $\beta _a =0$. Let $\varepsilon _a \in (0,2)$. Then problem \eqref{2layer-CLMUV-EDO} admits a finite number of equilibrium points. Among them, there exists a distinguished equilibrium $(T_a^{\mathrm{max}}, T_s^{\mathrm{max}})$, referred to as the \emph{warmest} equilibrium, such that for any equilibrium point $(T_a^{\mathrm{stat}}, T_s^{\mathrm{stat}})$ one has
$$T_a ^{\text{stat}} \leq T_a ^{\text{max}}\quad \text{and}\quad T_s ^{\text{stat}} \leq T_s ^{\text{max}}.$$
Moreover, every solution converges to an equilibrium point as $t \to +\infty$.
\end{Proposition}

\begin{Remark}\label{rmk:warmest-coldest}
    Note that in \cite{CLMUV-EDO} we did not introduce the concept of \lq\lq warmest" equilibrium point. Nevertheless, we proved in \cite[Proposition 2.3]{CLMUV-EDO} that there is a finite number of equilibrium points. If we consider the set of equilibrium points $(T^+_a,T^+_s)$ with highest $T^+_s$ among all equibria, the equation $T'_a(t)=0$ implies
    \begin{equation*}
        \lambda T_a^++ 2\eps_a\sigma_B (T^+_a)^4=\lambda T_s^++ \eps_a\sigma_B (T^+_s)^4.
    \end{equation*}
    Since the maps
    \begin{equation*}
        x\mapsto \lambda x+2\eps_a\sigma_B x^3\qquad x\mapsto \lambda x+\eps_a\sigma_B x^3
    \end{equation*}
    are increasing, we deduce that there exists a unique pair maximizing both $T^+_a$ and $T^+_s$. We call such a point the warmest equilibrium.

    By the same argument the existence of the \lq\lq coldest" equilibrium follows.
\end{Remark}
We emphasize that the preliminary analysis of the ODE system \eqref{2layer-CLMUV-EDO} carried out in \cite{CLMUV-EDO} plays a crucial role in the study of the 2-layer PDE model \eqref{2layer-pbm}-\eqref{eq:bound and in cond}. In both regimes $\varepsilon_a \in (0,2)$ and $\varepsilon_a > 2$, we will systematically exploit the results of \cite{CLMUV-EDO} to derive analogous properties for the PDE model by means of the comparison principle stated in Theorem~\ref{lem-comp-2layer-sub-sup}.

We now show how to produce super- and sub-solution to \eqref{2layer-pbm} as solutions of the ODE system \eqref{2layer-CLMUV-EDO}. Recalling that $q \in L^\infty (I)$ and $r \in L^\infty (0,+\infty)$, we set
$$\begin{array}{l}
\displaystyle{0} \leq q_{\text{min}} := \inf _{x\in I} \ q (x) \leq   \sup _{x\in I}  \ q (x) =: q_{\text{max}} < +\infty,  \\
\displaystyle{0} < r_{\text{min}} := \inf _{t \in  (0, +\infty)} \ r (t)  \leq   \sup _{t \in (0, +\infty)} \ r (t) =: r_{\text{max}} < +\infty. 
\end{array}$$
We define the constants
\begin{equation}\label{defqminmax}
 \bar q_{\rm min} := r_{\text{min}} q_{\text{min}} \geq 0 \qquad \text{ and } \qquad  \bar q_{\rm max} := r_{\text{max}} q_{\text{max}}  \geq 0 . 
 \end{equation}
For nonnegative initial conditions $\bTi \in \bDA_+$ we define (recall that $\bTi$ is continuous by Proposition \ref{prop-reg-borne}) 
\begin{equation}\label{defCIminmax}
0  \leq \bT ^{\text{min}, 0} := \inf _{x\in I} \bTi (x) \leq \sup _{x\in I} \bTi (x) =: \bT^{\text{max}, 0}   < +\infty.
\end{equation}
Let us consider the two ODE systems
\begin{equation}
\label{2layer-pbm-qmin}
\begin{cases}
\gamma _a  T_a '(t) = -\lambda (T_a - T_s) + \varepsilon _a \sigma _B \vert T_s \vert ^3 T_s  - 2 \varepsilon _a \sigma _B \vert T_a \vert ^3 T_a  
+\bar q_{\rm min} \beta _a (T_a)  ,  \\
\gamma _s T' _s (t) = -\lambda (T_s - T_a) - \sigma _B \vert T_s \vert ^3 T_s  + \varepsilon _a \sigma _B \vert T_a \vert ^3 T_a + \bar q_{\rm min}  \beta _s (T_s)  , \\
T_a (0)= T_a ^{\text{min}, 0} ,\\
T_s (0)= T_s ^{\text{min}, 0} ,
\end{cases}
\end{equation}
and
\begin{equation}
\label{2layer-pbm-qmax}
\begin{cases}
\gamma _a  T_a '(t) = -\lambda (T_a - T_s) + \varepsilon _a \sigma _B \vert T_s \vert ^3 T_s  - 2 \varepsilon _a \sigma _B \vert T_a \vert ^3 T_a 
+ \bar q_{\rm max} \beta _a (T_a)  ,  \\
\gamma _s T' _s (t) = -\lambda (T_s - T_a) - \sigma _B \vert T_s \vert ^3 T_s  + \varepsilon _a \sigma _B \vert T_a \vert ^3 T_a + \bar q_{\rm max} \beta _s (T_s)  , \\
T_a (0)= T_a ^{\text{max}, 0} ,\\
T_s (0)= T_s ^{\text{max}, 0} .
\end{cases}
\end{equation}
By construction, the solution of \eqref{2layer-pbm-qmin} is a sub-solution of \eqref{2layer-pbm} whereas the solution of 
\eqref{2layer-pbm-qmax} is a super-solution. Furthermore, if $\bar q_{\rm min},\bar q_{\rm max}>0$ the solutions of \eqref{2layer-pbm-qmin} and \eqref{2layer-pbm-qmax} satisfy Proposition \ref{Prop-II.1-chaos}, Lemma \ref{Lemma-IV.2-chaos} and Proposition \ref{prop.II.6.b-chaos}.

\subsection{Positivity of solutions to the 2-layer EBM\vspace{.2cm}} \hfill

We show that nonnegative initial conditions generate positive solutions of the 2-layer Energy Balance Model \eqref{2layer-pbm}-\eqref{eq:bound and in cond}. This property is essential for the physical consistency of the model, since the unknowns $T_a$ and $T_s$ represent temperatures measured in Kelvin degrees.
\begin{Proposition}{\rm \bf (Nonnegativity/positivity of the solutions)}\hfill
\label{prop-nonneg}

Let $\varepsilon_a >0$ and let $\bTi \in \bV_+$. Denote by $\bT$ the mild solution of \eqref{2layers-abstract-form}, defined on $ [0, \TT ^\star (\bTi)) $. Assume that $\beta_a \geq 0$, $\beta_s >0$. Then, the following statements hold.
\begin{itemize}
\item[(i)] Under Hyp. \ref{HYP1}, one has
 \begin{equation}
\label{eq-nonneg}
 \bT (t,x)  \geq \pmb{0}\quad
\text{ for all } (t,x) \in [0, \TT ^\star (\bTi))\times I.
\end{equation}

\item[(ii)] If, in addition, $q_{\text{min}} > 0$, then
 \begin{equation}
\label{eq-posit}
 \quad  \bT (t,x) > \pmb{0}\quad
\text{ for all } (t,x) \in (0, \TT ^\star (\bTi))\times I .
\end{equation}

\item[(iii)] If $q \geq 0$ in $I$ and $q >0$ in a subinterval $(a,b) \subset \subset I$, then \eqref{eq-posit} holds.
\end{itemize}
\end{Proposition}

\begin{proof}
Let $\bTi\in \bV_+$ and let $\bT$ denote the corresponding mild solution of \eqref{2layer-pbm}, \eqref{eq:bound and in cond} defined on $ [0, \TT ^\star (\bTi)) $.
\begin{itemize}
\item[(i)] The constant function $(0,0)$ is a sub-solution of \eqref{2layer-pbm}. Therefore, by the comparison principle stated in Theorem \ref{lem-comp-2layer-sub-sup}, we obtain
$$\bT(t,x)  \geq \pmb{0},\qquad\forall\,(t,x) \in [0, \TT ^\star (\bTi))\times I.$$
    
\item[(ii)] Let $\bT^-$ be the solution of the ODE system \eqref{2layer-CLMUV-EDO} with initial condition $(0,0)$ and constant incoming flux $\bar q= \bar q_{min}:=  r_{min}q_{min} >0$, namely
\begin{equation*}
\begin{cases}
\gamma _a  T_a '(t) = -\lambda (T_a - T_s) + \varepsilon _a \sigma _B \vert T_s \vert ^3 T_s  - 2 \varepsilon _a \sigma _B \vert T_a \vert ^3 T_a  
+  \bar q_{\rm min}  \beta _a (T_a)  ,  \vspace{.15cm}\\
\gamma _s T' _s (t) = -\lambda (T_s - T_a) - \sigma _B \vert T_s \vert ^3 T_s  + \varepsilon _a \sigma _B \vert T_a \vert ^3 T_a +  \bar q_{\rm min} \beta _s (T_s)  ,\vspace{.15cm} \\
T_a (0)= 0 ,\quad T_s (0)= 0.
\end{cases}
\end{equation*}
By Proposition \ref{Prop-II.1-chaos}, the solution $\bT^-$ is globally defined and satisfies 
$$\bT^-(t)>\pmb{0},\qquad\forall\, t >0,$$ 
since $\bar q_{min} >0$ and $\beta_s > 0$. Moreover, $\bT^-$ is  a sub-solution of \eqref{2layer-pbm}. Applying again Theorem \ref{lem-comp-2layer-sub-sup}, we deduce
$$\bT(t,x)\geq\bT^-(t)>\pmb{0},\qquad \forall\, (t,x)\in (0, \TT ^\star (\bTi))\times I.$$
\item[(iii)] From item (i), we already know that $\bT(t,x)  \geq 0$ for all $(t,x) \in [0, \TT ^\star (\bTi))\times I $. Assume by contradiction that there exists $(t_a,x_a) \in (0, \TT ^\star (\bTi)) \times I$ such that 
$$T_a (t_a,x_a) =0.$$
Define 
$$u_1:= -T_a,\qquad u_2:= -T_s$$
and introduce the operators
$$ P_1 u := \gamma_a \left[ \frac{\partial u}{\partial t} - \kappa_a \frac{\partial}{\partial x} \left( (1 - x^2) \frac{\partial u}{\partial x} \right) \right],\,\, P_2 u := \gamma_s \left[ \frac{\partial u}{\partial t} - \kappa_s \frac{\partial}{\partial x} \left( (1 - x^2) \frac{\partial u}{\partial x} \right) \right].$$
Then, $(u_1,\,u_2)$ satisfies
\begin{equation}
\label{eq:u1-u2-ineq}
\left\{
\begin{aligned}
P_1 u_1 &= \lambda(T_a - T_s) - \varepsilon_a \sigma_B T_s^4 + 2\varepsilon_a \sigma_B T_a^4 - R_a(x,T_a) , 
,\vspace{.15cm}\\
P_2 u_2 &= \lambda(T_s - T_a) + \sigma_B T_s^4 - \varepsilon_a \sigma_B T_a^4 - R_s(x,T_s),
\end{aligned}
\right.
\end{equation}
that can be recast as
\begin{equation*}
\left\{
\begin{aligned}
P_1 u_1 &= -(\lambda + 2\varepsilon_a \sigma_B T_a^3) u_1 + (\lambda  + \varepsilon_a \sigma_B T_s^3) u_2 - R_a(x,T_a)
, \vspace{.15cm}\\
P_2 u_2 &= (\lambda + \varepsilon_a \sigma_B T_a^3) u_1 - (\lambda  + \sigma_B T_s^3) u_2
- R_s(x,T_s).
\end{aligned}
\right.
\end{equation*}
Since $R_a,\,R_s\geq0$, we infer that
\begin{equation}\label{eq:u1 u2}
\left\{
\begin{aligned}
P_1 u_1 &\leq -(\lambda + 2\varepsilon_a \sigma_B T_a^3) u_1 + (\lambda  + \varepsilon_a \sigma_B T_s^3) u_2  , \vspace{.15cm}\\
P_2 u_2 &\leq (\lambda + \varepsilon_a \sigma_B T_a^3) u_1 - (\lambda  + \sigma_B T_s^3) u_2
.
\end{aligned}
\right.
\end{equation}
Fix $t_a\leq\TT < \TT ^\star (\bTi)$. Since $\bTi \in \bV$, standard regularity for mild solutions (see Proposition \ref{thm-further-reg}) yields 
$$\bT \in C((0, \TT); \bDA),$$ 
and therefore
$$T_a, T_s \in C([\varepsilon,\TT]\times\bar I) \quad \text{for every } 0<\varepsilon<\TT .$$
Moreover, the local regularity result for classical solutions, Proposition \ref{prop: loc reg}, implies that there exists $\theta\in(0,1)$ such that
\begin{equation*}
    T_a,T_s\in C^{1+\frac{\theta}{2},2+\theta}({[\eps,\TT]\times I})\subset C^{1,2}({[\eps,\TT]\times I}).
\end{equation*}
Hence, Proposition \ref{prop-ppmax-system-deg} applies. Let $t_1$ be the first time such that $u_1(t_1,x_1)=0$, for some $x_1\in \bar{I}$. The existence of such point follows from the assumption $T_a(t_a,x_a)=-u_1(t_a,x_a)=0$. By the strong maximum principle for linear parabolic systems, we conclude that 
$$u_1\equiv0 \quad \text{in } (0,t_1]\times I.$$
From the first equation in \eqref{eq:u1-u2-ineq}, this implies that $u_2 \equiv 0$ in $(0,t_1]\times I$. Using the equation satisfied by $u_2$, we deduce that $r(t)q(x) \beta_s (-u_2)\equiv0$ in $(0,t_1]\times I$, which forces $q \equiv 0$ in $I$, in contradiction with the assumption on $q$. 

Therefore, $T_a(t,x) > 0$ for all $(t,x)\in (0,\TT^\star(\bT^0))\times I$. The same argument applies to $T_s$, completing the proof.
 \end{itemize} 
\end{proof}

\subsection{Proof of Theorem \ref{prop-bounds}: Boundedness and global existence for $\varepsilon _a \in (0,2)$\vspace{.2cm}} \hfill
\label{sec-globalite}

In this section, we prove that solutions of \eqref{2layer-pbm}, \eqref{eq:bound and in cond} are globally defined in time and uniformly bounded. We begin by establishing the existence of invariant rectangles for regular initial data. This result is a key ingredient in the proof of Theorem~\ref{prop-bounds} and provides the fundamental a priori bounds needed for the global well-posedness analysis. 

\begin{Proposition}[\bf Invariant rectangles for regular initial conditions]
\label{prop-rect-inv}
Assume that Hyp. \ref{HYP1} holds. Let  $\varepsilon_a \in (0,2)$, $\beta_a \geq 0$, $\beta_s >0$, and suppose that $q>0$. Let $\bTi \in \bDA_+$, and denote by $\bT$ the corresponding solution of \eqref{2layers-abstract-form}, defined on $[0, \TT^\star(\bT^0))$. 

Let $\mu \in (\varepsilon_a^{1/4},2^{1/4})$. There exists $M_a$ sufficiently large such that, for any $M\geq M_a$, if
$$ ( T_a ^{ 0} (x) ,  T_s ^{ 0} (x) ) \in [0, M] \times [0, \mu M], \qquad \forall\, x \in \bar{I},$$ 
then the solution remains in the same rectangle for all times, namely
$$ ( T_a (t,x) ,  T_s (t,x)) \in [0, M] \times [0, \mu M], \qquad \forall\, t \in [0, \TT^\star(\bT^0)), \  \forall\, x \in \bar{I}.$$
\end{Proposition}
\begin{proof}
By Proposition \ref{prop-nonneg}, the solution $\bT$ in nonnegative for all $t \in [0,\TT^\star(\bT^0))$. It therefore suffices to establish uniform upper bounds. 

Let $\bar q_{\text{max}}$ be defined as in \eqref{defqminmax}, and let
$ \bT^{\text{max}, 0} $ be the constant initial datum defined in \eqref{defCIminmax}.
Denote by $\bT^{\text{max}}(t)$ the corresponding solution of the ODE system \eqref{2layer-pbm-qmax}. By construction, $\bT^{\text{max}}$ is a super-solution of the PDE system \eqref{2layer-pbm}.

Since $\varepsilon _a \in (0,2)$, Proposition \ref{Prop-II.1-chaos} ensures that
 $\bT^{\text{max}}(t)$ is globally defined and bounded for all $t\geq0$. Moreover, because
$$( T_a ^{\text{max}, 0},T_s ^{\text{max}, 0}) \in 
 [0, M_a] \times [0, \mu M_a],$$
Lemma \ref{Lemma-IV.2-chaos} yields 
$$(T_a ^{\text{max}} (t) ,T_s ^{\text{max}} (t) ) \in 
 [0, M_a] \times [0, \mu M_a], \qquad \forall\, t \geq 0.$$

Finally, since $\bT^{\text{max}}$ is a super-solution of \eqref{2layer-pbm}, the comparison principle (Theorem \ref{lem-comp-2layer-sub-sup}) implies that
$$ T_a (t,x) \leq T_a ^{\text{max}} (t) \leq M_a 
\quad \text{ and } \quad  T_s (t,x)  \leq T_s ^{\text{max}} (t)\leq \mu M_a ,  $$
for all $t \in [0,\TT^\star(\bT^0))$ and all $x \in \bar{I}$. This concludes the proof.
\end{proof}

We finally prove the global well-posedness for mild solutions to \eqref{2layer-pbm}-\eqref{eq:bound and in cond}.
\begin{proof}[Proof of Theorem \ref{prop-bounds}]
We first consider the case of more regular initial conditions, namely $\bTi \in \bDA_+$, and denote by $\bT$ the corresponding solution of \eqref{2layer-pbm}-\eqref{eq:bound and in cond} defined on its maximal interval of existence $[0,\TT^\star (\bTi) )$. 

By Proposition \ref{prop-rect-inv}, there exists a constant $M>0$ such that
$$
\forall\, t \in [0, \TT^*(\bTi) ),\,\forall\,x\in I \qquad  0 \leq T_a (t,x)  \leq M,  \qquad  0 \leq T_s (t,x) \leq M.
$$
When $\bTi\in \bDA$ the solution is uniformly bounded in $L^\infty(I)$, which is a stronger property than \eqref{eq-bounds}.

We now derive energy estimates.

Multiplying the first equation of \eqref{2layer-pbm} by $\frac{\partial T_a}{\partial t}$ and integrating over $I$, we obtain
\begin{multline*}
\int_I\frac{\partial T_a}{\partial t} \left[\frac{\partial T_a}{\partial t} -\kappa_a \frac{\partial}{\partial x}\left((1-x^2)\frac{\partial T_a}{\partial x}\right)\right]dx
\\
= \int_I \frac{\partial T_a}{\partial t} \left[\frac{1}{\gamma_a}\left[-\lambda(T_a - T_s)+\varepsilon_a \sigma_B |T_s|^3T_s - 2\varepsilon_a \sigma_B |T_a|^3 T_a +  r(t) q(x) \beta_a(T_a)\right]\right]dx .
\end{multline*}
Integration by parts yields
\begin{multline*}
\int_I\left(\frac{\partial T_a }{\partial t}\right)^2dx
+ \frac{\kappa_a}{2} \frac{\partial}{\partial t} \int_I (1-x^2) \left( \frac{\partial T_a}{\partial x}\right)^2dx
\\
= -\frac{\lambda}{2\gamma_a} \frac{\partial}{\partial t} \int_I T_a ^2\,dx
+ \frac{1}{\gamma_a} \int_I \frac{\partial T_a}{\partial t} \left[\lambda T_s+\varepsilon_a\sigma_B|T_s|^3T_s\right.\\\left.-2\varepsilon_a\sigma_B|T_a|^3T_a+ r(t)q(x)  \beta_a(T_a) \right]dx,
\end{multline*}
and therefore
\begin{multline*}
\int_I\left(\frac{\partial T_a}{\partial t}\right)^2dx
+ \frac{\lambda}{2\gamma_a}  \frac{\partial}{\partial t} \int_I T_a ^2\,dx
+ \frac{\kappa_a}{2} \frac{\partial}{\partial t} \int_I (1-x^2) \left(\frac{\partial T_a}{\partial x}\right)^2dx
\\ 
= \frac{1}{\gamma_a} \int_I \frac{\partial T_a}{\partial t} \left[\lambda T_s +\varepsilon_a \sigma_B |T_s|^3T_s-2\varepsilon_a\sigma_B|T_a|^3T_a+ r(t)q(x)  \beta_a(T_a)\right]dx
\\ \leq
\frac{1}{2}\int_I\left(\frac{\partial T_a}{\partial t}\right)^2dx
+
\frac{1}{2\gamma_a ^2 } \int_I\left[\lambda T_s +\varepsilon_a \sigma_B |T_s|^3T_s\right.\\\left. - 2\varepsilon_a\sigma_B|T_a|^3T_a +r(t) q(x)  \beta_a(T_a)\right] ^2dx.
\end{multline*}
Thus, we obtain
\begin{multline}
\label{wp-borneV-Ta}
\frac{d}{d t} \left( \frac{\lambda}{2\gamma_a}  \int_I T_a ^2\,dx
+ \frac{\kappa_a}{2} \int_I (1-x^2) \left(\frac{\partial T_a}{\partial x}\right)^2dx \right)
\\ 
\leq \frac{1}{2\gamma_a ^2 } \int_I\left[\lambda T_s +\varepsilon_a \sigma_B |T_s|^3T_s - 2\varepsilon_a\sigma_B|T_a|^3T_a + r(t)q(x)  \beta_a(T_a)\right] ^2dx \leq M'_1,
\end{multline}
for some constant $M_1'>0$, where we have used the uniform boundedness of $\bT$. Proceeding analogously for the second equation, we obtain
\begin{multline}
\label{wp-borneV-Ts}
\frac{d}{d t} \left( \frac{\lambda}{2 \gamma_s}  \int_I T_s ^2\,dx
+ \frac{\kappa_s}{2} \int_I (1-x^2) \left(\frac{\partial T_s}{\partial x}\right)^2dx \right)
\\ 
\leq \frac{1}{2\gamma_s ^2 } \int_I\left[\lambda T_a - \sigma_B |T_s|^3T_s + \varepsilon_a\sigma_B |T_a|^3 T_a 
+r(t) q(x)  \beta_s(T_s) \right] ^2dx \leq M'_2
\end{multline}
with $M_2'>0$. Combining \eqref{wp-borneV-Ta} and \eqref{wp-borneV-Ts}, we deduce that $t \mapsto \norm{\bT(t)}_{\bV}$ grows at most linearly in time. Consequently, no blow-up in the $\bV$-norm can occur in finite time, and the solution can be extended globally, i.e.,  $\TT^\star(\bT^0)=+\infty$.

Finally, let $\bT^0 \in \bV_+$. By standard parabolic regularization, the associated mild solution satisfies $\bT(t)\in \bDA_+$ for all $t\in(0,\TT^\star(\bT^0))$. Hence, the previous argument applies for any positive time, yielding uniform bounds on any interval of the form $[t,\TT^*(\bTi))$, with $t>0$. So, global existence also holds in this case.
This concludes the proof.
\end{proof}


\subsection{Proof of Theorem \ref{prop-blowup}: Blow-up in finite time for $\varepsilon _a >2$\vspace{.2cm}} \hfill

This section is devoted to showing that finite-time blow-up may occur for solutions of \eqref{2layer-pbm}-\eqref{eq:bound and in cond} when $\varepsilon_a>2$.
\begin{proof}[Proof of Theorem \ref{prop-blowup}]\hfill

\begin{itemize}
\item[(a)] Let $\bar q_{\text{\rm min}}$ be defined as in \eqref{defqminmax}, and denote $\bar{\bT}$ the solution of the ODE system \eqref{2layer-CLMUV-EDO} associated with the constant incoming flux $\bar q=\bar q_{\text{\rm min}}$. Since $\varepsilon _a >2$ and $\bar q_{\text{\rm min}} >0$, Proposition \ref{prop.II.6.b-chaos},(a) ensures the existence of nonnegative initial conditions $\bar{\bT}^0$ such that $\bar{\bT}$ blows up in finite time. Denote by $\bar \TT^\star<+\infty$ its maximal time of existence. 

We now take $\bT^0=\bar{\bT}^0$ and denote by $\bT$ the corresponding solution of \eqref{2layer-pbm}-\eqref{eq:bound and in cond} defined on $[0, \TT^\star(\bTi))$. Since  $\bar \bT$ is a sub-solution of \eqref{2layer-pbm}, Theorem~\ref{lem-comp-2layer-sub-sup} forces $\TT^\star(\bTi)) \leq  \overline\TT^\star$, thus yielding the conclusion. 
\item[(b)] Let $\bTi \in \bDA_+$, and define the constant initial datum $\bT^{\text{\rm min}, 0}$ as in \eqref{defCIminmax}. Let $\bar q_{\text{\rm min}}$ be given by \eqref{defqminmax} and denote by $\bT^{\rm min}$ the corresponding solution of the ODE system \eqref{2layer-CLMUV-EDO}, defined on $[0,\bar\TT^\star)$. Since $\varepsilon _a >2$ and $\bar q_{\text{\rm min}} >0$, Proposition \ref{prop.II.6.b-chaos},(b) implies that $\bT^{\text{\rm min}}$ blows up in finite time, that is, $ \bar \TT^\star<+\infty$.

Now, observe that $\bT^{\rm min}$ is a sub-solution of \eqref{2layer-pbm} and apply again Theorem~\ref{lem-comp-2layer-sub-sup} to conclude that $\TT^*(\bTi)<+\infty.$
\end{itemize}
\end{proof}
\section{Long time behavior}\label{sec-attractor-statement}
\subsection{Equilibria\vspace{.2cm}} \hfill

Throughout this section we assume that Hyp. \ref{HYP1} holds and that
\begin{equation}
    \label{eq:hyp-long time}
    \eps_a\in(0,2),\qquad\beta_a=0,\qquad\beta_s>0,\qquad r(t)\equiv r_0,  
\end{equation}  
so that $\bG(t,\cdot)=\bG(0,\cdot)$, and we denote such function $\bG_0(\cdot)$. We recall that an equilibrium (or stationary solution) of
\begin{equation}\label{eq:abstract system r0}
    \bT'(t)=\bA\bT(t)+\bG_0(\bT(t)),\qquad t\geq0
\end{equation}
is an element $\pmb E\in \bDA$ such that 
\begin{equation*}
    \bA\pmb{E}+\bG_0(\pmb{E})=\pmb0.
\end{equation*}
We now prove the existence of equilibria for \eqref{eq:abstract system r0}.
\begin{Lemma}\label{lem:ex equilibrium}
    Let $\overline{\pmb{E}}^{\max}$ be the warmest equilibrium of \eqref{2layer-CLMUV-EDO} with $\bar q=\bar q_{\max}$. Let $\bT^{q,\max}$ be the solution of \eqref{2layers-abstract-form} with initial condition $\overline{\pmb{E}}^{\max}$. Then, $\bT ^{ q, \max} (t,x)$ converges as $t\to+\infty$ to an equilibrium of \eqref{eq:abstract system r0}, $\bT ^{q, \text{max}}_\infty$, uniformly in $\bar I$.
\end{Lemma}
\begin{proof}
   Let $\bT ^{q_{\text{max}}, \text{max}}(t,x)$ denote the solution of \eqref{eq:abstract system r0}, with initial condition $\overline{\pmb E}^{\max}$ and $q=q_{\text{max}}$. By uniqueness, $\bT ^{q_{\text{max}}, \text{max}}(t,x)\equiv \overline{\pmb E}^{\max}$. Therefore, the comparison principle yields
$$
\forall\,t\geq0 \quad \bT ^{ q, \text{max}}(t,x)\leq \bT ^{q_{\text{max}}, \text{max}}(t,x)= \overline{\pmb E}^{\max}=\bT^{ q, \text{max}} (0,x).
$$
Fix $t\geq 0$ and consider the solutions of \eqref{eq:abstract system r0} with initial conditions $\bT ^{ q, \text{max}} (t,x)$ and $\bT^{ q, \text{max}} (0,x)$, respectively. Again by the comparison principle 
$$ \forall\, s \geq 0, \quad \bT ^{ q, \text{max}} (t+s,x) \leq \bT ^{q, \text{max}} (s,x).
$$
This shows that both components $t\mapsto T_a ^{q, \text{max}} (t,x)$ and $t\mapsto T_s ^{q, \text{max}} (t,x)$ are nonincreasing functions of time. Since they are also bounded below, there exists
$\bT^{q,\max}_\infty(x)$ such that 
$$\bT ^{ q, \text{max}} (t,x) \to \bT ^{q, \text{max}}_\infty(x),\quad \text{as } t\to+\infty,$$
component-wise, uniformly with respect to $x$ by Dini's Theorem. Such a limit is a fixed point of the semi-flow, and therefore a stationary solution
of \eqref{eq:abstract system r0}.
\end{proof}
We now investigate the existence of the so-called \emph{warmest} equilibrium. 
\begin{Theorem}
\label{thm-warmest}
The stationary solution $\bT ^{q, \text{max}}_\infty$ constructed in Lemma \ref{lem:ex equilibrium} is the \emph{warmest} equilibrium of \eqref{eq:abstract system r0}, that is, any other stationary solution $\bT^{q}$ of \eqref{eq:abstract system r0} satisfies
\begin{equation}\label{eq: warmest eq}
\forall\, x \in \bar I, \quad \bT^{q} (x)\leq \bT ^{q, \text{max}}_\infty (x). 
\end{equation}
\end{Theorem}
\begin{proof}
Let $\bT ^{q}(x)$ be any stationary solution of \eqref{eq:abstract system r0} and set 
$$ \pmb{M}=(\max_{\bar I}\, T^q_a,\max_{\bar I}\, T^q_s).
$$
Consider the solutions $\bT^{q,\pmb M}$ of \eqref{eq:abstract system r0} with initial conditions $\pmb M$. By comparison,
\begin{equation*}
\forall\,t\geq0,\,\forall\,x\in \bar I,\quad\bT^{q}(x)\leq \bT^{q,\pmb M}(t,x).
\end{equation*}
On the other hand, once again by comparison and uniqueness
\begin{equation*}
    \forall\,t\geq0,\,\forall\,x\in\bar I,\quad \bT^{q,\pmb M}(t,x)\leq \bT^{q_{\max},\pmb M}(t,x)=\overline \bT^{q_{\max},\pmb M}(t),
\end{equation*}
where $\overline\bT^{q_{\max},\pmb M}$ denotes the solution of the ODE problem \eqref{2layer-CLMUV-EDO} with $\overline q=r_0q_{\max}$ and initial condition $\pmb M$. Since $\overline\bT^{q_{\max},\pmb M}$ converges to some equilibrium $\overline{\pmb{E}}\leq \overline {\pmb{E}}^{\max}$ as $t\to+\infty$, we conclude that
$$ \forall \,x\in \bar I, \quad  \bT^{q}(x) \leq \overline {\pmb{E}}^{\max}.$$
Recalling that $\bT^{q,\max}$ is the solution of \eqref{2layers-abstract-form} with initial condition $\overline{\pmb{E}}^{\max}$, by comparison we finally have
\begin{equation*}
    \forall\,t\geq0,\,\forall\,x\in\bar I,\quad \bT^{q}(x)\leq \bT^{q,\max}(t,x).
\end{equation*}
Now, invoking Lemma \ref{eq: warmest eq} the conclusion follows. 
\end{proof}

\begin{Remark}
{\rm Both the Lemma and Theorem above have analogous statements for the \lq\lq coldest" equilibrium. It suffices to consider the coldest equilibrium of \eqref{2layer-CLMUV-EDO} in Lemma \ref{lem:ex equilibrium} and substitute $(\max_{\bar I}\, T^q_a,\max_{\bar I}\, T^q_s)$ with $(\min_{\bar I}\, T^q_a,\min_{\bar I}\, T^q_s)$ in the proof of Theorem \ref{thm-warmest}. We leave the details to the reader.} 
\end{Remark}


\subsection{Global attractor\vspace{.2cm}} \hfill

Hereafter we assume Hyp. \ref{HYP1} satisfied and \eqref{eq:hyp-long time} with $r_0=1$. Let us introduce the semi-flow on the metric space $\bV_+$
$$ S(t): \bV_+ \to \bV_+, \quad {\bT}^0  \mapsto S(t)\bTi={\bT} (t),\qquad t\geq0 $$
where ${\bT}$ denotes the solution of \eqref{eq:abstract system r0} with initial condition $\bTi$. We now recall the definition of a global attractor.
\begin{Definition}
    A set $\pmb{K}\subset \bV_+$ is a \emph{global attractor} for $\{S(t)\}_{t\geq0}$ in $\bV_+$ if
    \begin{itemize}
        \item[(i)] $S(t)\pmb K=\pmb K$, $\quad\forall\,t\geq0$,\vspace{.1cm}
        \item[(ii)] for all bounded subsets $\pmb{B}$ of $\bV_+$ it holds
        \begin{equation*}
            \sup_{\pmb u\in \bB}d_{\pmb K}(S(t)\pmb u) \to0\quad\text{as}\quad t\to+\infty,
        \end{equation*}
        with
        \begin{equation*}
            d_{\pmb K}(\pmb u):=\inf_{\pmb v\in\pmb K}\norm{\pmb u-\pmb v}_{\bV}\qquad \pmb (u\in\bV_+).
        \end{equation*}
    \end{itemize}
\end{Definition}

We are able to prove the following result on the asymptotic behavior of solutions to the 2-layer EBM \eqref{2layers-abstract-form}. 
\begin{Theorem}
\label{thm-attractor}
There exists a compact, connected global attractor in $\bV_+$.
\end{Theorem}
We will apply \cite[Theorem 1.1]{Temam}, for which several preliminary results are required. These include Lemma \ref{lem-decay-V} below, as well as the results on dissipation and regularization of solutions to \eqref{eq:abstract system r0} presented in Subsections \ref{subsec:dissip} and \ref{subsec-regular}. The proof of Theorem \ref{thm-attractor} is given in Subsection \ref{subsec: constr-glob-att}.

\begin{Lemma} 
\label{lem-decay-V}
There exists a constant $ \bM=(M_a,M_s)>\pmb 0$ such that, for all $\bTi \in \bDA_+$, there exists
$\tau_0=\tau_0(\norm \bTi _\bDA)$ such that
\begin{equation}
\label{borne-inf-sup-T}
\forall\, t\geq \tau_0,\, \forall\, x\in I, \quad 0\leq {\bT}(t,x)\leq \bM.
\end{equation}
\end{Lemma}
\begin{proof} 
Fix $\mu \in (\varepsilon_a ^{1/4}, 2^{1/4})$. Since $\bTi \in \bDA_+$ 
and $D(A) \subset C(\bar{I})$, there exists $M_0>0$ large enough such that
\begin{equation}
\label{eq:M0}
    \forall\, x \in \bar{I},\quad  T_a ^{0} (x) \in [0,M_0],\quad
T_s ^{0} (x) \in [0,\mu M_0] ,
\end{equation} 
and, by Proposition \ref{prop-rect-inv}, we have also
\begin{equation}
\label{eq:M00}
    \forall\,t\geq0,\,\forall\, x \in \overline I,\quad  T_a (t,x) \in [0,M_0],\quad
T_s (t,x) \in [0,\mu M_0] .
\end{equation}
Let us consider the solution $\overline{\bT}=(\overline T_a, \overline T_s)$ of the ODE system
\begin{equation}
\label{eq-EDO-sup}
\begin{cases}
\gamma _a \overline T_a ' = -\lambda (\overline T_a -\overline T_s) + \varepsilon_a \sigma_B \overline T_s ^4 -2 \varepsilon _a \sigma_B \overline T_a ^4 + q_{\max} \beta _a (\overline T_a), \\
\gamma _s \overline T_s ' = -\lambda (\overline T_s -\overline T_a) - \sigma_B \overline T_s ^4 + \varepsilon _a \sigma_B \overline T_a ^4 + q_{\max} \beta _s (\overline T_s), \\
\overline T_a (0)= M_0,\quad \overline T_s (0) = \mu M_0 .
\end{cases} 
\end{equation}
By the comparison principle we have that
\begin{equation}
\label{eq:bound-T-barT}
    \forall \,t\geq 0,\, \forall\, x\in I, \quad 0\leq \bT(t,x)\leq \overline{\bT}(t).
\end{equation}
We recall that (see Proposition \ref{prop.II.3-chaos}) system \eqref{eq-EDO-sup} has a finite number of equilibrium points and
the solution $\overline{\bT}$ converges as $t\to +\infty$ to one of such equilibria. Let $\overline{\pmb{E}}^{\max}$ be the warmest equilibrium point. Then, there exists $\tau_0$ such that
$$ t\geq \tau_0 \quad \implies \quad \overline{\bT}(t)\leq \overline{\pmb{E}}^{\max}+\pmb 1.$$
Thanks to \eqref{eq:bound-T-barT}, the last inequality provides a uniform bound for the solution $\bT$ of the PDE system \eqref{eq:abstract system r0}:
\begin{equation}
\label{borne-inf-sup-T-june}
\forall\, t\geq \tau_0, \forall\, x\in I,\, \quad 0\leq {\bT}(t,x)\leq  \overline{\pmb{E}}^{\max}+\pmb 1\quad .
\end{equation}
The time $\tau _0$ depends on  the initial condition $\bTi \in \bDA$ as a function of $M_0$, and therefore only through the norm $\norm{\bTi}_{\bDA}$. Hence \eqref{borne-inf-sup-T} with $\bM=\overline{\pmb{E}}^{\max}+\pmb 1$ is proved. 
\end{proof}

\subsubsection{Dissipation\vspace{.2cm}}\label{subsec:dissip}\hfill 

Let $\bTi \in \bV$, and let $S(t)\bTi={\bT}(t,x)=(T_a(t,x),T_s(t,x))$ be the solution of \eqref{eq:abstract system r0}. We introduce energy functionals: for any $\pmb{u}=(u_1,u_2)\in\bH$
\begin{equation}
\label{def-norm-H}
\mathcal E_{{\bH}} (\pmb{u}) := \int _{-1} ^1 \left(\gamma _a u_1 ^2 + \gamma _s u_2 ^2\right) \, dx ,
\end{equation}
and for any $\pmb{u}=(u_1,u_2)\in\bV$
\begin{equation}
\label{def-norme-V}
    \mathcal E_{{\bV}} (\pmb u) := \gamma _a \kappa_a \int _{-1} ^1 (1-x^2) \left(\frac{d u_1}{d x}\right) ^2 \, dx + \gamma _s \kappa_s \int _{-1} ^1 (1-x^2) \left(\frac{d u_2
    }{d x}\right) ^2 \, dx .
\end{equation}
Note that the mapping
$$\pmb u \mapsto \sqrt{\mathcal E_{{\bH}} (\pmb u)}$$ 
defines a norm on ${\bH}$ that is equivalent to the standard norm of ${\bH}$. Moreover, the mapping 
$$  \pmb u \mapsto \sqrt{\mathcal E_{{\bH}} (\pmb u) + \mathcal E_{{\bV}} (\pmb u)}$$ 
defines a norm on ${\bV}$ that is equivalent to the standard norm of ${\bV}$.

\begin{Lemma} 
\label{lem-decay-V2}
For all $\sigma > 0$, there exists  $L(\sigma,\pmb M)>0$ such that for all $\bTi\in \bDA_+$
\begin{equation}
\label{eq-decay-V}
\forall\, t \geq \tau_0, \quad  \mathcal E_{{\bV}} (\bT(t)) \leq L(\sigma,\pmb M)
+ e^{-\sigma (t-\tau _0)}  \mathcal E_{{\bV}} (\bT(\tau _0)),
\end{equation}
where $\pmb M$ and $\tau_0$ are the same constants as in Lemma \ref{lem-decay-V}.
\end{Lemma}
\begin{proof}
Let us assume first $\bTi\in \pmb D((\pmb I-\bA)^{3/2})_+$. Since, in particular, $\bTi\in \bV_+$, Theorem \ref{prop-bounds} ensures that, for any $\TT>0$, $\bT(t)=S(t)\bTi\in C([0,\TT];\bV)$. We set 
\begin{equation*}
    \pmb W (t)=(\pmb I-\bA)^{1/2}\bT(t).
\end{equation*}
Let us consider the following Cauchy problem
\begin{equation}
\label{eq:problem W}
    \begin{cases}
        \pmb W'(t)=\bA \pmb W(t)+\pmb g(t) & t\in[0,\TT]\\
        \pmb W(0)=(\pmb I-\bA)^{1/2}\bTi,
    \end{cases}
\end{equation}
with $\pmb W(0)\in \bDA_+$ and
\begin{equation*}
    \pmb g(t):=(\pmb I-\bA)^{1/2}\bG_0(\bT(t)).
\end{equation*}
Notice that $\pmb g\in L^2(0,\TT;\bH)$ because $\bG_0\circ \bT\in L^2(0,\TT;\bV)$. Therefore, by \cite[Theorem 3.1, p. 143]{bensoussan1992}, there exists a unique solution $\pmb W\in H^1(0,\TT;\bH)\cap L^2(0,\TT;\bDA)$ to \eqref{eq:problem W}. We deduce that
\begin{equation*}
    \bT(t)=(\pmb I-\bA)^{-1/2}\pmb W(t)\in H^1(0,\TT;\bV).
\end{equation*}
Hence, we can differentiate $\mathcal E_{{\bV}}(t):=\mathcal E_{{\bV}}(\bT(t))$ with respect to time and obtain
\begin{equation*}
\mathcal E_{\bV}'(t)
= 2\gamma_a \int_{-1}^1 \kappa_a(1-x^2) \frac{\partial T_{a}}{\partial x}\frac{\partial^2 T_{a}}{\partial t \partial x}\,dx
+2\gamma_s \int_{-1}^1 \kappa_s(1-x^2) \frac{\partial T_s}{\partial x}\frac{\partial^2 T_s}{\partial t \partial x}\,dx .
\end{equation*}
Integrating by parts and using the boundary conditions satisfied by
$(1-x^2)\frac{\partial T_a}{\partial x}$ and $(1-x^2)\frac{\partial T_s}{\partial x}$, we obtain
\begin{multline*}
\mathcal E_{\bV}'(t)
= -2\gamma_a \int_{-1}^1 \kappa_a\frac{\partial}{\partial x}\left((1-x^2)\frac{\partial T_a}{\partial x}\right) \frac{\partial T_a}{\partial t}\,dx
\\-2\gamma_s \int_{-1}^1 \kappa_s\frac{\partial}{\partial x}\left((1-x^2)\frac{\partial T_s}{\partial x}\right) \frac{\partial T_s}{\partial t}\,dx .
\end{multline*}
Using the equations of system~\eqref{2layer-pbm}, the above identity gives
\begin{multline}
\label{expression-E_V'}
\mathcal E_{\bV}'(t)
= -2\gamma_a \int_{-1}^1 \left(\frac{\partial T_a}{\partial t}\right)^2\,dx
-2\gamma_s \int_{-1}^1 \left(\frac{\partial T_s}{\partial t}\right)^2\,dx
\\
+2\int_{-1}^1 \frac{\partial T_a}{\partial t}\left(-\lambda(T_a-T_s)
+\varepsilon_a\sigma_B T_s^4
-2\varepsilon_a\sigma_B T_a^4
+q(x)\beta_a(T_a)\right)\,dx
\\
+2\int_{-1}^1 \frac{\partial T_s}{\partial t}\left(-\lambda(T_s-T_a)
-\sigma_B T_s^4
+\varepsilon_a\sigma_B T_a^4
+q(x)\beta_s(T_s)\right)\,dx .
\end{multline}
Moreover, the derivative of the energy $\mathcal E_{\bH}(t):=\mathcal E_{\bH}(\bT(t))$ reads as follows
\begin{equation}
\label{a26-0}
\mathcal E_{\bH}'(t)
=2\int_{-1}^1
\left(\gamma_a \frac{\partial T_a}{\partial t}T_a+\gamma_s \frac{\partial T_s}{\partial t}T_s\right)\,dx .
\end{equation}
Now, $\mathcal E_{\bH}'$ can be computed from the above identity by multiplying the first equation of~\eqref{2layer-pbm} by $T_a$, the second by $T_s$, integrating by parts, and summing the resulting identities. We thus obtain
\begin{multline}
\label{a26-1}
\frac12\mathcal E_{\bH}'(t)+\mathcal E_{\bV}(t)
= -\lambda\int_{-1}^1 (T_a-T_s)^2\,dx
\\
+\int_{-1}^1 q(x)
\left(\beta_a(T_a)T_a+\beta_s(T_s)T_s\right)\,dx
\\
+\varepsilon_a\sigma_B\int_{-1}^1
\left(
|T_s|^3T_sT_a
+|T_a|^3T_aT_s
-2|T_a|^5
-\varepsilon_a^{-1}|T_s|^5
\right)\,dx .
\end{multline}
Fix $\sigma>0$, multiply both sides of \eqref{a26-1} by $\sigma$, move the term $\frac{\sigma}{2}\mathcal E_{\bH}'$ to the right-hand side and combine the resulting equation with \eqref{expression-E_V'} to obtain
\begin{equation*}
\begin{split}
\mathcal E_{{\bV}} '(t) + \sigma \mathcal E_{{\bV}} (t)
=& -2 \gamma _a \int _{-1} ^1 \left(\frac{\partial T_a}{\partial t}\right) ^2 \, dx 
-2 \gamma _s \int _{-1} ^1 \left(\frac{\partial T_s}{\partial t}\right) ^2 \, dx 
\\
&+ 2 \int _{-1} ^1 \frac{\partial T_a}{\partial t} \left(-\lambda (T_a-T_s) + \varepsilon_a \sigma_B T_s ^4 -2 \varepsilon_a \sigma_B T_a ^4 + q \beta _a (T_a) \right) \, dx
\\
&+ 2 \int _{-1} ^1 \frac{\partial T_s}{\partial t} \left(-\lambda (T_s-T_a) - \sigma_B T_s ^4 + \varepsilon_a \sigma_B T_a ^4 + q \beta _s (T_s) \right) \, dx
\\
&- \sigma \int _{-1} ^1 \left(\gamma _a \frac{\partial T_a}{\partial t} T_a + \gamma _s \frac{\partial T_s}{\partial t} T_s\right)\, dx
\\
&-\sigma\lambda \int _{-1} ^1 (T_a-T_s)^2 \, dx + \sigma\int _{-1} ^1 q(x) \left(\beta _a (T_a) T_a  + \beta _s (T_s) T_s\right) \, dx 
\\
&+ \sigma \varepsilon _a \sigma _B \int _{-1} ^1 \left(\vert T_s \vert ^3 T_s T_a + \vert T_a \vert ^3 T_a T_s - 2 \vert T_a \vert ^5 - \varepsilon _a^{-1} \vert T_s \vert ^5\right) \, dx .
\end{split}
\end{equation*}
Observe that, by Young’s inequality, the terms involving $\frac{\partial T_a}{\partial t}$ and $\frac{\partial T_s}{\partial t}$
can be absorbed into the negative contributions
$-2\gamma_a\norm{\frac{\partial T_a}{\partial t}}_{L^2(I)}^2$ and
$-2\gamma_s\norm{\frac{\partial T_s}{\partial t}}_{L^2(I)}^2$. So, thanks to the uniform bounds provided by \eqref{borne-inf-sup-T}, we deduce that there exists a constant $N(\sigma,\pmb M)>0$,
such that $ \mathcal E_{{\bV}} '(t) + \sigma \mathcal E_{{\bV}} (t)
\leq N(\sigma,\pmb M)$, $\forall\,t\geq \tau_0.$
The conclusion follows by a Gr{\"o}nwall-type argument. Indeed, for $t \geq \tau _0$,
$$
\frac{d}{dt} \left[ e^{\sigma t} \left( \mathcal E_{{\bV}} (t) - \frac{N(\sigma,\pmb M)}{\sigma } \right)  \right] 
= \sigma  e^{\sigma t} \left( \mathcal E_{{\bV}} (t) - \frac{N(\sigma,\pmb M)}{\sigma } \right)
+ e^{\sigma t} \mathcal E_{{\bV}} '(t)  \leq 0,
$$
and therefore $t \mapsto e^{\sigma t} \left( \mathcal E_{{\bV}} (t) - \frac{N(\sigma,\pmb M )}{\sigma} \right)$ is nonincreasing on $[\tau _0, +\infty)$.  Thus, we conclude that
$$ \forall\, t \geq \tau _0, \quad 
\mathcal E_{{\bV}} (t) \leq \frac{N(\sigma,\pmb M )}{\sigma }
+ e^{-\sigma  (t-\tau _0)}  \mathcal E_{{\bV}} (\tau_0) .$$
If the initial condition $\bTi\in \bDA_+$, we consider a smoother sequence $(\bT^k)_k\in \pmb D(\pmb I-\bA)^{3/2}$ that approximates the initial condition:
\begin{equation*}
    \bT^k\to \bTi\qquad \text{in }\bDA.
\end{equation*}
For instance, we can choose $\bT^k=e^{\frac{1}{k}\bA}\bTi$. Observe that $\norm{\bT^k}_{\bDA}\leq\norm{\bTi}_{\bDA}$ because $e^{t\bA}$ is a contraction semigroup. Furthermore,
\begin{equation*}
    S(t)\bT^k\to S(T)\bTi \quad\text{in }\bDA.
\end{equation*}
Indeed,
\begin{equation*}
    \norm{\bA S(t)\bT^k-\bA S(t)\bTi}_{\bH}\leq \norm{\bT^k-\bTi}_{\bDA},
\end{equation*}
and the term on the right-hand side tends to zero as $k\to+\infty$. So, the argument used to obtain \eqref{eq-decay-V} applies to the sequence of initial data $(\bT^k)_k$. Passing to the limit, we deduce that \eqref{eq-decay-V} also holds for any initial condition $\bTi \in \bDA_+$.
\end{proof}
We recall that the sets $\pmb B^+_{\bDA}(\cdot)$ are defined in \eqref{dfn:ball in D(A)+}, and we prove the following
\begin{Lemma}\label{rmk:bound E'}
    For any $R>0$, there exists $C_1(R)>0$ such that for all $\bTi\in \bB_{\bDA}^+(R)$ we have that
    \begin{equation*}
        \mathcal{E}'_{\bV}(t)\leq C_1(R),\qquad\forall\,t\geq0.
    \end{equation*}
\end{Lemma}
\begin{proof}
    It is enough to apply Young's inequality to the second and third lines of formula \eqref{expression-E_V'}, and absorb the resulting terms with the time derivatives of the temperature into the negative terms in the first line. Consequently, one obtains a uniform bound for $\mathcal{E}'_{\bV}(t)$ depending only on the constant $M_0$ in \eqref{eq:M00}.
\end{proof}
Now we prove the existence of a certain absorbing set in $\bV$.
\begin{Proposition}
\label{prom-absorb-bounded}
There exists a set $\pmb U\subseteq \bV_+$, bounded in $\bV$, with the following property: for all $R>0$ there exists $t_1=t_1(R)>0$ such that
\begin{equation*}
    S(t)\pmb B^+_{\bDA}(R)\subseteq \pmb U,\quad \forall\,t\geq t_1.
\end{equation*}
\end{Proposition}

\begin{proof}
Let $R>0$. In view of \eqref{borne-inf-sup-T} there exists $\pmb M=(M_a,M_s)>\pmb{0}$, independent of $R$, such that
\begin{equation*}
    \pmb 0\leq S(t)\bTi\leq \pmb{M},\qquad\forall\, t\geq \tau_0,\,\forall\,\bTi\in \bB_{\bDA}^+(R),
\end{equation*}
for some $\tau_0=\tau_0(R)$. Thus
\begin{equation}
    \label{eq:stima E_H}
    \mathcal E_{{\bH}} (S(t)\bTi) \leq 2 \gamma _a M_a ^2 + 2 \gamma _s M_s ^2=:C_0,\quad \forall\, t \geq \tau_0,\,\forall\,\bTi\in\bB_{\bDA}^+(R) .
\end{equation}
Furthermore, by Lemma \ref{rmk:bound E'} and the fact that $\bB_{\bDA}^+(R)$ is bounded in $\bV$, we conclude that 
$$\mathcal E_{{\bV}} (S(t)\bTi)\leq \mathcal E_{{\bV}} (\bTi) + C_1(R) t\leq C_2(R)(1+t),\quad \forall\,t\geq0,\,\forall\,\bTi\in\bB_{\bDA}^+(R),$$
for some constants $C_i(R)>0$ ($i=1,2$). Therefore, applying \eqref{eq-decay-V}, we deduce that for any $\sigma>0$ there exists $L(\sigma,\pmb M)>0$ with
\begin{multline*}\mathcal E_{{\bV}} (S(t)\bTi) \leq  L(\sigma,\pmb M )
+ e^{-\sigma  (t-\tau _0)} \mathcal{E}_\bV(S(\tau_0)\bTi)\\
\leq L(\sigma,\pmb M)+C_2(R)e^{-\sigma  (t-\tau _0)}(1+\tau_0),\quad \forall\, t \geq \tau _0,\,\forall\,\bTi\in\bB_{\bDA}^+(R),
\end{multline*}
where $\tau_0$ is the same constant as in \eqref{eq:stima E_H}. Now, fixing $\sigma=1$, we obtain, for some $t_1=t_1(R) \geq \tau _0$,
$$\quad 
\mathcal E_{{\bV}} (S(t)\bTi) \leq  L (1,\pmb M ) + 1 ,\quad  \forall\, t \geq t_1,\, \forall\, \bTi \in \bB_{\bDA}^+(R).$$
The conclusion follows from the above inequality and \eqref{eq:stima E_H} taking
\begin{equation}
    \label{eq: def U}
    \pmb U=\{\pmb{u}\in\bV_+\,:\,\,\mathcal{E}_\bH(\pmb{u})\leq C_0,\,\,\,\mathcal{E}_\bV(\pmb{u})\leq L(1,\pmb M)+1\}.\qedhere
\end{equation}
\end{proof}
\subsubsection{Regularization\vspace{.2cm}}\label{subsec-regular}\hfill

\begin{Proposition} 
\label{prop-borne-t_1-D(A)}
For all $R>0$ there exist $\tau_R>0$ and $C(R)>0$ such that 
\begin{equation}
    S(\tau_R)\pmb B^+_\bV(R)\subseteq \pmb B^+_{\bDA}(C(R)).
\end{equation}
\end{Proposition}
The following technical result, together with \eqref{eq-T-bounded}, is needed for the proof of Proposition \ref{prop-borne-t_1-D(A)}.
\begin{Lemma}
\label{lem-maj-T-infty}
For every $R>0$ there exists $C(R)>0$ such that, for all $\bTi\in \pmb B_\bV(R)$ it holds that $S(t)\bTi\in\pmb L^\infty(I)$ for a.e. $t\in[0,\tau_R]$ and
\begin{equation}
\label{22mai-2}
\int _0 ^{\tau_R} t^6 \norm {S(t)\bTi} _{\pmb L^\infty (I)} ^6 \, dt \leq C(R),
\end{equation}
where $\tau_R$ is defined in Proposition \ref{thm-loc-wellpos} and we have set $\norm{\pmb u}_{\pmb L^\infty(I)}=\norm{u_a}_{L^\infty(I)}+\norm{u_s}_{L^\infty(I)}$.
\end{Lemma}
\begin{proof}
Let $\bTi \in \pmb B_V(R)$. Then
\begin{equation}
\label{eq:sol with Y and Z}
    S(t)\bTi = e^{t {\bA}} \bTi + \int _0 ^t e^{(t-s) {\bA}} \bG_0 (S(s)\bTi) \, ds=:\pmb Y(t)+\pmb Z(t),\quad\forall\,t\geq0.
\end{equation}
Since $e^{t\bA}$ is an analytic semigroup of contractions we have that $\pmb Y(t)\in \bDA$ and there exist constants $C_0,\, C_1>0$ such that for all $t>0$
\begin{equation}
\label{eq:estim Y_H}
\Vert  \bY(t) \Vert_{{\bH}}
= \Vert e^{t \mathcal {\bA}} \bTi \Vert_{{\bH}}
\leq C_0 \Vert  \bTi \Vert_{{\bH}} ,
\end{equation}
and
\begin{equation}
\label{eq:estim Y_DA}
t \Vert \bA {\bY}(t) \Vert_{{\bH}}
= \Vert t \bA e^{t \bA} \bTi \Vert_{\bH}
\leq C_1 \Vert \bTi \Vert_{\bH} .
\end{equation}
Therefore,
\begin{multline*}
t^6 \norm{ \bY (t)}_{\bDA} ^6\\\leq C t ^6 \left(\norm{\pmb Y(t)} _{\bH} ^6 + \norm{\bA\pmb Y(t) } _{\bH} ^6\right)\leq C(t^6+1)R^6,\quad\forall\,t>0,\,\forall\,\bTi\in\pmb B_\bV(R).
\end{multline*}
Now, since $\bDA \hookrightarrow \pmb L^\infty(I)$, from the above estimate we deduce
\begin{equation*}
    t^6\norm{\pmb Y(t)}^6_{\pmb L^\infty(I)}\leq C(t^6+1)R^6,\quad\forall t>0,\,\forall\,\bTi\in \pmb B_\bV(R).
\end{equation*}
Next, recalling that $t\mapsto S(t)\bTi=:\bT(t)$ is a $\bV$-continuous function and \eqref{eq-T-bounded} holds, we have that
\begin{multline*}
    \norm{\bG_0(\bT(t))}^2_\bH\leq C(1+ \norm{\bT(t)}^2_\bH+\norm{\bT(t)}^8_\bV)\\\leq C(1+\sup_{0\leq t\leq \tau}\norm{\bT(t)}^8_\bV)\leq C(1+M^8(R)),\quad\forall\,\bTi\in\pmb B_\bV(R),\,\forall\,t\in[0,\tau_R].
\end{multline*}
Therefore, $\pmb g(t):=\bG_0(\bT(t))$ is in $ L^p(0,\tau_R;\bH)$ for all $p\in[1,+\infty]$. By maximal regularity (see \cite{PC-Vespri-1986}) for $p=6$ we have that
\begin{equation*}
    \int_0^{\tau_R} \left(\norm{\pmb Z'(t)}^6_\bH+\norm{\bA\pmb Z(t)}^6_\bH\right)\, dt\leq C\int_0^{\tau_R} \norm{\pmb g(t)}^6_\bH\,dt.
\end{equation*}
Since $\pmb Z(0)=\pmb 0$, from the above estimate it also follows that
\begin{equation*}
    \int_0^{\tau_R} \norm{\pmb Z(t)}_{\bH}^6\leq C\int_0^{\tau_R}\norm{\pmb g(t)}^6_\bH dt.
\end{equation*}
Therefore,
\begin{equation*}
    \int_0^{\tau_R} \norm{\pmb Z(t)}^6_{\bDA}dt\leq C\int_0^{\tau_R} \norm{\pmb g(t)}^6_\bH dt\leq C(R),
\end{equation*}
and we conclude that
\begin{equation*}
    \int _0 ^{\tau_R} t^6 \norm {S(t)\bTi} _{\pmb L^\infty (I)} ^6 \, dt\leq \int _0 ^{\tau_R} t^6 (\norm{\pmb Y(t)}_{\pmb L^\infty (I)} ^6+\norm{\pmb Z(t)}^6_{\pmb L^\infty (I)}) \,  \leq C(R).
\end{equation*}
Note that the constant’s value may change from line to line; here, we are only concerned with its dependence on $R$.
\end{proof}
\begin{proof}[Proof of Proposition \ref{prop-borne-t_1-D(A)}]

Let $\bTi\in \pmb B_\bV^+(R)$ and $\pmb{Y}(\cdot)$, $\pmb Z(\cdot)$ be defined as in \eqref{eq:sol with Y and Z}. 

From \eqref{eq:estim Y_H} and \eqref{eq:estim Y_DA} it follows that
\begin{equation}
\label{eq:estim Y D(A)}
        \norm{\pmb Y(t)}_{\bDA}\leq C(1+\frac{1}{t})R,\qquad\forall\,t>0,\,\forall\,\bTi\in \pmb B_\bV^+(R).
\end{equation}
Recall that $\pmb Z$ is the solution to the following problem
\begin{equation}
\label{eq: C-prb Z}
    \begin{cases}
        \pmb Z'(t)=\bA \pmb Z(t)+\pmb g(t),& t\geq0\\
        \pmb Z(0)=\pmb 0,
    \end{cases}
\end{equation}
with $\pmb g(t)=\bG_0(S(t)\bTi)$. Our aim is to prove that 
\begin{equation*}
    t\mapsto t^3 \pmb Z(t)\quad\text{is of class}\quad C([0,\tau_R];\bDA),
\end{equation*}
with $\tau_R$ defined in Proposition \ref{thm-loc-wellpos}. Since $\pmb g\in L^2(0,\tau_R;\bH)$, we also have that $\pmb Z\in C([0,\tau_R];\pmb V)$. Let us set
\begin{equation*}
    \pmb W(t)=t^3(\pmb I-\bA)^{1/2}\pmb Z(t),\qquad t\in[0,\tau_R],
\end{equation*}
and observe that $\pmb{W}$ solves the following Cauchy problem
\begin{equation}
\label{eq:problem W2}
    \begin{cases}
        \pmb W'(t)=\bA\pmb W(t)+\pmb h(t), &t\in(0,\tau_R)\\
        \pmb W(0)=\pmb 0,
    \end{cases}
\end{equation}
with 
\begin{equation*}
    \pmb h(t)=t^3(\pmb I-\bA)^{1/2}\pmb g(t)+3t^2(\pmb I -\bA)^{1/2}\pmb Z(t),\qquad t\in (0,\tau_R).
\end{equation*}
We note that
\begin{equation}
\label{eq:estim h}
    \norm{\pmb h(t)}^2_\bH\leq C\left(t^6\norm{\pmb g(t)}^2_\bV+t^4\norm{\pmb Z(t)}^2_\bV\right).
\end{equation}
The first term on the right-side of the above inequality can be bounded as follows
\begin{multline*}
\norm{\pmb g(t)}_\bV ^2
\leq C \left( 1+ \Vert T_a (t) \Vert _{L^2(I)} ^2 + \Vert T_s (t) \Vert _{L^2(I)} ^2 + \Vert T_a (t) \Vert _{L^8(I)} ^8 + \Vert T_s (t) \Vert _{L^8(I)} ^8 \right)
\\
\quad+ C \int _I (1-x^2) \left( \left(\frac{\partial T_a}{\partial x}\right) ^2 + \left(\frac{\partial T_s}{\partial x}\right) ^2 + T_a ^6 \left(\frac{\partial T_a}{\partial x}\right) ^2 + T_s ^6 \left(\frac{\partial T_s}{\partial x}\right) ^2 \right) \, dx 
\\
\leq C \left( 1+ \norm{S(t)\bTi}_\bV ^2  
+ \norm{S(t)\bTi}_{\pmb L^8(I)} ^8 \right)
\\+ C \int _I (1-x^2) \left( T_a ^6 \left(\frac{\partial T_a}{\partial x}\right) ^2 + T_s ^6 \left(\frac{\partial T_s}{\partial x}\right) ^2 \right) \, dx.
\end{multline*}
Since $S(t)\bTi$ is uniformly bounded in $\bV$ on $[0,\tau_R]$ due to Proposition \ref{thm-loc-wellpos}, recalling the continuous embedding $\bV\hookrightarrow{\pmb L^8}(I)$ we deduce that
\begin{multline*}
    t^6\norm{\pmb g(t)} _{\bV} ^2
\leq Ct^6\left(1+M^2(R)+M^8(R)\right)\\
+t^6 \norm{S(t)\bTi}^6_{\pmb L^\infty(I)} M^2(R),\qquad\forall\,t\in[0,\tau_R],\,\forall\, \bTi\in \pmb B^+_\bV(R).
\end{multline*}
Therefore, by applying Lemma \ref{lem-maj-T-infty} we obtain
\begin{equation*}
    \int_0^{\tau_R}t^6\norm{\pmb g(t)}^2_\bV dt\leq C(R).
\end{equation*}
Plugging the above estimate in \eqref{eq:estim h} and recalling that $\pmb Z\in C([0,\tau_R];\pmb V)$, we conclude that $\pmb h(t)\in L^2(0,\tau_R;\bH)$. Thus, \cite[Theorem 3.1, p. 143]{bensoussan1992} yields the following regularity for the solution of \eqref{eq:problem W2}
\begin{equation*}
   \pmb W(t)\in C([0,\tau_R];\bV), 
\end{equation*}
that in turns implies
\begin{equation*}
    t^3\pmb Z(t)\in C([0,\tau_R],\bDA).
\end{equation*}
Therefore,
\begin{equation}
    \sup_{t\in[0,\tau_R]}\norm{t^3\pmb Z(t)}_{\bDA}\leq \sup_{t\in[0,\tau_R]}\norm{\pmb W(t)}_\bV\leq C\norm{\pmb h}_{L^2((0,\tau_R),\bH)}\leq C(R).
\end{equation}
Together with \eqref{eq:estim Y D(A)}, this implies that
\begin{equation*}
    \norm{S(\tau_R)\bTi}_{\bDA}\leq C(R),\qquad\forall\,\bTi\in\pmb B_\bV^+(R).\qedhere
\end{equation*}
\end{proof}
\subsubsection{Construction of the global attractor\vspace{.2cm}}\label{subsec: constr-glob-att}\hfill

The following is the last item needed to prove Theorem \ref{thm-attractor}.
\begin{Lemma}
\label{lem-compacte}
The injection of $\bDA$ into ${\bV}$ is compact.
\end{Lemma}
\begin{proof}
The result follows from the continuous embedding $D(A)\hookrightarrow V$, from item (iii) of Proposition \ref{resu-prop-V} and the interpolation inequality
$$\exists\,C>0\,:\,\forall\, u \in D(A), \quad \Vert u \Vert _V \leq C \, \Vert u \Vert _H ^{1/2} \, \Vert u \Vert _{D(A)} ^{1/2} .$$
\end{proof}
We now have all the tools needed to prove the existence of the global attractor.

\begin{proof}[Proof of Theorem \ref{thm-attractor}]

Let $R>0$. By Proposition \ref{prop-borne-t_1-D(A)} there exists $\tau_R >0$ such that $S(\tau_R) \bB_\bV^+(R)\subseteq \bB^+_{\bDA}(C(R))$. Thus, applying Proposition \ref{prom-absorb-bounded}, we infer the existence of $t_1=t_1(R)>0$ such that
\begin{equation}
\label{eq:t_1}
    \forall\, t \geq t_1, \quad S(t) S(\tau_R) \bB_\bV^+(R)  \subset \pmb U,
\end{equation}
where $\pmb U$ defined in \eqref{eq: def U}.

Appealing to Proposition \ref{prop-borne-t_1-D(A)} once again, there exists $t_0>0$ (which depends on $\pmb U$, but not on $R$) such that
$S(t_0) \pmb U$ is bounded in $\bDA$. Therefore, 
\begin{equation}
\label{eq:t_0}
    S(t_0) \pmb U\Subset \bV_+
\end{equation}
by Lemma \ref{lem-compacte}. So, from \eqref{eq:t_1} and \eqref{eq:t_0} it follows that
$$\forall\, t \geq t_0  + t_1 + \tau_R, \quad S(t) \bB_\bV^+(R)=S(t_0)S(t-t_0-\tau_R)S(\tau_R) \bB_\bV^+(R)\subset S(t_0) \pmb U\Subset \bV_+ .$$
Finally, invoking \cite[Theorem 1.1]{Temam}, we conclude that the $\omega$-limit set of $S(t_0) \pmb U$ is a compact, connected global attractor in $\bV_+$. \footnote{Note that the entire conclusion of \cite[Theorem 1.1]{Temam})---including connectedness---applies to our problem although we are working in a metric space, because $\bV_+$ is convex and $S(t)$ is uniformly compact for $t$ sufficiently large.}
\end{proof}


\section{Concluding remarks}
\label{sec:conclusions}

In this work, we provided a rigorous analysis of a two-layer Energy Balance Model with degenerate diffusion, focusing on well-posedness, qualitative properties, and long-time behavior. We proved positivity of solutions, invariant regions, global existence, and the existence of a global attractor for $\varepsilon_a\in(0,2)$, as well as finite-time blow-up for $\varepsilon_a>2$.

Compared with the classical results in \cite{Hetzer-Schmidt1991}, our analysis establishes global existence and the presence of a global attractor under the weaker condition $\varepsilon_a \in (0,2)$. This demonstrates that the two-layer model remains mathematically well behaved over a broader parameter range. Furthermore, a possible physical interpretation can be drawn by relating this extended range to regimes of increasing atmospheric opacity. While values $\varepsilon_a \leq 1$ correspond to physically consistent atmospheric single-layer descriptions, the interval $\varepsilon_a \in (1,2)$ may be viewed as an effective parametrization of highly opaque atmospheres, still compatible with a stable radiative balance. In this perspective, the persistence of global attractors suggests that even strong greenhouse conditions do not necessarily lead to instability within the model. On the other hand, the critical threshold $\varepsilon_a = 2$ emerges as a tipping point beyond which the mathematical dynamics lose global regularity, reflecting the onset of a runaway greenhouse regime. This reinforces the interpretation of $\varepsilon_a = 2$ as a meaningful boundary between stable and unstable climate configurations, rather than a purely technical limitation of the analysis.

A central outcome is the existence of a \emph{warmest} stationary solution. While extremal equilibria are classical in monotone dynamical systems, here this solution has a clear modeling interpretation as the climate state corresponding to the highest admissible temperature compatible with the energy balance, and may be viewed as a mathematical analogue of the present climatic regime.

Several natural questions remain open. In particular, it is unclear whether the warmest stationary solution is asymptotically stable, possibly in a strong or exponential sense. Addressing this would require a refined spectral or nonlinear stability analysis in a setting complicated by degenerate diffusion and strong nonlinear couplings.

Another key issue concerns the dependence of the warmest stationary solution on the emissivity parameter $\varepsilon_a$. For the spatially homogeneous ODE system \cite{CLMUV-EDO}, this equilibrium is asymptotically stable and depends monotonically and analytically on $\varepsilon_a$, providing a rigorous formulation of the greenhouse effect. Whether similar properties hold in the PDE setting remains an open and challenging problem.

More broadly, extending qualitative results from the ODE model to the spatially heterogeneous PDE framework raises further questions, such as the structure of the global attractor, the possible existence of multiple stationary states, and the role of spatial heterogeneities in long-term dynamics. The analytical framework developed here provides a basis for these investigations and for future studies of more complex climate models.

\bigskip

\textbf{Acknowledgments}. The authors would like to thank Jochen Br\"ocker, Giulia Carigi and Tobias Kuna for fruitful discussions.

This research was supported by Istituto Nazionale di Alta Matematica - GNAMPA Research Projects - and by the MIUR Excellence Department Project awarded to the Department of Mathematics, University of Rome Tor Vergata, CUP E83C18000100006. 

Piermarco Cannarsa was also supported by the PRIN 2022 PNRR Project P20225SP98 ``Some mathematical approaches to climate change and its impacts"  (funded by the European Community-Next Generation EU), CUP E53D2301791 0001.

P. Martinez and J. Vancostenoble gratefully acknowledge the Istituto Nazionale di Alta Matematica (INDAM) and the University of Roma Tor Vergata for funding research stays during which part of this work was conducted. They were also partially supported by the French National Research Agency (ANR) through the TRECOS project (grant ANR‑20‑CE40‑0009) and the LabEx CIMI – Centre International de Mathématiques et d'Informatique (grant ANR‑11‑LABX‑0040) within the French programme \lq\lq Investissements d'Avenir".

V. Lucarini acknowledges the partial support provided by the Horizon Europe Projects ClimTIP (Grant No. 100018693) and Past2Future (Grant No. 101184070)., by ARIA SCOP-PR01-P003 - Advancing Tipping Point Early Warning AdvanTip, by the European Space Agency project PREDICT (contract 4000146344/24/I-LR), and by the NNSFC  International Collaboration Fund for Creative Research Teams (Grant No. W2541005).


\appendix
\section{Proof of Lemma \ref{lem-propr-G}}
\label{app: A}

In order to prove that $\bG$ is well-defined from $[0,\TT]\times \bV$ to $ \bH$, we choose any $t\in[0,\TT]$ and any $\pmb{u}\in  \bV$ and we study the $ \bH$-norm of $\bG(t,\pmb{u})$. First, we observe that the $\bH $-norm of $\bG(t,\pmb{u})$ is finite since $V$ is continuously embedded in $L^q(I)$ for all $1\leq q<+\infty$ (see (2) of Proposition \ref{resu-prop-V}). 

The continuity of $\bG: [0,\TT] \times \bV \to \bH$ will follow from   
\eqref{propr2-G} and \eqref{propr3-G}. First, we prove \eqref{propr2-G}: for every $t,s\in[0,\TT]$ and $\pmb{u}\in \bV$, we have that
\begin{equation*}
\begin{split}
& \norm{\bG(t,\pmb{u})-\bG(s,\pmb{u})}^2_{\bH}
\\& \quad 
=\int_I|r(t)-r(s)|^2q(x)^2\beta_a(u_a)^2dx
+\int_I|r(t)-r(s)|^2q(x)^2\beta_s(u_s)^2dx
\\& \quad
\leq 
2 \left( \norm{r}^2_{Lip}\norm{q}^2_{L^\infty(I)}\norm{\beta_a}^2_{L^\infty(\RR)}
+
\norm{r}^2_{Lip}\norm{q}^2_{L^\infty(I)}\norm{\beta_s}^2_{L^\infty(\RR)}\right)
|t-s|^2,
\end{split}
\end{equation*}
where $\norm{\cdot}_{Lip}$ denotes the Lipschitz semi-norm in $(0+\infty)$. Then, \eqref{propr2-G} is verified.

Next, we prove \eqref{propr3-G}: let $t\in[0,\TT]$ and consider $\pmb{u},\pmb{\tilde u}\in \bV$. Then, we have
\begin{equation*}
\begin{split}
&\norm{\bG(t,\pmb{u})-\bG(t,\pmb{\tilde u})}_{\bH}^2
\\
&\qquad
=\frac{1}{\gamma_a^2}\int_I
\Big|-\lambda \left((u_a-u_s)-(\tilde u_a-\tilde u_s)\right)
+\varepsilon_a\sigma_B\left(|u_s|^3u_s-|\tilde u_s|^3|\tilde u_s\right)
\\
&\qquad\qquad \qquad\qquad
-2\varepsilon_a\sigma_B\left(|u_a|^3u_a-|\tilde u_a|^3\tilde u_a\right) +r(t)q\left(\beta_a(u_a)-\beta_a(\tilde u_a)\right)\Big|^2dx
\\
&\qquad\quad
+\frac{1}{\gamma_s^2}\int_I
\Big|-\lambda\left((u_s-u_a)-(\tilde u_s-\tilde u_a)\right)
-\sigma_B\left(|u_s|^3u_s-|\tilde u_s|^3\tilde u_s\right)\\
&\qquad\qquad \qquad\qquad
+\varepsilon_a\sigma_B\left(|u_a|^3u_a-|\tilde u_a|^3\tilde u_a\right)+r(t)q\left(\beta_s(u_s)-\beta_s(\tilde u_s)\right)\Big|^2dx
.\end{split}
\end{equation*}
Since $r, q \in L^\infty$, and $\beta _a, \beta_s$ are globally Lipschitz, it is easy to see that
\begin{multline*}
\norm{\bG(t,\pmb{u})-\bG(t,\pmb{\tilde u})}_{\bH}^2
\leq \\
C \left( \norm{\pmb{u} - \pmb{\tilde u}}_{\bH}^2 
+ \int_I \left||u_a|^3u_a-|\tilde u_a|^3\tilde u_a\right|^2 \, dx
+ \int_I \left||u_s|^3u_s-|\tilde u_s|^3\tilde u_s\right|^2 \, dx
\right) .
\end{multline*}
Moreover,
\begin{equation*}
\begin{split}
& \int_I\left||u_s|^3u_s-|\tilde u_s|^3\tilde u_s\right|^2dx
= \int_I
\left|  |u_s|^3 (u_s-\tilde u_s) + \tilde u_s (|u_s|^3-|\tilde u_s|^3)
 \right|^2dx
\\
& \quad 
\leq
2\int_I \left[ |u_s|^6|u_s-\tilde u_s|^2 \right]
+ \left[ |\tilde u_s|^2 \left(|u_s|^2+|u_s||\tilde u_s|+|\tilde u_s|^2\right)^2
\left| |u_s|-|\tilde u_s| \right|^2 \right] \, dx
\\
& \quad 
\leq 2\norm{u_s}_{L^{12}(I)}^6\norm{u_s-\tilde u_s}_{L^4(I)}^2
+ C \int_I \left(|u_s|^6+|\tilde u_s|^6\right)
|u_s-\tilde u_s|^2 \, dx
\\
& \quad 
\leq C' \left( \norm{u_s}_{V}^6 + \norm{\tilde u_s}_{V}^6 \right) \norm{u_s-\tilde u_s}_{V}^2 ,
\end{split}
\end{equation*}
where, in the last inequality, we have used that $V$ is continuously embedded in $L^p(I)$ for all $p\in [1,+\infty)$. This proves \eqref{propr3-G}.

\section{Proof of Proposition \ref{prop-reg-borne}}

The proof of Proposition \ref{prop-reg-borne} follows from the following Hardy-type inequality, adapted from \cite{JV-DCDS2011}:

\begin{Lemma}
\label{lem-Hardy}
Let $n>0$ and $\gamma \in (0,1)$. Then, there exists a constant $C_{n,\gamma}>0$ such that for all $v \in V$, we have
\begin{equation}
\label{eq-Hardy}
n \int _{-1} ^1 \frac{v(x)^2}{(1-x^2)^{\gamma}} \, dx \leq \int _{-1} ^1 (1-x^2) v'(x) ^2 \, dx + C_{n,\gamma} \int _{-1} ^1 v(x)^2 \, dx .
\end{equation}
\end{Lemma}
\begin{proof}
The argument of \cite{JV-DCDS2011} cannot be directly applied because the operator at hand does not satisfy Dirichlet boundary conditions. However, we can follow the same strategy. Let $a \in [-\frac{1}{2},0]$. Then, we have
\begin{multline*}
0 \leq \int _{-1} ^a \left( \sqrt{1-x^2} v'(x) + \frac{v(x)}{(1-x^2)^{\gamma /2}} \right)^2 \, dx
\\
= \int _{-1} ^a (1-x^2) v'(x) ^2 + \frac{v(x)^2}{(1-x^2)^{\gamma }} + (1-x^2)^{(1-\gamma) /2} (v(x)^2)' \, dx .
\end{multline*}
Integrating by parts the last term, we get
\begin{multline*}
\int _{-1} ^a (1-x^2)^{(1-\gamma) /2} (v(x)^2)' \, dx
\\
=
\left[ (1-x^2)^{(1-\gamma) /2} v(x)^2 \right] _{-1} ^a 
- \frac{1-\gamma}{2} \int _{-1} ^a -2x (1-x^2)^{(-1-\gamma) /2} v(x)^2 \, dx .
\end{multline*}
Therefore
$$
0 \leq \int _{-1} ^a (1-x^2) v'(x) ^2 + \left( \frac{1}{(1-x^2)^{\gamma }}
+ \frac{x(1-\gamma)}{(1-x^2)^{(1+\gamma) /2}} \right) v(x)^2 \, dx 
+ (1-a^2)^{(1-\gamma) /2} v(a)^2,
$$
and
\begin{multline*}
n \int _{-1} ^a \frac{v(x)^2}{(1-x^2)^{\gamma }} \, dx 
\leq (1-a^2)^{(1-\gamma) /2} v(a)^2
\\
+ \int _{-1} ^a 
\left\{ (1-x^2) v'(x) ^2  + \left( \frac{n+1}{(1-x^2)^{\gamma }}
+ \frac{x(1-\gamma)}{(1-x^2)^{(1+\gamma) /2}} \right) v(x)^2  \right\}
\, dx .
\end{multline*}
We note that the function
$$ x \mapsto \frac{n+1}{(1-x^2)^{\gamma }}
+ \frac{x(1-\gamma)}{(1-x^2)^{(1+\gamma) /2}} $$
is continuous on $(-1,0]$ and tends to $-\infty$ as $x\to (-1)^+$ (since $\gamma <1$). Setting
$$ c_{n,\gamma}:= \sup _{(-1,0]} \frac{n+1}{(1-x^2)^{\gamma }}
+ \frac{x(1-\gamma)}{(1-x^2)^{(1+\gamma) /2}} ,$$
we deduce that, for all $a \in [-\frac{1}{2},0]$ :
\begin{multline*}
n \int _{-1} ^{-1/2} \frac{v(x)^2}{(1-x^2)^{\gamma }} \, dx 
\leq  n \int _{-1} ^a \frac{v(x)^2}{(1-x^2)^{\gamma }} \, dx 
\\
\leq (1-a^2)^{(1-\gamma) /2} v(a)^2
+ \int _{-1} ^a 
\left\{ (1-x^2) v'(x) ^2  + c_{n,\gamma} v(x)^2 \right\} \, dx \\
\leq (1-a^2)^{(1-\gamma) /2} 
 v(a)^2
+ \int _{-1} ^0 \left\{ (1-x^2) v'(x) ^2  + c_{n,\gamma} v(x)^2 \right\} \, dx
.
\end{multline*}
Integrating with respect to $a$ over $[-\frac{1}{2},0]$, we obtain that
$$ \frac{n}{2} \int _{-1} ^{-1/2} \frac{v(x)^2}{(1-x^2)^{\gamma }} \, dx 
\leq \int _{-1/2} ^0 
(1-a^2)^{(1-\gamma) /2} v(a)^2 \, da
+ \frac{1}{2} \int _{-1} ^0 
\left\{(1-x^2) v'(x) ^2  + c_{n,\gamma} v(x)^2 \right\}
\, dx .$$
Therefore
\begin{equation}
\label{eq-Hardy-}
 n \int _{-1} ^{-1/2} \frac{v(x)^2}{(1-x^2)^{\gamma }} \, dx 
\leq (2 + c_{n,\gamma})\int _{-1} ^0 v(x)^2 \, dx
+ \int _{-1} ^0 (1-x^2) v'(x) ^2 \, dx.
\end{equation}
Fix now $b \in [0,\frac{1}{2}]$. Then, we have
$$
0 \leq \int _{b} ^1 \left( \sqrt{1-x^2} v'(x) - \frac{v(x)}{(1-x^2)^{\gamma /2}} \right)^2 \, dx ,
$$
and integrating with respect to $b$ over $[0,\frac{1}{2}]$, we obtain that 
\begin{equation}
\label{eq-Hardy+}
n \int _{1/2} ^1 \frac{v(x)^2}{(1-x^2)^{\gamma }} \, dx 
\leq (2 + c' _{n,\gamma})\int _{0} ^1 v(x)^2 \, dx
+ \int _{0} ^1 (1-x^2) v'(x) ^2 \, dx ,
\end{equation}
where 
$$ c' _{n,\gamma} = \sup _{[0,1)} \frac{n+1}{(1-x^2)^{\gamma }}
- \frac{x(1-\gamma)}{(1-x^2)^{(1+\gamma) /2}} .$$
Then \eqref{eq-Hardy-} and \eqref{eq-Hardy+} yield
$$ n \int _{-1} ^{1} \frac{v(x)^2}{(1-x^2)^{\gamma }} \, dx 
\leq (4 + c _{n,\gamma} + c' _{n,\gamma} + \frac{4}{3}n ) \int _{-1} ^1 v(x)^2 \, dx
+ \int _{-1} ^1 (1-x^2) v'(x) ^2 \, dx .$$
This concludes the proof of  \eqref{eq-Hardy}. 
\end{proof}
We now prove Proposition \ref{prop-reg-borne}.
\begin{proof}[Proof of Proposition \ref{prop-reg-borne}] Let $u \in D(A)$, and fix $\gamma <1$. Then
\begin{multline}\label{Z1}
\int _{-1} ^1 \vert (u (x)^2)' \vert \, dx
=  2 \int _{-1} ^1 \left \vert \frac{u (x)}{(1-x^2)^{\gamma /2}} \right \vert \, \vert (1-x^2)^{\gamma /2} u' (x) \vert\, dx
\\ 
\leq \int _{-1} ^1  \frac{u (x)^2}{(1-x^2)^{\gamma}} \, dx 
+ \int _{-1} ^1 (1-x^2)^{\gamma} u' (x) ^2 \, dx .
\end{multline}
We use Lemma \ref{lem-Hardy} to bound the first term: since $u\in V$, we have
\begin{equation}\label{Z2}
\int _{-1} ^1 \frac{u(x)^2}{(1-x^2)^{\gamma}} \, dx \leq \int _{-1} ^1 (1-x^2) u'(x) ^2 \, dx + C_{1,\gamma} \int _{-1} ^1 u(x)^2 \, dx \leq C \Vert u \Vert _{V} ^2 .
\end{equation}
For the second one, we observe that
$$ \int _{-1} ^1 (1-x^2)^{\gamma} u' (x) ^2 \, dx
= \int _{-1} ^1 (1-x^2)^{\gamma -2} \left( (1-x^2) u' (x) \right) ^2 \, dx 
=\int _{-1} ^1 (1-x^2)^{\gamma -2} w(x) ^2 \, dx  
$$
where we denote 
$$ w(x) := (1-x^2) u' (x) .$$
Since $u \in D(A)$, we know that 
$w \in H^1_0(I)$ and one can deduce that 
\begin{equation}\label{appl-Hardy-class}
\int _{-1} ^1 \frac{w(x)^2}{(1-x^2)^{2}}  \, dx
\\ 
\leq C_H \int _{-1} ^1  w' (x) ^2 \, dx,
\end{equation}
for some $C_H>0$. 
Indeed the classical Hardy inequality ensures that all $W \in H^1 _0(0,1)$: 
\begin{equation}\label{Hardy-class}
 \int_0^1 \frac{W(x)^2}{x^2}dx \leq \frac{1}{4} \int_0 ^1  W'(x)^2 dx .
\end{equation}
By a  change of variables, one obtain \eqref{appl-Hardy-class} from \eqref{Hardy-class}.

Then, applying \eqref{appl-Hardy-class} to $w(x) := (1-x^2) u' (x)$ and using  the fact that  $ \gamma >0$, we obtain 
$$
\int _{-1} ^1 (1-x^2)^{\gamma -2} w (x) ^2 \, dx 
\leq \int _{-1} ^1 (1-x^2)^{-2} w (x) ^2 \, dx 
 \leq C_H \int _{-1} ^1 w'(x)  ^2 \, dx 
.
$$
It follows that 
\begin{multline}\label{Z3}
\int _{-1} ^1 (1-x^2)^{\gamma}   u' (x)  ^2 \, dx
= \int _{-1} ^1 (1-x^2)^{\gamma -2} w (x) ^2 \, dx
\\ 
\leq C_H \int _{-1} ^1 w'(x)  ^2 \, dx =
 C_H \int _{-1} ^1 \vert ((1-x^2) u' (x))'  \vert^2 \, dx
\leq C_H \Vert u \Vert _{D(A)} ^2 .
\end{multline}
From \eqref{Z1}, \eqref{Z2} and \eqref{Z3}, we deduce that there exists $C$, independent of $u$, such that
$$ \int _{-1} ^1 \vert (u (x)^2)' \vert \, dx \leq C \Vert u \Vert _{D(A)} ^2 .$$
This implies that $u^2 \in W^{1,1} (I)$. Consequently, $u^2$ is absolutely continuous on $[-1,1]$, and therefore continuous on $[-1,1]$. 

In particular, $u^2$ admits a finite limit as $x\to 1^-$. If this limit is zero, then necessarily $u (x) \to 0$ as $x\to 1^-$.  Otherwise, if the limit is strictly positive, then $u^2$ remains bounded away from zero in a neighborhood of $x=1$. By continuity, $u$ does not change sign near $x=1$, and hence admits a finite limit as $x\to 1^-$.

Moreover, since $u \in H^1 (-d,d)$ for every $d \in (0,1)$, it follows that $u$ is continuous on $I$. An analogous argument applies to the limit $x\to (-1)^+$.

Finally, we obtain the estimate 
$$ \Vert u \Vert _{\infty} ^2 = \Vert u ^2 \Vert _{\infty} \leq C \Vert u ^2 \Vert _{W^{1,1} (I)} \leq C'  \Vert u \Vert _{D(A)} ^2 ,$$
for suitable constants $C,C'>0$.

This completes the proof of Proposition \ref{prop-reg-borne}. 
\end{proof}

\bibliographystyle{plain}
\bibliography{biblio}

@PREAMBLE{
 "\providecommand{\noopsort}[1]{}" 
 # "\providecommand{\singleletter}[1]{#1}%" 
}

@article{Sheffrey2006,
      author = "Shaffrey, L. and Sutton, R.",
      title = "Bjerknes Compensation and the Decadal Variability of the Energy Transports in a Coupled Climate Model",
      journal = "Journal of Climate",
      year = "2006",
      publisher = "American Meteorological Society",
      address = "Boston MA, USA",
      volume = "19",
      number = "7",
      doi = "10.1175/JCLI3652.1",
      pages=      "1167 - 1181",
      url = "https://journals.ametsoc.org/view/journals/clim/19/7/jcli3652.1.xml"
}

@article{Stone1978,
   author    =  "P. H. Stone",
   title     =  "Constraints on dynamical transports of energy on a spherical planet",
   year      =  "1978",
   journal   =  "Dyn. Atmos. Oceans",
   volume    =  "2",
   pages     =  "123--139"}

@incollection{Bjerknes1964,
title = {Atlantic Air-Sea Interaction},
editor = {H.E. Landsberg and J. {Van Mieghem}},
series = {Advances in Geophysics},
publisher = {Elsevier},
volume = {10},
pages = {1-82},
year = {1964},
issn = {0065-2687},
doi = {https://doi.org/10.1016/S0065-2687(08)60005-9},
url = {https://www.sciencedirect.com/science/article/pii/S0065268708600059},
author = {J. Bjerknes},
}

@article{Outten2018,
      author = "Outten, S. and Esau, I. and Otterå, O. H.",
      title = "Bjerknes Compensation in the CMIP5 Climate Models",
      journal = "Journal of Climate",
      year = "2018",
      publisher = "American Meteorological Society",
      address = "Boston MA, USA",
      volume = "31",
      number = "21",
      doi = "10.1175/JCLI-D-18-0058.1",
      pages=      "8745 - 8760",
      url = "https://journals.ametsoc.org/view/journals/clim/31/21/jcli-d-18-0058.1.xml"
}

@article{Lucarini2024,
  title = {Detecting and Attributing Change in Climate and Complex Systems: Foundations, Green's Functions, and Nonlinear Fingerprints},
  author = {Lucarini, V. and Chekroun, M. D.},
  journal = {Phys. Rev. Lett.},
  volume = {133},
  issue = {24},
  pages = {244201},
  numpages = {11},
  year = {2024},
  month = {Dec},
  publisher = {American Physical Society},
  doi = {10.1103/PhysRevLett.133.244201},
  url = {https://link.aps.org/doi/10.1103/PhysRevLett.133.244201}
}

@article{Ramirez2024,
doi = {10.3847/PSJ/ad0729},
url = {https://dx.doi.org/10.3847/PSJ/ad0729},
year = {2024},
month = {jan},
publisher = {The American Astronomical Society},
volume = {5},
number = {1},
pages = {2},
author = {Ramirez, R. M.},
title = {A New 2D Energy Balance Model for Simulating the Climates of Rapidly and Slowly Rotating Terrestrial Planets},
journal = {The Planetary Science Journal},
}

@article{HaqqMisra2022,
doi = {10.3847/PSJ/ac49eb},
url = {https://dx.doi.org/10.3847/PSJ/ac49eb},
year = {2022},
month = {feb},
publisher = {The American Astronomical Society},
volume = {3},
number = {2},
pages = {32},
author = {Haqq-Misra, Jacob and Hayworth, Benjamin P. C.},
title = {An Energy Balance Model for Rapidly and Synchronously Rotating Terrestrial Planets},
journal = {The Planetary Science Journal},
}

@Article{Knietzsch2015,
AUTHOR = {Knietzsch, M.-A. and Schr\"oder, A. and Lucarini, V. and Lunkeit, F.},
TITLE = {The impact of oceanic heat transport on the atmospheric circulation},
JOURNAL = {Earth System Dynamics},
VOLUME = {6},
YEAR = {2015},
NUMBER = {2},
PAGES = {591--615},
URL = {https://esd.copernicus.org/articles/6/591/2015/},
DOI = {10.5194/esd-6-591-2015}
}

@article{AL-CAN-FRA,
    AUTHOR = {Alabau-Boussouira, F. and Cannarsa, P. and Fragnelli, G.},
     TITLE = {Carleman estimates for degenerate parabolic operators with
              applications to null controllability},
   JOURNAL = {J. Evol. Equ.},
  FJOURNAL = {Journal of Evolution Equations},
    VOLUME = {6},
      YEAR = {2006},
    NUMBER = {2},
     PAGES = {161--204},
      ISSN = {1424-3199,1424-3202},
   MRCLASS = {93B05 (35B45 35K20 35K65 93C20)},
  MRNUMBER = {2227693},
MRREVIEWER = {Angelo\ Favini},
       DOI = {10.1007/s00028-006-0222-6},
       URL = {https://doi.org/10.1007/s00028-006-0222-6},
}

@book {bensoussan1992,
    AUTHOR = {Bensoussan, A. and Da Prato, G. and Delfour, M. C. and Mitter, S. K.},
     TITLE = {Representation and control of infinite-dimensional systems.
              {V}ol. {II}},
    SERIES = {Systems \& Control: Foundations \& Applications},
 PUBLISHER = {Birkh\"auser Boston, Inc., Boston, MA},
      YEAR = {1993},
     PAGES = {xviii+345},
      ISBN = {0-8176-3642-0},
   MRCLASS = {49-02 (34H05 47D03 47N70 49J27 49K27 93C95)},
  MRNUMBER = {1246331},
MRREVIEWER = {Irena\ Lasiecka},
       DOI = {10.1007/978-1-4612-2750-2},
       URL = {https://doi.org/10.1007/978-1-4612-2750-2},
}

@article{Budyko,
  title={The effect of solar radiation variations on the climate of the Earth},
  author={Budyko, M. I.},
  journal={Tellus},
  volume={21},
  number={5},
  pages={611--619},
  year={1969},
  publisher={Taylor \& Francis}
}

@article {Campiti1998,
    AUTHOR = {Campiti, M. and Metafune, G. and Pallara, D.},
     TITLE = {Degenerate self-adjoint evolution equations on the unit
              interval},
   JOURNAL = {Semigroup Forum},
  FJOURNAL = {Semigroup Forum},
    VOLUME = {57},
      YEAR = {1998},
    NUMBER = {1},
     PAGES = {1--36},
      ISSN = {0037-1912,1432-2137},
   MRCLASS = {35K65 (34D05 47D06)},
  MRNUMBER = {1621852},
MRREVIEWER = {Paolo\ Acquistapace},
       DOI = {10.1007/PL00005959},
       URL = {https://doi.org/10.1007/PL00005959},
}

@article {CA-MA-UR,
    AUTHOR = {Cannarsa, P. and Martinez, P. and Urbani,
              C.},
     TITLE = {Bilinear control of a degenerate hyperbolic equation},
   JOURNAL = {SIAM J. Math. Anal.},
  FJOURNAL = {SIAM Journal on Mathematical Analysis},
    VOLUME = {55},
      YEAR = {2023},
    NUMBER = {6},
     PAGES = {6517--6553},
      ISSN = {0036-1410,1095-7154},
   MRCLASS = {93B05 (33C10 35L80 42C40 93C10 93C20)},
  MRNUMBER = {4662407},
MRREVIEWER = {Julian\ Edward},
       DOI = {10.1137/22M148745X},
       URL = {https://doi.org/10.1137/22M148745X},
}

@article {CPAA2004,
    AUTHOR = {Cannarsa, P. and Martinez, P. and Vancostenoble,
              J.},
     TITLE = {Persistent regional null controllability for a class of
              degenerate parabolic equations},
   JOURNAL = {Commun. Pure Appl. Anal.},
  FJOURNAL = {Communications on Pure and Applied Analysis},
    VOLUME = {3},
      YEAR = {2004},
    NUMBER = {4},
     PAGES = {607--635},
      ISSN = {1534-0392,1553-5258},
   MRCLASS = {93B05 (35B37 35K20 35K65 93C20)},
  MRNUMBER = {2106292},
MRREVIEWER = {Kim\ Dang\ Phung},
       DOI = {10.3934/cpaa.2004.3.607},
       URL = {https://doi.org/10.3934/cpaa.2004.3.607},
}

@article{CMV-note,
    author = {Cannarsa, P. and Martinez, P. and Vancostenoble, J.},
    title = {High order Hardy-type inequalities and optimal Sobolev embeddings for strongly degenerate elliptic operators},
    journal = {submitted},
}

@article {CLMUV-EDO,
    AUTHOR = {Cannarsa, P. and Lucarini, V. and Martinez, P. and Urbani, C.
              and Vancostenoble, J.},
     TITLE = {Analysis of a two-layer energy balance model: long time
              behavior and greenhouse effect},
   JOURNAL = {Chaos},
  FJOURNAL = {Chaos. An Interdisciplinary Journal of Nonlinear Science},
    VOLUME = {33},
      YEAR = {2023},
    NUMBER = {11},
     PAGES = {Paper No. 113111, 34},
      ISSN = {1054-1500,1089-7682},
   MRCLASS = {86A10},
  MRNUMBER = {4663335},
       DOI = {10.1063/5.0136673},
       URL = {https://doi.org/10.1063/5.0136673},
}

@incollection{Diaz1997,
  title={On the mathematical treatment of energy balance climate models},
  author={D{\'\i}az, J. I.},
  booktitle={The mathematics of models for climatology and environment},
  pages={217--251},
  year={1997},
  publisher={Springer}
}

@book {Fife,
    AUTHOR = {Fife, P. C.},
     TITLE = {Mathematical aspects of reacting and diffusing systems},
    SERIES = {Lecture Notes in Biomathematics},
    VOLUME = {28},
 PUBLISHER = {Springer-Verlag, Berlin-New York},
      YEAR = {1979},
     PAGES = {iv+185},
      ISBN = {3-540-09117-3},
   MRCLASS = {35-02 (92A15)},
  MRNUMBER = {527914},
MRREVIEWER = {C.\ V.\ Pao},
}

@article {Ghil76,
    AUTHOR = {Ghil, M.},
     TITLE = {Climate stability for a {S}ellers-type model},
   JOURNAL = {Journal of the Atmospheric Sciences},
    VOLUME = {33},
      YEAR = {1976},
    NUMBER = {1},
     PAGES = {3--20}
}

@article {GL2020,
    AUTHOR = {Ghil, M. and Lucarini, V.},
     TITLE = {The physics of climate variability and climate change},
   JOURNAL = {Reviews of Modern Physics},
    VOLUME = {92},
      YEAR = {2020},
    NUMBER = {3},
     PAGES = {035002, 77}
}

@book{Hartmann,
  title={Global physical climatology, 2nd edn},
  author={Hartmann, D. L.},
  volume={103},
  year={2016},
  publisher={Newnes}
}

@article{Hoffman2000,
  title={The snowball Earth hypothesis: testing the limits of global change},
  author={Hoffman, P. F and Schrag, D. P.},
  journal={Terra nova},
  volume={14},
  number={3},
  pages={129--155},
  year={2002},
  publisher={Wiley Online Library}
}

@book {lions1968problemes,
    AUTHOR = {Lions, J.-L. and Magenes, E.},
     TITLE = {Probl\`emes aux limites non homog\`enes et applications.
              {V}ol. 1},
    SERIES = {Travaux et Recherches Math\'ematiques},
    VOLUME = {No. 17},
 PUBLISHER = {Dunod, Paris},
      YEAR = {1968},
     PAGES = {xx+372},
   MRCLASS = {35.00 (46.00)},
  MRNUMBER = {247243},
MRREVIEWER = {R.\ S.\ Freeman},
}

@article{Lucarini2014,
  title={Mathematical and physical ideas for climate science},
  author={Lucarini, V. and Blender, R. and Herbert, C. and Ragone, F. and Pascale, S. and Wouters, J.},
  journal={Reviews of Geophysics},
  volume={52},
  number={4},
  pages={809--859},
  year={2014},
  publisher={Wiley Online Library}
}

@article {Lucarini-Bodai2017,
    AUTHOR = {Lucarini, V. and B\'{o}dai, T.},
     TITLE = {Edge states in the climate system: exploring global
              instabilities and critical transitions},
   JOURNAL = {Nonlinearity},
    VOLUME = {30},
      YEAR = {2017},
    NUMBER = {7},
     PAGES = {R32--R66}
}

@article {Lucarini-Bodai2020,
    AUTHOR = {Lucarini, V. and B\'{o}dai, T.},
     TITLE = {Global stability properties of the climate: melancholia
              states, invariant measures, and phase transitions},
   JOURNAL = {Nonlinearity},
    VOLUME = {33},
      YEAR = {2020},
    NUMBER = {9},
     PAGES = {R59--R92}
}

@book {lunardi1995analytic,
    AUTHOR = {Lunardi, A.},
     TITLE = {Analytic semigroups and optimal regularity in parabolic
              problems},
    SERIES = {Modern Birkh\"auser Classics},
      NOTE = {[2013 reprint of the 1995 original] [MR1329547]},
 PUBLISHER = {Birkh\"auser/Springer Basel AG, Basel},
      YEAR = {1995},
     PAGES = {xviii+424},
      ISBN = {978-3-0348-0556-8; 978-3-0348-0557-5},
   MRCLASS = {47D06 (01A75 34G20 35Kxx 46M35 46N20 47N20 58D25)},
  MRNUMBER = {3012216},
}

@article{North81,
  title={Energy balance climate models},
  author={North, G. R. and Cahalan, R. F. and Coakley, J. A.},
  journal={Reviews of Geophysics},
  volume={19},
  number={1},
  pages={91--121},
  year={1981},
  publisher={Wiley Online Library}
}

@book{Peixoto1992,
  title={Physics of climate},
  author={Peixoto, J. P. and Oort, A. H. and Lorenz, E. N.},
  volume={520},
  year={1992},
  publisher={Springer}
}

@book {Pierrehumber2011,
    AUTHOR = {Pierrehumbert, R. T.},
     TITLE = {Principles of planetary climate},
 PUBLISHER = {Cambridge University Press, Cambridge},
      YEAR = {2010},
     PAGES = {xxvi+652}
}

@book {Protter-Weinberger,
    AUTHOR = {Protter, M. H. and Weinberger, H. F.},
     TITLE = {Maximum principles in differential equations},
 PUBLISHER = {Prentice-Hall, Inc., Englewood Cliffs, N.J.},
      YEAR = {1967},
     PAGES = {x+261}
}

@article {BSchmidt,
    AUTHOR = {Schmidt, B. E.},
     TITLE = {Bifurcation from {S}-shaped solution curves in a class of
              {S}turm-{L}iouville problems related to climate modeling},
   JOURNAL = {Adv. Math. Sci. Appl.},
  FJOURNAL = {Advances in Mathematical Sciences and Applications},
    VOLUME = {10},
      YEAR = {2000},
    NUMBER = {2},
     PAGES = {513--537},
      ISSN = {1343-4373},
   MRCLASS = {34B18 (34C23 34C60 35B20 35K57 86A10)},
  MRNUMBER = {1807440},
MRREVIEWER = {Junping\ Shi},
}

@article{sellers,
  title={A global climatic model based on the energy balance of the earth-atmosphere system},
  author={Sellers, W. D.},
  journal={Journal of Applied Meteorology (1962-1982)},
  pages={392--400},
  year={1969},
  publisher={JSTOR}
}

@book {Smith,
    AUTHOR = {Smith, H. L.},
     TITLE = {Monotone dynamical systems},
    SERIES = {Mathematical Surveys and Monographs},
    VOLUME = {41},
      NOTE = {An introduction to the theory of competitive and cooperative
              systems},
 PUBLISHER = {American Mathematical Society, Providence, RI},
      YEAR = {1995},
     PAGES = {x+174}
}

@book {Temam,
    AUTHOR = {Temam, R.},
     TITLE = {Infinite-dimensional dynamical systems in mechanics and
              physics},
    SERIES = {Applied Mathematical Sciences},
    VOLUME = {68},
 PUBLISHER = {Springer-Verlag, New York},
      YEAR = {1988},
     PAGES = {xvi+500},
      ISBN = {0-387-96638-2},
   MRCLASS = {58Fxx (35-02 47H20 58-02 58D25 76-02 76D05)},
  MRNUMBER = {953967},
MRREVIEWER = {J. E. Marsden},
       DOI = {10.1007/978-1-4684-0313-8},
       URL = {https://doi.org/10.1007/978-1-4684-0313-8},
}

@article {TORT2012683,
    AUTHOR = {Tort, J. and Vancostenoble, J.},
     TITLE = {Determination of the insolation function in the nonlinear
              {S}ellers climate model},
   JOURNAL = {Annales de l'Institut Henri Poincar\'{e} C. Analyse Non Lin\'{e}aire},
    VOLUME = {29},
      YEAR = {2012},
    NUMBER = {5},
     PAGES = {683--713}
}

@article {JV-DCDS2011,
    AUTHOR = {Vancostenoble, J.},
     TITLE = {Improved {H}ardy-{P}oincar\'e{} inequalities and sharp
              {C}arleman estimates for degenerate/singular parabolic
              problems},
   JOURNAL = {Discrete Contin. Dyn. Syst. Ser. S},
  FJOURNAL = {Discrete and Continuous Dynamical Systems. Series S},
    VOLUME = {4},
      YEAR = {2011},
    NUMBER = {3},
     PAGES = {761--790},
      ISSN = {1937-1632,1937-1179},
   MRCLASS = {35K67 (35A23 35K65 93B05 93B07)},
  MRNUMBER = {2746432},
MRREVIEWER = {Eduardo\ Cerpa},
       DOI = {10.3934/dcdss.2011.4.761},
       URL = {https://doi.org/10.3934/dcdss.2011.4.761},
}

@article {Yang-Deng,
    AUTHOR = {Yang, L. and Deng, Z.-C.},
     TITLE = {Uniqueness for an inverse source problem in degenerate
              parabolic equations},
   JOURNAL = {J. Math. Anal. Appl.},
  FJOURNAL = {Journal of Mathematical Analysis and Applications},
    VOLUME = {488},
      YEAR = {2020},
    NUMBER = {2},
     PAGES = {124095, 9},
      ISSN = {0022-247X,1096-0813},
   MRCLASS = {35R30 (35K20 35K65)},
  MRNUMBER = {4082539},
       DOI = {10.1016/j.jmaa.2020.124095},
       URL = {https://doi.org/10.1016/j.jmaa.2020.124095},
}

@article {PC-Vespri-1986,
    AUTHOR = {Cannarsa, P. and Vespri, V.},
     TITLE = {On maximal {$L^p$} regularity for the abstract {C}auchy
              problem},
   JOURNAL = {Boll. Un. Mat. Ital. B (6)},
  FJOURNAL = {Unione Matematica Italiana. Bollettino. B. Serie VI},
    VOLUME = {5},
      YEAR = {1986},
    NUMBER = {1},
     PAGES = {165--175},
   MRCLASS = {34G10},
  MRNUMBER = {841623},
MRREVIEWER = {Tom\'as\ Dom\'inguez Benavides},
}

@phdthesis{Fornasaro,
AUTHOR= {Fornasaro, F.},
TITLE={Well-posedness and long-time dynamics for a deterministic and stochastic quasi-geostrophic ocean-atmosphere model with heat exchange},
SCHOOL={Sapienza Università di Roma}, 
YEAR={2026},
}

@article {Hetzer-Schmidt1990,
    AUTHOR = {Hetzer, G. and Schmidt, P. G.},
     TITLE = {A global attractor and stationary solutions for a
              reaction-diffusion system arising from climate modeling},
   JOURNAL = {Nonlinear Anal.},
  FJOURNAL = {Nonlinear Analysis. Theory, Methods \& Applications. An
              International Multidisciplinary Journal},
    VOLUME = {14},
      YEAR = {1990},
    NUMBER = {11},
     PAGES = {915--926},
      ISSN = {0362-546X,1873-5215},
   MRCLASS = {35K57 (58F12 58F39 86A10)},
  MRNUMBER = {1058413},
MRREVIEWER = {Pierre\ A.\ Vuillermot},
       DOI = {10.1016/0362-546X(90)90109-T},
       URL = {https://doi.org/10.1016/0362-546X(90)90109-T},
}

@article {Hetzer-Schmidt1991,
    AUTHOR = {Hetzer, G. and Schmidt, P. G.},
     TITLE = {Global existence and asymptotic behavior for a quasilinear
              reaction-diffusion system from climate modeling},
   JOURNAL = {J. Math. Anal. Appl.},
  FJOURNAL = {Journal of Mathematical Analysis and Applications},
    VOLUME = {160},
      YEAR = {1991},
    NUMBER = {1},
     PAGES = {250--262},
      ISSN = {0022-247X,1096-0813},
   MRCLASS = {35K57 (58F39 86A10)},
  MRNUMBER = {1124090},
MRREVIEWER = {J\'anos\ T\'oth},
       DOI = {10.1016/0022-247X(91)90303-H},
       URL = {https://doi.org/10.1016/0022-247X(91)90303-H},
}

@incollection{Jentsch0,
  title={Cloud-ice-vapor feedbacks in a global climate model},
  author={Jentsch, V.},
  booktitle={Irreversible Phenomena and Dynamical Systems Analysis in Geosciences},
  pages={417--437},
  year={1987},
  publisher={Springer}
}

@article{Jentsch1,
  title={An energy balance climate model with hydrological cycle: 1. Model description and sensitivity to internal parameters},
  author={Jentsch, V.},
  journal={Journal of Geophysical Research: Atmospheres},
  volume={96},
  number={D9},
  pages={17169--17179},
  year={1991},
  publisher={Wiley Online Library}
}

@article{Jentsch2,
  title={An energy balance climate model with hydrological cycle: 2. Stability and sensitivity to external forcing},
  author={Jentsch, V.},
  journal={Journal of Geophysical Research: Atmospheres},
  volume={96},
  number={D9},
  pages={17181--17193},
  year={1991},
  publisher={Wiley Online Library}
}

@article{Gill1982,
  title={Atmosphere-ocean dynamics},
  author={Gill, A. E.},
  journal={International Geophysics Series},
  volume={30},
  pages={662 p.},
  year={1982},
  publisher={Academic Press}
}

@article{Paltridge1974,
  title={Global cloud cover and earth surface temperature},
  author={Paltridge, G. W.},
  journal={Journal of atmospheric Sciences},
  volume={31},
  number={6},
  pages={1571--1576},
  year={1974}
}

@article{BCCKU,
  title={Stochastic Two-Layer Energy Balance Model: well-posedness and exponential ergodicity},
  author={Broecker, J. and Cannarsa, P. and Carigi, G. and Kuna, T. and Urbani, C.},
  journal={preprint},
  year={2026}
}

@book {Evans,
    AUTHOR = {Evans, L. C.},
     TITLE = {Partial differential equations},
    SERIES = {Graduate Studies in Mathematics},
    VOLUME = {19},
   EDITION = {Second},
 PUBLISHER = {American Mathematical Society, Providence, RI},
      YEAR = {2010},
     PAGES = {xxii+749},
      ISBN = {978-0-8218-4974-3},
   MRCLASS = {35-01},
  MRNUMBER = {2597943},
MRREVIEWER = {Diego\ M.\ Maldonado},
       DOI = {10.1090/gsm/019},
       URL = {https://doi.org/10.1090/gsm/019},
}

@book {Krylov,
    AUTHOR = {Krylov, N. V.},
     TITLE = {Lectures on elliptic and parabolic equations in {H}\"older
              spaces},
    SERIES = {Graduate Studies in Mathematics},
    VOLUME = {12},
 PUBLISHER = {American Mathematical Society, Providence, RI},
      YEAR = {1996},
     PAGES = {xii+164},
      ISBN = {0-8218-0569-X},
   MRCLASS = {35-01 (35Jxx 35Kxx)},
  MRNUMBER = {1406091},
MRREVIEWER = {Manfred\ W.\ Kracht},
       DOI = {10.1090/gsm/012},
       URL = {https://doi.org/10.1090/gsm/012},
}

@book {Ladyz,
    AUTHOR = {Ladyzhenskaya, O. A. and Ural'tseva, N. N.},
     TITLE = {Linear and quasilinear elliptic equations},
      NOTE = {Translated from the Russian by Scripta Technica, Inc,
              Translation editor: Leon Ehrenpreis},
 PUBLISHER = {Academic Press, New York-London},
      YEAR = {1968},
     PAGES = {xviii+495},
   MRCLASS = {35.47},
  MRNUMBER = {244627},
}

@book {Lieberman,
    AUTHOR = {Lieberman, G. M.},
     TITLE = {Second order parabolic differential equations},
 PUBLISHER = {World Scientific Publishing Co., Inc., River Edge, NJ},
      YEAR = {1996},
     PAGES = {xii+439},
      ISBN = {981-02-2883-X},
   MRCLASS = {35-02 (35Bxx 35Dxx 35Kxx)},
  MRNUMBER = {1465184},
MRREVIEWER = {Siegfried\ Carl},
       DOI = {10.1142/3302},
       URL = {https://doi.org/10.1142/3302},
}

@book {henry,
    AUTHOR = {Henry, Daniel},
     TITLE = {Geometric theory of semilinear parabolic equations},
    SERIES = {Lecture Notes in Mathematics},
    VOLUME = {840},
 PUBLISHER = {Springer-Verlag, Berlin-New York},
      YEAR = {1981},
     PAGES = {iv+348},
      ISBN = {3-540-10557-3},
   MRCLASS = {35K55 (34G20 58D25)},
  MRNUMBER = {610244},
MRREVIEWER = {J.\ A.\ Goldstein},
}
\end{document}